\documentclass[11pt,oneside]{amsart}
\usepackage{amssymb}
\usepackage{color}
\usepackage{amsmath}
\usepackage{latexsym}
\usepackage{verbatim}
\usepackage[right=1in,top=1in,bottom=1in]{geometry}
\usepackage{graphicx}
\usepackage{tikz}
\usepackage{amsbsy}
\usepackage{bm}
\usepackage{hyperref}

\date{}

\newcommand{\hot}[1]{\textcolor{red}{#1}}
\newcommand{\cold}[1]{\textcolor{blue}{#1}}
\definecolor{orange}{rgb}{1,.5,0}

\newcommand{\Z}{{\mathbb Z}}
\newcommand{\R}{{\mathbb R}}
\newcommand{\Q}{{\mathbb Q}}
\newcommand{\C}{{\mathbb C}}

\newcommand{\T}{{\mathbb T}}

\newcommand{\bbI}{{\mathbb I}}

\newcommand{\CE}{{\mathcal E}}

\newcommand{\CH}{{\mathcal H}}

\newcommand{\CJ}{{\mathcal J}}

\newcommand{\CV}{{\mathcal V}}

\newcommand{\sq}{{\mathrm{sq}}}
\newcommand{\sqn}{{\mathrm{sqn}}}
\newcommand{\hex}{{\mathrm{hex}}}
\newcommand{\tri}{{\mathrm{tri}}}

\newcommand{\bthe}{\bm{\theta}}
\newcommand{\tbthe}{\widetilde{\bm{\theta}}}
\newcommand{\tthe}{\widetilde{\theta}}
\newcommand{\bbe}{\bm{\beta}}

\newcommand{\set}[1]{\left\{#1\right\}}

\newtheorem{theorem}{Theorem}[section]

\theoremstyle{definition}
\newtheorem*{remark*}{Remark}
\theoremstyle{definition}
\newtheorem{remark}[theorem]{Remark}
\theoremstyle{definition}

\theoremstyle{definition}
\newtheorem*{ex*}{Example}

\newtheorem{lemma}[theorem]{Lemma}

\numberwithin{theorem}{section}
\numberwithin{equation}{section}

\allowdisplaybreaks

\begin{document}
\title[Bethe--Sommerfeld Conjecture for Triangular, Square and Hexagonal Lattices]{Discrete Bethe--Sommerfeld Conjecture for \\ Triangular, Square, and Hexagonal Lattices}

\author[J.\ Fillman]{Jake Fillman}
\author[R.\ Han]{Rui Han}

\begin{abstract} We study discrete Schr\"odinger operators on the graphs corresponding to the triangular lattice, the hexagonal lattice, and the square lattice with next-nearest neighbor interactions. 
For each of these lattice geometries, we analyze the behavior of small periodic potentials. In particular, we provide sharp bounds on the number of gaps that may perturbatively open, we describe sharp arithmetic criteria on the periods that ensure that no gaps open, and we characterize those energies at which gaps may open in the perturbative regime. In all three cases, we provide examples that open the maximal number of gaps and estimate the scaling behavior of the gap lengths as the coupling constant goes to zero.
%We show that for small potentials, the spectrum of a Schr\"odinger operator on the hexagonal lattice consists of no more than four intervals, that the spectrum consists of no more than two intervals if at least one period is odd, and that gaps may open only at energies 0, 1, and -1. 
%For the triangular lattice, we show that the spectrum may consist of no more than two intervals at small coupling, that the spectrum is a single interval if at least one period is odd, and that a gap may only open at energy -2.
%We also show in the case of the square lattice with both nearest and next-nearest neighbor interaction, the spectrum may consist of no more than two intervals at small coupling, that the spectrum is a single interval if at least one period is not divisible by three, and that a gap may only open at enery -1. 
%In all three cases, we exhibit potentials that perturbatively open the maximal number of gaps to show that these results are sharp and we estimate the length of the gap.
 \end{abstract}
\maketitle

\setcounter{tocdepth}{1}
\tableofcontents

\section{Introduction}

The Bethe--Sommerfeld conjecture is the following statement: for any $d \geq 2$ and any periodic function $V:\R^d \to \R$, the spectrum of the Schr\"odinger operator 
\[
L_V := -\nabla^2 + V\]
has only finitely many gaps. This was studied by many people with important advances in \cite{HelMoh98,Karp97,PopSkr81,Skr79,Skr84,Skr85,Vel88}, and culminating in the paper of Parnovskii \cite{Parn2008AHP}. One way to think about the Bethe--Sommerfeld conjecture is that any energy $E$ that is very large relative to the potential $V$ lies in the spectrum of $L_V$. Since discrete Schr\"odinger operators are bounded, the high-energy region is absent, so the appropriate discrete version of the Bethe--Sommerfeld conjecture lies in the region of small $V$. Discrete versions of the conjecture were proved on square lattices by Embree--Fillman in dimension $d=2$ \cite{EmbFil2017} and by Han--Jitomirskaya in arbitrary dimensions $d \geq 2$ \cite{HanJit2017}. In those works, the spectrum of a discrete periodic Schr\"odinger operator on the square lattice $\ell^2(\Z^d)$ with a small potential was shown to consist of at most two intervals. Moreover, they showed that as soon as at least one period of the potential is odd, then the spectrum is an interval, and, in the event that a gap opens perturbatively, it must happen at the exceptional energy $E=0$.

Many interesting physical models occur with different underlying lattice geometries beyond the standard square lattice. One of the most prominent such models is supplied by graphene, a two-dimensional material that consists of carbon atoms at the vertices of a hexagonal lattice. The fascinating properties of graphene have led to a substantial amount of attention in mathematics and physics, see e.g.\ \cite{BZ2018,BHJ2018,CGPNG,DelMon2010,FW12,HKR2016,KP07,N11} and references therein. 
In view of this, we are motivated to study the Bethe--Sommerfeld conjecture for the hexagonal lattice and for the corresponding dual lattice (the triangular lattce).

In addition to the hexagonal and triangular lattices, we also study the square lattice with next-nearest neighbor interactions, which is motivated by the extended Harper model (EHM). The EHM was proposed by Thouless \cite{Thouless83} and has also led to a lot of study in mathematics and physics \cite{AJM17,H17,H18,HJ17,Thouless94,JM15}; it corresponds to an electron in a square lattice that interacts not only with its nearest neighbors but also its next-nearest neighbors.
In the following, we will refer to square lattice with next-nearest neighbor interactions as the {\it EHM lattice}, in order to distinguish it from the standard square lattice.

{Let us mention in particular the closely related work \cite{HKR2016}.} In \cite{HKR2016}, Helffer, Kerdelhu\'e and Royo-Letelier developed a Chambers analysis for magnetic Laplacians on the hexagonal lattice (and its dual lattice: triangular lattice) with rational flux.
They showed that for a non-trivial rational flux $p/q\notin\Z$, the magnetic Laplacians on hexagonal and triangular lattices have non-overlapping (possibly touching) bands.
This recovers a similar feature of the square lattice \cite{BelSim82}.
However, unlike the square lattice that has no touching bands except at the center for $q$ even \cite{VMou89}, 
they were able to give an explicit example of non-trivial touching bands for hexagonal and triangular lattices. 
Indeed they showed that the triangular Laplacian has touching bands at energy $E=-\sqrt{3}$ for $p/q=1/6$, and the hexagonal Laplacian has touching bands at energies $E=\pm \sqrt{3}$ and $0$ for $p/q=1/2$.
Therefore, the underlying geometry is greatly responsible for the formation of touching bands.
But it {has remained} unclear that whether there will be other touching bands for different fluxes (and if any, what are the locations).
In our work we are able to give a sharp criterion of the formation of touching bands for the free Laplacians on these lattices and the EHM lattice, see Theorems \ref{t:bsc:tri}, \ref{t:bsc:hex} and \ref{t:bsc:nnn}.

Motivated by these models, we prove the Bethe--Sommerfeld conjecture for the triangular, hexagonal, and EHM lattices. 
Similar to the square lattice case, we show that small perturbations of the free Laplacian may only open gaps at certain {\it exceptional energies}.
Our proof uses the perturb-and-count technique developed in \cite{HanJit2017}.
The overall  strategy is to argue by contradiction. 
Namely, we assume two adjacent spectral bands of the free Laplacian have a trivial overlap containing a single energy $E$.
Then, we carefully choose a Floquet parameter and perturb all the Floquet eigenvalues along two different directions. 
It is then argued that different directions lead to different counting of eigenvalues that move above/below $E$, hence a contradiction.
At the exceptional energies, we are able to develop a {\it sharp} criterion, in terms of the periods, of whether the gaps could possibly open {under an infinitesimal perturbation}.
We also construct potentials that do open (the theoretically existing) gaps at these exceptional energies.

Although the general strategy follows that of \cite{HanJit2017}, {there are several challenges to overcome in the present work.}:
\begin{itemize}
\item %\cold{In order to apply the perturb-and-count ideas of \cite{HanJit2017}, one typically needs to carefully construct a Floquet parameter and two directions in which to perturb. The first direction is ``generic'' and is generally trivial to construct. For the second though, one must carefully construct both the Floquet parameter and the perturbation direction in such a way that one can control the sign of the quadratic form associated with Hessian corresponding to an analytic parameterization of the eigenvalues of the Floquet matrix simultaneously for all members of said analytic family that pass through a given energy. The exact nature  what one must do here.} 
The Floquet parameters and perturbation directions that we choose in the perturb-and-count technique are strongly model-dependent in a subtle fashion. For example, at non-exceptional energies, we locate Floquet parameters and a perturbation direction in a way such that the Floquet eigenvalues with vanishing linear terms have quadratic terms of the same sign along this direction. 
At the exceptional energy of the triangular lattice, we choose two directions such that the eigenvalues with vanishing gradients have quadratic terms of different signs along the two directions; for a more detailed discussion, see Remark~\ref{rem:tri}.
This is similar to what was done in \cite{HanJit2017} for the square lattice case.
However, for the EHM lattice, any direction will lead to the same number of positive and negative quadratic terms; see Remark \ref{rem:sqn}.
This issue is resolved by a new construction: we find a direction that moves approximately $2/3$ of the degenerate eigenvalues  up while the other $1/3$ move down.
All these constructions depend heavily on the Floquet representation of the eigenvalues, and thus get more difficult as the underlying geometry gets more complicated.
% \cold{For example, at non-exceptional energies, we locate a Floquet parameter and a perturbation direction with the property that the sign of the quadratic form associated with the Hessian is the same for all members of the analytic family whose gradients are orthogonal to the perturbing direction. Compare  \eqref{eq:Jbeta10tri} and \eqref{eq:beta1sqnt2sign}. At the exeptional energy $E=-2$ of the triangular lattice, we instead find a Floquet parameter at which exactly one analytic eigenvalue has a critical point and for which the Hessian has signature $(1,1)$, which suffices to run our arguments. The argument at the exceptional energy $E=-1$ of the square lattice with NNN interactions is different still. Here, the Hessian is difficult to deal with, so we find an argument that entirely bypasses it. Concretely, we only perturb in a single direction; using the assumption that the bands of the Laplacian do not overlap, we carefully construct a direction that simultaneously moves exactly half the eigenvalues to the right and moves almost exactly one-third of the eigenvalues to the right. Making this precise leads to a contradiction.}
\item Applying the perturb-and-count ideas directly to the hexagonal lattice  is quite difficult,
due to the fact that the Floquet eigenvalues do not have simple expressions; compare \eqref{eq:hexTriBandRel}.
However, one can relate Laplacians and Floquet matrices for the triangular and hexagonal lattices in a fairly elegant fashion ({see \cite{HKR2016} and our \eqref{eq:hexSquareTri}}). Thus, we prove the Bethe--Sommerfeld conjecture directly for the triangular lattice and then derive the corresponding statement for the hexagonal lattice via a somewhat soft argument.
\item Because of the more complicated structure of the lattices involved, constructing potentials that open gaps at the exceptional energies is substantially more difficult than in the square lattice. 
In particular, we need to construct (2,2)-periodic potentials that live on eight vertices for the hexagonal lattice, and (3,3)-periodic potential for the EHM lattice.
In this paper we develop an {robust} technique to study these finite volume problems in a sharp way.
Indeed, we can not only prove that a gap exists, but also estimate its size up to a constant factor (see Theorems~\ref{thm:triExampleGapLength}, \ref{thm:hexQ}, and \ref{thm:nnnExGapLength}). 
In the case of the triangular lattice, we are even able to use our technique \emph{exactly} compute the gap, not only estimate its size (Theorem~\ref{thm:triExampleGapLength}).
\end{itemize}

%For the square and triangular lattices, we show that the spectrum consists of no more than two intervals, with the only possible gap openening at energy $E =-2$ for the triangular lattice and energy $E = -1$ for the NNN lattice. Moreover, we show that the spectrum is a single interval for the triangular lattice if at least one period is odd and for the NNN lattice if at least one period is \emph{not divisible by three}. For the hexagonal lattice, we show that the spectrum consists of no more than four intervals in general, and consists of no more than two intervals when at least one period is odd. In the first case, gaps may only open perturbatively at $0$ and $\pm1$, while in the second case gaps may only open at zero. Finally, in all of these situations, we exhibit examples of potentials that open the maximal number of gaps to show that these results are sharp and we investigate the scaling behavior of the gap length in the limit where the coupling constant goes to zero.
\bigskip

\subsection{Main Results}

Let us now describe more precisely the setting in which we work and the results that we prove. By a \textit{graph}, we shall mean a pair $\Gamma = (\CV, \CE)$ where $\CV$ is a nonempty set and $\CE$ is a nonempty subset of $\CV \times \CV$ with the following properties:
\begin{itemize}
\item For no $v \in \CV$ does one have $(v,v) \in \CE$;
\item If $(u,v) \in \CE$, then $(v,u) \in \CE$.
\end{itemize}
If $(u,v) \in \mathcal{E}$, we write $u\sim v$ and we say that $u$ and $v$ are neighbors or neighboring vertices. We think of $\CE$ as the set of \emph{directed edges}; $(u,v)$ represents the edge that originates at $u$ and terminates at $v$.

Given such a graph, we consider $\mathcal{H}_\Gamma = \ell^2(\mathcal{V})$ and the associated \emph{graph Laplacian} $\Delta_\Gamma: \mathcal{H}_\Gamma \to \mathcal{H}_\Gamma$, which acts via
\[
[\Delta_\Gamma \psi]_u
=
\sum_{v \sim u} \psi_v,
\quad
u \in \mathcal{V}, \; \psi \in \mathcal{H}_\Gamma.
\]
Technically, this is the adjacency operator of the graph. Other authors use $\psi_v - \psi_u$ where we have only $\psi_v$. Our convention is slightly more natural for the setting in which we wish to work. Concretely, all of the graphs that we consider in the present work have uniform degree (all vertices in a given graph have the same number of incident edges), and hence leaving off the $-\psi_u$ term merely costs us a multiple of the identity operator, and it simplifies the appearance of a few calculations.

By a \emph{Schr\"odinger operator} on $\Gamma$, we mean an operator of the form $H_Q = H_{\Gamma,Q} = \Delta_\Gamma + Q$, where $Q:\CV \to \R$ is a bounded function that acts on $\CH_\Gamma$ by multiplication:
\[
[Q\psi]_u
=
Q(u) \psi_u,
\quad
u \in \mathcal{V}, \; \psi \in \mathcal{H}_\Gamma.
\]
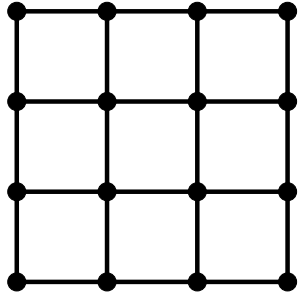
\begin{figure*}[t]
\begin{tikzpicture}[yscale=.6,xscale=.6]
\draw [-,line width = .06cm] (0,0) -- (6,0);
\draw [-,line width = .06cm] (0,0) -- (0,6);
\draw [-,line width=.06cm] (0,2) -- (6,2);
\draw [-,line width = .06cm] (2,0) -- (2,6);
\draw [-,line width=.06cm] (0,4) -- (6,4);
\draw [-,line width = .06cm] (4,0) -- (4,6);
\draw [-,line width=.06cm] (0,6) -- (6,6);
\draw [-,line width = .06cm] (6,0) -- (6,6);
\filldraw[color=black, fill=black](0,0) circle (.2);
\filldraw[color=black, fill=black](2,0) circle (.2);
\filldraw[color=black, fill=black](4,0) circle (.2);
\filldraw[color=black, fill=black](6,0) circle (.2);
\filldraw[color=black, fill=black](0,2) circle (.2);
\filldraw[color=black, fill=black](2,2) circle (.2);
\filldraw[color=black, fill=black](4,2) circle (.2);
\filldraw[color=black, fill=black](6,2) circle (.2);
\filldraw[color=black, fill=black](0,4) circle (.2);
\filldraw[color=black, fill=black](2,4) circle (.2);
\filldraw[color=black, fill=black](4,4) circle (.2);
\filldraw[color=black, fill=black](6,4) circle (.2);
\filldraw[color=black, fill=black](0,6) circle (.2);
\filldraw[color=black, fill=black](2,6) circle (.2);
\filldraw[color=black, fill=black](4,6) circle (.2);
\filldraw[color=black, fill=black](6,6) circle (.2);
\end{tikzpicture}
\caption{The square lattice.}
\end{figure*}
In the present work, we study $\Z^2$-periodic graphs. That is, we consider graphs whose vertices $\mathcal{V}$ comprise a subset of $\R^2$ and for which there exist linearly independent translations $\bm{a}_1, \bm{a}_2 \in \R^2$ which leave $\Gamma$ invariant. That is to say:
\begin{itemize}
\item For any vertex $v \in \CV$, $v + \bm{a}_j \in \CV$ for $j=1,2$;
\item For any edge $(u,v) \in \CE$, $(u + \bm{a}_j,v + \bm{a}_j) \in \CE$  for $j=1,2$.
\end{itemize}
We will then be most interested in studying the case when the potential $Q$ is itself periodic. In general, we will say that $Q:\CV \to \R$ is $\bm{p} = (p_1,p_2)$-periodic for some $p_1,p_2 \in \Z_+$ if and only if 
\[
Q(u+p_1 \bm{a}_1) = Q(u + p_2 \bm{a}_2) = Q(u),\quad
\text{for all } u \in \CV.
\]

The square lattice is the graph with vertices $\CV_\sq = \Z^2$ and where 
\[
\bm{n} \sim \bm{n}'
\iff
\| \bm{n} - \bm{n'} \|
=
1.
\]
Here and throughout the paper, $\|\cdot\|$ denotes the Euclidean norm on $\R^2$. It is easy to see that the associated Laplacian acts on $\ell^2(\Z^2)$ via
\[
[\Delta_\sq \psi]_{n,m}
=
\psi_{n-1,m} + \psi_{n+1,m} + \psi_{n,m-1} + \psi_{n,m+1}.
\]

Part of the motivation for the present work comes from \cite{EmbFil2017, HanJit2017, KrugPreprint}. In \cite{EmbFil2017}, Embree and Fillman showed that if $Q: \Z^2 \to \R$ is $(p_1,p_2)$-periodic and sufficiently small, then $\sigma(\Delta_\sq+Q)$ consists of one or two intervals and that the spectrum consists of exactly one interval whenever at least one of $p_1$ or $p_2$ is odd, which generalized the work of Kr\"uger, who proved a similar result under the stricter condition that the periods were coprime \cite{KrugPreprint}. In \cite{HanJit2017}, Han and Jitomirskaya showed that if $Q:\Z^d \to \R$ is $(p_1,\ldots,p_d)$-periodic and small, then the same results hold true: the spectrum has no more than one gap and has no gaps as long as at least one period is odd.

\subsection{The Triangular Lattice} 
\begin{figure*}[b]
 \begin{minipage}{0.45\textwidth}
 \centering
\begin{tikzpicture}[yscale=.85,xscale=.85]
\draw [-,line width = .06cm] (0,0) -- (6,0);
\draw [-,line width = .06cm] (0,{sqrt(3)}) -- (6,{sqrt(3)});
\draw [-,line width = .06cm] (0,{2*sqrt(3)}) -- (6,{2*sqrt(3)});
\draw [-,line width = .06cm] (0,{3*sqrt(3)}) -- (6,{3*sqrt(3)});
\draw [-,line width=.06cm] (0,{2*sqrt(3)}) -- (1,{3*sqrt(3)});
\draw [-,line width=.06cm] (0,0) -- (3,{3*sqrt(3)});
\draw [-,line width=.06cm] (2,0) -- (5,{3*sqrt(3)});
\draw [-,line width=.06cm] (4,0) -- (6,{2*sqrt(3)});
\draw [-,line width=.06cm] (0,{2*sqrt(3)}) -- (2,0);
\draw [-,line width=.06cm] (1,{3*sqrt(3)}) -- (4,0);
\draw [-,line width=.06cm] (3,{3*sqrt(3)}) -- (6,0);
\draw [-,line width=.06cm] (5,{3*sqrt(3)}) -- (6,{2*sqrt(3)});
\draw [->,line width=.06cm,color=blue] (1,{sqrt(3)}) -- (2,{2*sqrt(3)});
\filldraw[color=black, fill=black](0,0) circle (.2);
\filldraw[color=black, fill=black](2,0) circle (.2);
\filldraw[color=black, fill=black](4,0) circle (.2);
\filldraw[color=black, fill=black](6,0) circle (.2);
\filldraw[color=black, fill=black](1,{sqrt(3)}) circle (.2);
\filldraw[color=black, fill=black](3,{sqrt(3)}) circle (.2);
\filldraw[color=black, fill=black](5,{sqrt(3)}) circle (.2);
\filldraw[color=black, fill=black](0,{2*sqrt(3)}) circle (.2);
\filldraw[color=black, fill=black](2,{2*sqrt(3)}) circle (.2);
\filldraw[color=black, fill=black](4,{2*sqrt(3)}) circle (.2);
\filldraw[color=black, fill=black](6,{2*sqrt(3)}) circle (.2);
\filldraw[color=black, fill=black](1,{3*sqrt(3)}) circle (.2);
\filldraw[color=black, fill=black](3,{3*sqrt(3)}) circle (.2);
\filldraw[color=black, fill=black](5,{3*sqrt(3)}) circle (.2);
\draw [->,line width=.06cm,color=blue] (1,{sqrt(3)}) -- (2,{2*sqrt(3)});
\draw [->,line width=.06cm,color=blue] (1,{sqrt(3)}) -- (3,{sqrt(3)});
\node [above] at (1,{sqrt(3)+.3}) {\cold{$\bm{a}_2$}};
\node [below] at (1.7,{sqrt(3)-.1}) {\cold{$\bm{a}_1$}};
\end{tikzpicture}
\caption{A portion of the triangular lattice}\label{fig:trilat}
\end{minipage}
\hfill
 \begin{minipage}{0.45\textwidth}
 \centering
\begin{tikzpicture}[yscale=.75,xscale=.75]
\draw [-,line width = .06cm] (0,0) -- (6,0);
\draw [-,line width = .06cm] (0,0) -- (0,6);
\draw [-,line width=.06cm] (0,2) -- (6,2);
\draw [-,line width = .06cm] (2,0) -- (2,6);
\draw [-,line width=.06cm] (0,4) -- (6,4);
\draw [-,line width = .06cm] (4,0) -- (4,6);
\draw [-,line width=.06cm] (0,6) -- (6,6);
\draw [-,line width = .06cm] (6,0) -- (6,6);
\draw [-,line width = .06cm] (0,6) -- (6,0);
\draw [-,line width = .06cm] (2,0) -- (0,2);
\draw [-,line width = .06cm] (4,0) -- (0,4);
\draw [-,line width = .06cm] (2,6) -- (6,2);
\draw [-,line width = .06cm] (6,4) -- (4,6);
\filldraw[color=black, fill=black](0,0) circle (.2);
\filldraw[color=black, fill=black](2,0) circle (.2);
\filldraw[color=black, fill=black](4,0) circle (.2);
\filldraw[color=black, fill=black](6,0) circle (.2);
\filldraw[color=black, fill=black](0,2) circle (.2);
\filldraw[color=black, fill=black](2,2) circle (.2);
\filldraw[color=black, fill=black](4,2) circle (.2);
\filldraw[color=black, fill=black](6,2) circle (.2);
\filldraw[color=black, fill=black](0,4) circle (.2);
\filldraw[color=black, fill=black](2,4) circle (.2);
\filldraw[color=black, fill=black](4,4) circle (.2);
\filldraw[color=black, fill=black](6,4) circle (.2);
\filldraw[color=black, fill=black](0,6) circle (.2);
\filldraw[color=black, fill=black](2,6) circle (.2);
\filldraw[color=black, fill=black](4,6) circle (.2);
\filldraw[color=black, fill=black](6,6) circle (.2);
\end{tikzpicture}
\caption{The triangular lattice after shearing.}\label{fig:trishear}
\end{minipage}
\end{figure*}

The first graph that we consider is the \emph{triangular lattice}. The graph has vertices
\[
\CV_\tri
=
\set{n \bm{a}_1 + m \bm{a}_2 : n,m \in \Z},
\]
where the generating vectors are
\[
\bm{a}_1
=
\begin{bmatrix}
1 \\ 0
\end{bmatrix},
\quad
\bm{a}_2
=
\frac{1}{2} \begin{bmatrix}
1 \\ \sqrt{3}
\end{bmatrix}.
\]
One then declares $v \sim w$ for $v,w \in \mathcal{V}$ if $\|v - w\| = 1$. Thus, every $v \in \mathcal{V}$ has 6 neighbors; more specifically, if $v = n\bm{a}_1 + m \bm{a}_2$, then $v$ has neighbors
\[
(n \pm1) \bm{a}_1 + m \bm{a}_2, \quad
n\bm{a}_1 + (m\pm 1) \bm{a}_2,\quad
(n \pm 1) \bm{a}_1 + (m\mp 1) \bm{a}_2.
\]
Consequently, after identifying $n \bm{a}_1 + m \bm{a}_2$ with the point $(n,m) \in \Z^2$, we may view the Laplacian on the triangular lattice as an operator on $\ell^2(\Z^2)$ via
\begin{equation} \label{eq:triLaplacianSqVersion}
[\Delta_{\rm tri} \psi]_{n,m}
=
[\Delta_{\rm sq} \psi]_{n,m}
+
\psi_{n-1,m+1} + \psi_{n+1,m-1}.
\end{equation}
This correspondence amounts to shearing and stretching the the triangular lattice, and essentially maps the triangular lattice to the square lattice with skewed next-nearest-neighbor interactions added. See Figures~\ref{fig:trilat} and \ref{fig:trishear}.

\begin{theorem}[Bethe--Sommerfeld for the triangular lattice] \label{t:bsc:tri}
For all $\bm{p} = (p_1,p_2 )\in \Z_+^2$, there is a constant $c = c_{\bm{p}} > 0$ such that, if $Q:\CV_\tri \to \R$ is $\bm{p}$-periodic and $\|Q\|_\infty \leq c$, the following hold true for $H_Q = \Delta_\tri + Q$:
\begin{enumerate}
\item[{\rm(1)}] $\sigma(H_Q)$ consists of no more than two intervals.
\item[{\rm(2)}] If at least one of $p_1$ or $p_2$ is odd, then $\sigma(H_Q)$ consists of a single interval.
\end{enumerate}
Moreover, the gap in the first setting may only open at the energy $E = -2$.
\end{theorem}

This theorem is sharp vis-\`a-vis the number of intervals in the spectrum and the arithmetic restrictions on the periods. Concretely, we exhibit a $(2,2)$-periodic potential that perturbatively opens a gap at $-2$.

\begin{theorem} \label{t:triExamples}
There exists $Q:\CV_\tri \to \R$ which is $(2,2)$-periodic, such that $\sigma(H_{\lambda Q})$ has exactly two connected components for any sufficiently small $\lambda > 0$.
\end{theorem}

\subsection{The Hexagonal Lattice} The set of vertices of the hexagonal lattice is closely related to that of the triangular lattice. Concretely, define $\bm{b}_\pm$ by
\[
\bm{b}_\pm
=
\frac{1}{2}
\begin{bmatrix}
 3 \\ \pm \sqrt{3}
\end{bmatrix}.
\]
Then, we obtain the hexagonal lattice by deleting the centers of some of the hexagons formed by the triangular lattice; more precisely,
\[
\CV_\hex
=
\set{n \bm{a}_1 + m \bm{a}_2 \in \CV_\tri:n,m \in \Z} \setminus \{- \bm{a}_1 + k \bm{b}_+ + \ell \bm{b}_- : k,\ell \in \Z\}.
\]
Equivalently, it is not hard to check that $\{0, \bm{a}_1\}$ is a fundamental set of vertices and hence every $v \in \CV_\hex$ may be written uniquely as either $n \bm{b}_+ + m \bm{b}_-$ or $\bm{a}_1 + n \bm{b}_+ + m \bm{b}_-$ for integers $n,m$, so we have
\begin{align*}
\CV_\hex
& =
\set{n \bm{b}_+ + m \bm{b}_- : n,m \in \Z} \cup \set{\bm{a}_1 + n \bm{b}_+ + m \bm{b}_- : n,m \in \Z}.
\end{align*}
We define $\CE_\hex$ by declaring $u \sim v$ for $u, v \in \CV_\hex$ if $\|u - v\|_2 = 1$. After some calculations, we see that
\begin{align*}
[\Delta_\hex \psi]_{n \bm{b}_+ + m \bm{b}_-}
& =
\psi_{\bm{a}_1 + n \bm{b}_+ + m \bm{b}_-} + \psi_{\bm{a}_1 + n \bm{b}_+ + (m-1) \bm{b}_-} + \psi_{\bm{a}_1 + (n-1) \bm{b}_+ + m \bm{b}_-} \\
[\Delta_\hex \psi]_{\bm{a}_1 + n \bm{b}_+ + m \bm{b}_-}
& =
\psi_{n \bm{b}_+ + m \bm{b}_-} + \psi_{n \bm{b}_+ + (m+1) \bm{b}_-} + \psi_{(n+1) \bm{b}_+ + m \bm{b}_-}
\end{align*}

\begin{figure*}[t]
\begin{tikzpicture}[yscale=.75,xscale=.75]
\draw [-,line width = .06cm] (0,0) -- (1,{sqrt(3)});
\draw [-,line width = .06cm] (0,{2*sqrt(3)}) -- (1,{sqrt(3)});
\draw [-,line width = .06cm] (0,{2*sqrt(3)}) -- (1,{3*sqrt(3)});
\draw [-,line width = .06cm] (0,{4*sqrt(3)}) -- (1,{3*sqrt(3)});
\draw [-,line width = .06cm,color=red] (1,{sqrt(3)}) -- (3,{sqrt(3)});
\draw [-,line width = .06cm] (1,{3*sqrt(3)}) -- (3,{3*sqrt(3)});
\draw [-,line width = .06cm] (4,{2*sqrt(3)}) -- (6,{2*sqrt(3)});
\draw [-,line width = .06cm] (4,0) -- (3,{sqrt(3)});
\draw [-,line width = .06cm] (4,{2*sqrt(3)}) -- (3,{sqrt(3)});
\draw [-,line width = .06cm] (4,{2*sqrt(3)}) -- (3,{3*sqrt(3)});
\draw [-,line width = .06cm] (4,{4*sqrt(3)}) -- (3,{3*sqrt(3)});
\draw [-,line width = .06cm] (4,0) -- (6,0);
\draw [-,line width = .06cm] (4,{2*sqrt(3)}) -- (6,{2*sqrt(3)});
\draw [-,line width = .06cm] (4,{4*sqrt(3)}) -- (6,{4*sqrt(3)});
\draw [-,line width = .06cm] (6,0) -- (7,{sqrt(3)});
\draw [-,line width = .06cm] (6,{2*sqrt(3)}) -- (7,{sqrt(3)});
\draw [-,line width = .06cm] (6,{2*sqrt(3)}) -- (7,{3*sqrt(3)});
\draw [-,line width = .06cm] (6,{4*sqrt(3)}) -- (7,{3*sqrt(3)});
\filldraw[color=black, fill=black](0,0) circle (.2);
\filldraw[color=black, fill=black](4,0) circle (.2);
\filldraw[color=black, fill=black](6,0) circle (.2);
\filldraw[color=red, fill=red](1,{sqrt(3)}) circle (.2);
\filldraw[color=red, fill=red](3,{sqrt(3)}) circle (.2);
\filldraw[color=black, fill=black](7,{sqrt(3)}) circle (.2);
\filldraw[color=black, fill=black](0,{2*sqrt(3)}) circle (.2);
\filldraw[color=black, fill=black](4,{2*sqrt(3)}) circle (.2);
\filldraw[color=black, fill=black](6,{2*sqrt(3)}) circle (.2);
\filldraw[color=black, fill=black](1,{3*sqrt(3)}) circle (.2);
\filldraw[color=black, fill=black](3,{3*sqrt(3)}) circle (.2);
\filldraw[color=black, fill=black](7,{3*sqrt(3)}) circle (.2);
\filldraw[color=black, fill=black](0,{4*sqrt(3)}) circle (.2);
\filldraw[color=black, fill=black](4,{4*sqrt(3)}) circle (.2);
\filldraw[color=black, fill=black](6,{4*sqrt(3)}) circle (.2);
\draw [->,line width = .06cm,color=blue] (1.2,{sqrt(3)-.1}) -- (3.8,{.1});
\draw [->,line width = .06cm,color=blue] (1.2,{sqrt(3)+.1}) -- (3.8,{2*sqrt(3)-.1});
\node at (2,{1.8*sqrt(3)}) {\cold{$\bm{b}_+$}};
\node at (2,{.2*sqrt(3)}) {\cold{$\bm{b}_-$}};
\end{tikzpicture}
\caption{A portion of the hexagonal lattice. A fundamental domain is highlighted in red.}\label{fig:hexlat}
\end{figure*}
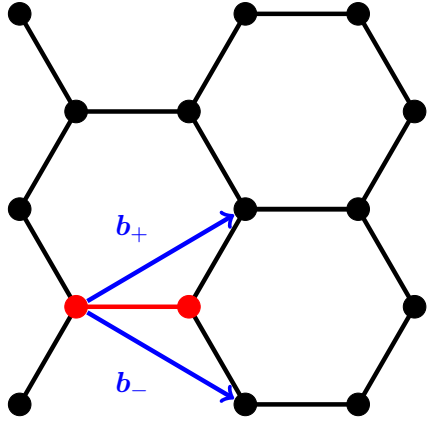
See Figure~\ref{fig:hexlat}. The formula for $\Delta_\hex$ can be made more compact if we view the associated Hilbert space as 
\[
\ell^2(\Z^2,\C^2)
=
\set{\Psi:\Z^2 \to \C^2 : \sum_{n,m} \|\Psi_{n,m}\|^2 < \infty},
\]
where the standard basis of $\C^2$ corresponds to the left and right vertices of the fundamental domain, respectively. More precisely, given $\psi \in \ell^2(\CV_\hex)$, define $\Psi \in \ell^2(\Z^2,\C^2)$ by 
\[
\Psi_{n,m}
=
\begin{bmatrix}
\psi_{n \bm{b}_+ + m \bm{b}_-} \\ \psi_{\bm{a}_1 + n \bm{b}_+ + m \bm{b}_-}
\end{bmatrix}.
\]
Identifying $\ell^2(\CV_\hex)$  and $\ell^2(\Z^2,\C^2)$ in this fashion, the Laplacian for the hexagonal lattice is given by
\begin{align*}
[\Delta_\hex \Psi]_{n,m}
& =
U(\Psi_{n,m-1}+ \Psi_{n-1,m}) + L(\Psi_{n,m+1}+  \Psi_{n+1,m}) + J \Psi_{n,m},
\end{align*}
where
\[
U
=
\begin{bmatrix}
0 & 1 \\ 0 & 0
\end{bmatrix},
\quad
L 
=
U^\top
=
\begin{bmatrix}
0 & 0 \\ 1 & 0
\end{bmatrix},
\quad
J
=
U+L
=
\begin{bmatrix}
0 & 1 \\ 1 & 0
\end{bmatrix}.
\]
{Equivalently, if we denote by $S_1, S_2 : \ell^2(\Z^2) \to \ell^2(\Z^2)$ the shift operators
\[
[S_1 \psi]_{n,m}
=
\psi_{n+1,m},
\quad
[S_2\psi]_{n,m}
=
\psi_{n,m+1},
\]
we have
\[
\Delta_\hex \Psi
=
\begin{bmatrix}
(S_1^* + S_2^* + \bbI)\psi^- \\  (S_1 + S_2 + \bbI)\psi^+
\end{bmatrix}
\quad
\text{for any} \quad 
\Psi 
=
\begin{bmatrix} \psi^+ \\ \psi^- \end{bmatrix}
\in \ell^2(\Z^2,\C^2).
\]
Abbreviating somewhat, we write:
\begin{equation} \label{eq:hexDecomp}
\Delta_\hex
=\begin{bmatrix}
0 & S_1^* + S_2^* + \bbI \\
S_1 + S_2 + \bbI & 0
\end{bmatrix}.
\end{equation}}

\begin{theorem}[Bethe--Sommerfeld for the hexagonal lattice] \label{t:bsc:hex}
For all $\bm{p} = (p_1,p_2 )\in \Z_+^2$, there is a constant $c = c_{\bm{p}} > 0$ such that, if $Q:\CV_\hex \to \R$ is $\bm{p}$-periodic and $\|Q\|_\infty \leq c$, the following statements hold true for $H_Q = \Delta_\hex + Q$:
\begin{enumerate}
\item[{\rm(1)}]  $\sigma(H_Q)$ consists of no more than four intervals.
\item[{\rm(2)}] If at least one of $p_1$ or $p_2$ is odd, then $\sigma(H_Q)$ consists of no more than two intervals.
\end{enumerate}
Moreover, gaps may only open at $0$ and $\pm1$ in the first case, and only at zero in the second case.
\end{theorem}

Moreover, this theorem is sharp in the following sense: there exists a $(1,1)$-periodic potential $Q_1$ which infinitesimally opens a gap at zero, and there is a $(2,2)$-periodic potential $Q_2$ which infinitesimally opens gaps at $-1$, $0$, and $1$ in the following sense:

\begin{theorem} \label{t:hexExamples}
\begin{enumerate}
\item[{\rm(1)}] There exists $Q_1: \CV_\hex \to \R^2$ which is $(1,1)$-periodic such that $\sigma(H_{\lambda Q_1})$ has exactly two connected components for all $\lambda>0$.
\item[{\rm(2)}] There exists $Q_2: \CV_\hex \to \R^2$  which is $(2,2)$ periodic such that $\sigma(H_{\lambda Q_2})$ has exactly four connected components for any sufficiently small $\lambda > 0$.
\end{enumerate}
\end{theorem}

Let us remark that Theorem~\ref{t:hexExamples}.(1) is well-known; we merely list it for completeness. The example in Theorem~\ref{t:hexExamples}.(2) is novel.

\subsection{The EHM Lattice}

The EHM lattice also has vertex set $\CV_\sqn = \CV_\sq = \Z^2$. However, now, one connects $\bm{n}$ and $\bm{n}'$ if and only if they are nearest neighbors or next-nearest-neighbors in the square lattice. Equivalently, one declares
\[
\bm{n} \sim \bm{n}'
\iff
\|\bm{n} - \bm{n}' \|_\infty
=
1.
\] 
The associated Laplacian acts on $\ell^2(\Z^2)$ via
\[
[\Delta_\sqn \psi]_{n,m}
=
[\Delta_\sq]_{n,m}+\psi_{n-1,m-1}+\psi_{n-1,m+1}+\psi_{n+1,m-1}+\psi_{n+1,m+1}.
\]
See Figure~\ref{fig:ehmlat}.
\begin{figure*}[t]

\begin{tikzpicture}[yscale=.6,xscale=.6]
\draw [-,line width = .06cm] (0,0) -- (6,0);
\draw [-,line width = .06cm] (0,0) -- (0,6);
\draw [-,line width=.06cm] (0,2) -- (6,2);
\draw [-,line width = .06cm] (2,0) -- (2,6);
\draw [-,line width=.06cm] (0,4) -- (6,4);
\draw [-,line width = .06cm] (4,0) -- (4,6);
\draw [-,line width=.06cm] (0,6) -- (6,6);
\draw [-,line width = .06cm] (6,0) -- (6,6);
\draw [-,line width = .06cm] (0,0) -- (6,6);
\draw [-,line width = .06cm] (2,0) -- (6,4);
\draw [-,line width = .06cm] (4,0) -- (6,2);
\draw [-,line width = .06cm] (0,2) -- (4,6);
\draw [-,line width = .06cm] (0,4) -- (2,6);
\draw [-,line width = .06cm] (6,0) -- (6,6);
\draw [-,line width = .06cm] (0,6) -- (6,0);
\draw [-,line width = .06cm] (2,0) -- (0,2);
\draw [-,line width = .06cm] (4,0) -- (0,4);
\draw [-,line width = .06cm] (2,6) -- (6,2);
\draw [-,line width = .06cm] (6,4) -- (4,6);
\filldraw[color=black, fill=black](0,0) circle (.2);
\filldraw[color=black, fill=black](2,0) circle (.2);
\filldraw[color=black, fill=black](4,0) circle (.2);
\filldraw[color=black, fill=black](6,0) circle (.2);
\filldraw[color=black, fill=black](0,2) circle (.2);
\filldraw[color=black, fill=black](2,2) circle (.2);
\filldraw[color=black, fill=black](4,2) circle (.2);
\filldraw[color=black, fill=black](6,2) circle (.2);
\filldraw[color=black, fill=black](0,4) circle (.2);
\filldraw[color=black, fill=black](2,4) circle (.2);
\filldraw[color=black, fill=black](4,4) circle (.2);
\filldraw[color=black, fill=black](6,4) circle (.2);
\filldraw[color=black, fill=black](0,6) circle (.2);
\filldraw[color=black, fill=black](2,6) circle (.2);
\filldraw[color=black, fill=black](4,6) circle (.2);
\filldraw[color=black, fill=black](6,6) circle (.2);
\end{tikzpicture}
\caption{A portion of the EHM lattice.}\label{fig:ehmlat}
\end{figure*}
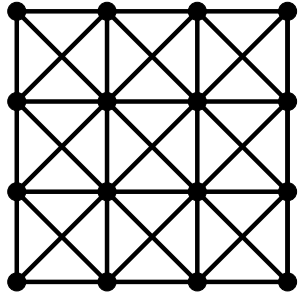

\begin{theorem}[Bethe--Sommerfeld for the EHM lattice] \label{t:bsc:nnn}
For all $\bm{p} = (p_1,p_2 )\in \Z_+^2$, there is a constant $c = c_{\bm{p}} > 0$ such that, if $Q:\CV_\sqn \to \R$ is $\bm{p}$-periodic and $\|Q\|_\infty \leq c$, the following hold true for $H_Q = \Delta_\sqn + Q$:
\begin{enumerate}
\item[{\rm(1)}] $\sigma(H_Q)$ consists of no more than two intervals.
\item[{\rm(2)}] If at least one of $p_1$ or $p_2$ is not divisible by three, then $\sigma(H_Q)$ consists of a single interval.
\end{enumerate}
Moreover, the gap in the first setting may only open at the energy $E = -1$.
\end{theorem}
This theorem is also sharp:
\begin{theorem} \label{t:nnnExamples}
There exists $Q:\Z^2 \to \R$ which is $(3,3)$-periodic such that $\sigma(H_{\lambda Q})$ has exactly two connected components for any sufficiently small $\lambda>0$.
\end{theorem}

\bigskip

The remainder of the paper is organized as follows. Section~\ref{sec:floquet} recalls Floquet theory for $\Z^2$-periodic graphs. We work with the triangular lattice in Section~\ref{sec:tri}, proving Theorems~\ref{t:bsc:tri} and \ref{t:triExamples}. We then work with the hexagonal lattice in Section~\ref{sec:hex}, proving Theorems~\ref{t:bsc:hex} and \ref{t:hexExamples}. Finally, we conclude with the EHM lattice in Section~\ref{sec:nnn}, proving Theorems~\ref{t:bsc:nnn} and \ref{t:nnnExamples}.

\section{Floquet Theory for Periodic Schr\"odinger Operators on Periodic Graphs}
\label{sec:floquet}

Let $\Gamma= (\CV,\CE)$ be a $\Z^2$-periodic graph with translation symmetries $\bm{a}_1, \bm{a}_2 \in \R^2$, and suppose $Q:\CV \to \R$ is $\bm{p} = (p_1,p_2)$-periodic, that is,
\[
Q(u + p_j \bm{a}_j)
=
Q(u),\quad
u \in \CV,\; j = 1,2.
\]
We will briefly describe Floquet theory for $H_Q = \Delta_\Gamma + Q$, following \cite{KorSab2014}. The main purpose of this section is to establish notation, so we do not give any proofs. One may write $H_Q$ as a constant-fiber direct integral over the fundamental domain. Concretely, let
\[
\CV_{\rm f}
=
\CV \cap \set{s \bm{a}_1 + t \bm{a}_2 : 0 \le s < p_1, \; 0 \le t < p_2}.
\]
By periodicity, $|\CV_{\rm f}| = P := p_0 p_1 p_2$, where
\[
p_0
=
|\CV \cap \set{s\bm{a}_1 + t \bm{a}_2 : 0 \le s,t < 1}|.
\]
Here, and throughout the paper, we use $|S|$ to denote the cardinality of the set $S$. For each edge $(u,v) \in \CE$ there exist unique vertices $u_{\rm f}, v_{\rm f} \in \CV_{\rm f}$ and unique integers $n,m,n',m' \in \Z$ with
\[
u = u_{\rm f} + np_1 \bm{a}_1 + mp_2 \bm{a}_2,
\quad
v = v_{\rm f} + n'p_1\bm{a}_1 + m'p_2 \bm{a}_2,
\]
We then define the \emph{index} of $(u,v)$ by $\tau(u,v) = (n'-n,m'-m)$. Finally, for $u,v \in \CV_{\rm f}$, we define $B(u,v)$ to be the set of all translates of $v$ that connect to $u$ via an edge of $\Gamma$:
\[
B(u,v)
=
\set{w \in \CV : w \sim u \text{ and } w = v + np_1 \bm{a}_1 + mp_2 \bm{a}_2 \text{ for some }n,m \in \Z}.
\]

Then, for each $\bthe =(\theta_1,\theta_2) \in \R^2$, the corresponding Floquet matrix is a self-adjoint operator on $\CH_{\rm f}:= \ell^2(\CV_{\rm f}) = \C^{\CV_{\rm f}}$ defined by
\begin{equation} \label{eq:floqMatDef}
\langle \delta_u, H_Q(\bthe) \delta_v \rangle
=
\sum_{w \in B(u,v)} \exp\Big(i \big\langle \tau(u,w), \bthe \big\rangle\Big).
\end{equation}
In the event that the sum in \eqref{eq:floqMatDef} is empty, $\langle \delta_u, H_Q(\bthe) \delta_v \rangle = 0$. Clearly, if $\theta_j' - \theta_j \in 2\pi\Z$ for $j=1,2$, then $H_Q(\bthe) = H_Q(\bthe')$, so $H_Q(\bthe)$ descends to a well-defined function of $\bthe \in \T^2 := \R^2 / (2\pi\Z)^2 \cong [0,2\pi)^2$. We will freely use $\bthe \in \R^2$ or $\bthe \in \T^2$ depending on which is more convenient in a given setting.

Informally, \eqref{eq:floqMatDef} represents the restriction of $H_Q$ to the discrete torus 
\[
(\Z\bm{a}_1 \oplus \Z\bm{a}_2) / (p_1 \Z \bm{a}_1 \oplus p_2 \Z \bm{a}_2)
\cong
\Z_{p_1} \oplus \Z_{p_2}.
\]
with the following boundary conditions: wrapping once around the torus in the positive $\bm{a}_1$ direction accrues a phase $e^{i\theta_1}$ and wrapping around once in the positive $\bm{a}_2$ direction accrues a phase $e^{i\theta_2}$. More precisely, we may view $H_Q(\bthe)$ in the following manner. The operator $H_Q$ acts on the space $\C^\CV$ of arbitrary (not necessarily square-summable) functions $\CV \to \C$. When $Q$ is $(p_1,p_2)$-periodic, then for each $\bthe \in \T^2$, $H_Q$ preserves the subspace
\[
\mathcal{H}(\bthe)
=
\set{\psi \in \C^\CV : \psi(u+p_j \bm{a}_j) = e^{i\theta_j} \psi(u)}.
\]
Then, $H_Q(\bthe)$ is equivalent to the restriction of $H_Q$ to $\CH(\bthe)$.

For each $\bthe$, order the eigenvalues of $H_Q(\bthe)$ as
\[
E_1(\bthe)
\leq
\cdots
\leq
E_P(\bthe)
\]
with each eigenvalue listed according to its multiplicity. Then, for $1 \le j \le P$, the $j$th spectral \emph{band} of $H_Q$ is defined by
\[
F_j
=
F_j(Q)
:=
\mathrm{ran}(E_j)
=
\set{E_j(\bthe) : \bthe \in \T^2}
=
\set{E_j(\bthe) : \bthe \in \R^2}.
\]

\begin{theorem} \label{t:floquet}
With notation as above,
\[
\sigma(H_Q)
=
\bigcup_{\bthe \in \T^2} H_Q(\bthe)
=
\bigcup_{j=1}^P F_j.
\]
\end{theorem}

We will use Theorem~\ref{t:floquet} in the following way. Making the dependence on the potential $Q$ explicit, one may write
\[
F_j = F_j(Q)
=
[E_j^-(Q),E_j^+(Q)].
\]
The key fact is the following: by standard perturbation theory for self-adjoint operators, $E_j^\pm(Q)$ are 1-Lipschitz functions of $Q$. Here, one views $Q$ as an element of $\R^P$ and the perturbation is with respect to the uniform metric thereupon. In particular, if an energy $E$ satisfies $E \in \mathrm{int}(F_j(Q))$, then $(E-\delta,E+\delta) \subseteq F_j(Q)$ for some positive $\delta$, and it follows that $E \in F_j(Q') \subseteq \sigma(H_{Q'})$ for any $(p_1,p_2)$-periodic $Q'$ with $\|Q-Q'\|_\infty < \delta$. Note that here it is very important that one views the periods as fixed: one may only perturb within $\R^P$ for a fixed $P$. Thus, our analysis revolves around determining for a given energy $E$, whether $E$ belongs to the interior of some band of the Laplacian, where the Laplacian is viewed as a degenerate $(p_1,p_2)$-periodic operator.

\section{Triangular Laplacian}\label{sec:tri}

We view the triangular Laplacian as acting on the square lattice $\ell^2(\Z^2)$, but with extra connections as in \eqref{eq:triLaplacianSqVersion}:
\begin{align*}
[\Delta_{\rm tri} u]_{n,m}
& =
u_{n-1,m} + u_{n+1,m} + u_{n,m-1} + u_{n,m+1} + u_{n-1,m+1} + u_{n+1,m-1} \\
& =
[\Delta_{\rm sq}u]_{n,m} + u_{n-1,m+1} + u_{n+1,m-1}.
\end{align*}
Now, given $p_1,p_2 \in \Z_+$, we view $\Delta_\tri$ as a $\bm{p}$-periodic operator and perform the Floquet decomposition. Define $P:=p_1p_2$ as in Section~\ref{sec:floquet}, and put
\[
\Lambda :=
\Z^2 \cap \Big( [0,p_1) \times [0,p_2) \Big).
\]
For $\bthe = (\theta_1,\theta_2) \in \R^2$, it is straightforward to check that
\[
\sigma(H(\bthe))
=
\set{e_{\bm{\ell}}^{\Lambda}(\bthe) : \bm{\ell}\in \Lambda },
\]
where $\bm{\ell}=(\ell_1,\ell_2)$ and
\[
e^{\Lambda}_{\bm{\ell}}(\bthe)
=
2\cos\left( \frac{\theta_1+2\pi \ell_1}{p_1}\right) 
+ 2\cos\left(\frac{\theta_2+2\pi \ell_2}{p_2}\right) 
+ 2\cos\left(\frac{\theta_1 + 2\pi \ell_1}{p_1}  - \frac{\theta_2 + 2\pi \ell_2}{p_2}\right).
\]
Let us point out that one needs to be somewhat careful at this point; namely, $e^\Lambda_{\bm{\ell}}(\bthe)$ is not a well-defined function of $\bthe \in \T^2$. However, the error incurred in using a different coset representative of $\bthe \in \T^2$ is simply a change in the index $\bm{\ell}$, and one can check that the \emph{family} $\set{e^\Lambda_{\bm{\ell}}(\bthe) : \bm{\ell} \in \Lambda}$ is a well-defined function on $\T^2$ (as well it should, since the \emph{operator} $H(\bthe)$ is itself a well-defined function of $\bthe \in \T^2$). In any case, the ambiguity disappears when one considers the covering space $\R^2$, which we do for most of the paper. One could also use the minimal covering space $\R^2 / (p_1\Z \oplus p_2 \Z)$ on which the $e_{\bm{\ell}}^\Lambda$ are well-defined, but this does not accrue any benefits vis-\`a-vis the present work, so we simply use $\R^2$.

As in Section~\ref{sec:floquet}, we label these eigenvalues in increasing order according to multiplicity by
\[
E_1^{\Lambda}(\bthe)
\le 
E_2^{\Lambda}(\bthe)\le \cdots E_P^{\Lambda}(\bthe)
\]
and denote the $P$ spectral bands by
\[
F_k^{\Lambda}
=
\set{E_k^{\Lambda}(\bthe) : \bthe \in \R^2},
\quad
1 \le k \le P.
\]
Straightforward computations shows that $\sigma(\Delta_{\rm tri})=[-3,6]$, and thus
\[
\bigcup_{k=1}^P F_k^{\Lambda}
=
[-3,6].
\]
Henceforth, we view $p_1$ and $p_2$ as fixed and so we drop $\Lambda$ from the superscripts.
Our main theorem of this section is the following.
\begin{theorem}\label{thm:trimain}
Let $p_1,p_2 \in \Z_+$ be given.
\begin{enumerate}
\item[{\rm 1.}]
Each $E\in (-3, 6)\setminus \{-2\}$ belongs to $\mathrm{int}(F_k)$ for some $1\leq k\leq P$.
\item[{\rm 2.}] If one of the periods $p_1, p_2$ is odd, then $E=-2$ belongs to $\mathrm{int}(F_k)$ for some $1\leq k\leq P$.
\end{enumerate}
\end{theorem}
\begin{proof}[Proof of Theorem~\ref{t:bsc:tri}]
As already discussed, this follows immediately from Theorem~\ref{thm:trimain}.
\end{proof}
\subsection{Proof of Theorem \ref{thm:trimain}} 
We will divide the proof into two different cases: $E\neq -2$ and $E=-2$.
Our general strategy is to argue by contradiction.
More specifically, we assume $E=\min F_{k+1}=\max F_k$ for some $1\leq k\leq P-1$, and show that this leads to a contradiction. We will use the following two lemmas, whose proofs we provide at the end of the present section.
\begin{lemma}\label{lem:constructiontri}
For any $E \in {[-3,6]}$, there exist $x, y \in [0,2\pi)$ such that
\begin{align}
\label{eq:xyCondA} \cos(x) + \cos(y) + \cos(x-y) & = \frac{E}{2} \\
\label{eq:xyCondB} \sin(x) + \sin(y) & = 0.
\end{align}
Furthermore, {if $E \neq -2$}, we have
\begin{align}
\label{eq:xyCondC}\cos(x) + \cos(y)=-1+\sqrt{E+3}  \neq 0
\end{align}
for any $x,y$ that satisfy conditions \eqref{eq:xyCondA} and \eqref{eq:xyCondB}.
\end{lemma}

\begin{lemma}\label{lem:triJ0empty}
Consider the following system:
\begin{equation} \label{eq:triJ0syst}
\begin{cases}
\cos(x) + \cos(y) + \cos(x-y)  = \frac{E}{2},\\
\sin(x)+\sin(x-y)=0,\\
\sin(y)-\sin(x-y)=0.
\end{cases}
\end{equation}
For any $E \in (-3,6) \setminus \{-2\}$, the solution set of \eqref{eq:triJ0syst} is empty. For $E = -2$, the solutions of \eqref{eq:triJ0syst} in $[0,2\pi)^2$ are $(0,\pi)$, $(\pi,0)$ and $(\pi, \pi)$.
\end{lemma}

We will use Lemma \ref{lem:constructiontri} in the $E\neq -2$ case, and Lemma \ref{lem:triJ0empty} in the $E=-2$ case.

\subsubsection{$E\neq -2$}\

\begin{proof}[Proof of Theorem~\ref{thm:trimain}.1]
Let $E \in (-3,6) \setminus\{-2\}$ be given and suppose for the purpose of establishing a contradiction that $E = \max F_k = \min F_{k+1}$ for some $1 \le k < P$. Let $(x,y)$ denote a solution to \eqref{eq:xyCondA} and \eqref{eq:xyCondB} from Lemma \ref{lem:constructiontri}, and take $\tbthe=(\tthe_1,\tthe_2)\in [0,2\pi)^2$ and  $\bm{\ell}^{(1)}=(\ell_1^{(1)},\ell_2^{(1)})\in \Lambda$ such that 
\[p_1^{-1}(\tthe_1+2\pi \ell_1^{(1)})=x,\, \ \text{and }\ \ p_2^{-1}(\tthe_2+2\pi \ell_2^{(1)})=y.
\]
It is clear that $\tbthe$ and $\bm{\ell}^{(1)}$ are uniquely determined by $x$ and $y$.
Let us also note that \eqref{eq:xyCondA} is equivalent to 
\[e_{\bm{\ell}^{(1)}}(\tbthe)=E.\]
Define $\Lambda_E(\tbthe) \subseteq \Lambda$ to be the set of all $\bm{\ell} \in \Lambda$ such that $e_{\bm{\ell}}(\tbthe) = E$. Then $r := |\Lambda_E(\tbthe)|$ is the multiplicity of $E \in \sigma(H(\tbthe))$ and clearly $\bm{\ell}^{(1)} \in \Lambda_E(\tbthe)$.

%Let  $\bm{\ell}^{(2)},\ldots,\bm{\ell}^{(r)}\in \Lambda$ denote all the vectors (if any) that are different from $\bm{\ell}^{(1)}$ such that
%\[e_{\bm{\ell}^{(2)}}(\tbthe)=\cdots=e_{\bm{\ell}^{(r)}}(\tbthe)=E.\]
%Hence $r\geq 1$ is the multiplicity of eigenvalue $E$.

Since $E \in F_k$ by assumption, let $s\in \Z\cap [1,r]$ be chosen so that
\[E_{k-s}(\tbthe)<E_{k-s+1}(\tbthe)=\cdots=E_k(\tbthe)=\cdots=E_{k+r-s}(\tbthe)<E_{k+r-s+1}(\tbthe).\]
Since all the eigenvalues are continuous in $\bthe$, we can take $\varepsilon>0$ small enough such that 
\[E_{k-s}(\bthe)<E_{k-s+1}(\bthe)\, \text{ and }\, E_{k+r-s}(\bthe)<E_{k+r-s+1}(\bthe)\]
hold whenever $\|\bthe-\tbthe\|_{\R^2}<\varepsilon$. Our goal is to perturb about the point $\tbthe$ in two directions, one of which is ``generic'' and one of which is carefully chosen. The generic perturbation moves half of the eigenvalues to the right and half to the left, which we shall use to conclude that $r = 2s$. The non-generic perturbation is carefully chosen to contradict this.

Given {$\bm{\ell}\in \Lambda$ and} a unit vector $\bbe=(\beta_1, \beta_2)$, we have
\begin{align}
e_{\bm{\ell}}(\tbthe+t \bbe) & =e_{\bm{\ell}}(\tbthe)+t\bbe \cdot \nabla e_{\bm{\ell}}(\tbthe)+O(t^2) \label{eq:pertgeneralbetatri1order} \\
& = e_{\bm{\ell}}(\tbthe)+t\bbe \cdot \nabla e_{\bm{\ell}}(\tbthe) \label{eq:pertgeneralbetatri2order}\\
&\qquad -\frac{t^2}{2}\Bigg[\frac{\beta_1^2}{p_1^2}\cos\Big(\frac{\tthe_1+2\pi \ell_1}{p_1}\Big)
+\frac{\beta_2^2}{p_2^2}\cos\Big(\frac{\tthe_2+2\pi \ell_2}{p_2}\Big) \notag\\ 
&\qquad +\Big(\frac{\beta_1}{p_1}-\frac{\beta_2}{p_2}\Big)^2 \cos\Big(\frac{\tthe_1+2\pi \ell_1}{p_1}-\frac{\tthe_2+2\pi \ell_2}{p_2}\Big)\Bigg] +O(t^3).  \notag
\end{align}
In particular, we will use \eqref{eq:pertgeneralbetatri1order} if $\bbe\cdot  \nabla e_{\bm{\ell}}(\tbthe)\neq 0$, and \eqref{eq:pertgeneralbetatri2order} otherwise.

For any vector $\bbe\in \R^2\setminus \{\vec{0}\}$, let 
\begin{equation}\label{def:Jbetatri}
\begin{aligned}
\mathcal{J}_{\bbe}^0
& =
\CJ_{\bbe}^0(\tbthe)
:=
\{\bm{\ell} \in \Lambda_E(\tbthe):\ \bbe\cdot \nabla e_{\bm{\ell}}(\tbthe)=0\},\\
\mathcal{J}_{\bbe}^{\pm}
& =
\CJ_{\bbe}^\pm(\tbthe)
:=
\{ \bm{\ell} \in \Lambda_E(\tbthe) :\ \pm \bbe \cdot \nabla e_{\bm{\ell}}(\tbthe)>0\}.
\end{aligned}
\end{equation}
Consequently, we always have 
\begin{align}\label{eq:sumJbetatri}
|\mathcal{J}_{\bbe}^0|+|\mathcal{J}_{\bbe}^+|+|\mathcal{J}_{\bbe}^-|=r.
\end{align}
We also define $\mathcal{J}_0$ as follows
\begin{align}\label{def:J0tri}
\mathcal{J}_0
=
\CJ_0(\tbthe)
:=\{ \bm{\ell} \in \Lambda_E(\tbthe) :\ \nabla e_{\bm{\ell}}(\tbthe) = \bm{0}\}.
\end{align}
Since $E \neq -2$, Lemma~\ref{lem:triJ0empty} clearly implies $\mathcal{J}_0=\emptyset$.

We choose $\bbe_1=(\beta_{1,1},\beta_{1,2})=(p_1,p_2)/\sqrt{p_1^2+p_2^2}$. Then \eqref{eq:xyCondB} is equivalent to 
\[\bbe_1 \cdot \nabla e_{\bm{\ell}^{(1)}}(\tbthe)=0,\]
hence $\mathcal{J}_{\bbe_1}^0\neq \emptyset$.

Next we are going to perturb the point $\tbthe$ and count the eigenvalues.
Since $\mathcal{J}_0 = \emptyset$, we can choose a unit vector $\bbe_2$ such that 
\begin{align}\label{eq:beta2trinon-empty}
\bbe_2\cdot \nabla e_{\bm{\ell}}(\tbthe)\neq 0,
\end{align}
holds for any $\bm{\ell} \in \Lambda_E(\tbthe)$.
Thus, $\mathcal{J}_{\bbe_2}^0=\emptyset$, so one concludes
\begin{align}\label{eq:beta2trinon-empty'}
|\mathcal{J}_{\bbe_2}^+|+|\mathcal{J}_{\bbe_2}^-|=r.
\end{align}

\subsubsection*{Perturbation along $\bbe_2$}
We first perturb the eigenvalues along the $\bbe_2$ direction.
Since $\mathcal{J}_{\bbe_2}^0 = \emptyset$, we will always employ \eqref{eq:pertgeneralbetatri1order}.

For $t > 0$ small enough, we have the following.
\begin{itemize}
\item  If ${\bm{\ell}}  \in \mathcal{J}_{\bbe_2}^+$, we have
\[
E_{k+r-s+1}(\tbthe + t\bbe_2)
>
e_{{\bm{\ell}}}(\tbthe + t\bbe_2)
>
E
=
\max F_k,
\]
which implies 
\begin{align}\label{eq1:Jbeta2+tri}
|\mathcal{J}_{\bbe_2}^+|\leq r-s.
\end{align}

\item If ${\bm{\ell} } \in \mathcal{J}_{\bbe_2}^-$, we have 
\[
E_{k-s}(\tbthe + t\bbe_2)
<
e_{{\bm{\ell}}}(\tbthe + t\bbe_2)
<
E
=
\min F_{k+1},
\]
which implies
\begin{align}\label{eq2:Jbeta2+tri}
|\mathcal{J}_{\bbe_2}^-|\leq s.
\end{align}
\end{itemize}
{In view of \eqref{eq:beta2trinon-empty'}, Equations~\eqref{eq1:Jbeta2+tri} and \eqref{eq2:Jbeta2+tri} imply
\begin{equation} \label{eq:Jbeta2MinusCard+tri}
|\mathcal{J}_{\bbe_2}^-|
=
s.
\end{equation}
Upon realizing that $\mathcal{J}_{-\bbe_2}^0 = \emptyset$ and $\mathcal{J}_{-\bbe_2}^\pm = \mathcal{J}_{\bbe_2}^\mp$, we may apply the analysis above with $\bbe_2$ replaced by $-\bbe_2$ and conclude that
\begin{equation} \label{eq:JMinusBeta2+tri}
|\mathcal{J}_{\bbe_2}^+| = |\mathcal{J}_{-\bbe_2}^-| = s.
\end{equation}
In particular, \eqref{eq:Jbeta2MinusCard+tri} and \eqref{eq:JMinusBeta2+tri} imply
\begin{align}\label{eq4:Jbeta2tri}
r=2s.
\end{align}} 

\begin{comment}
\begin{align}\label{eq3:Jbeta2+tri}
|\mathcal{J}_{\bbe_2}^+|=r-s.
\end{align}

When $t$ is small enough and $t < 0$, we have the following.
\begin{itemize}
\item If $(\ell_m,k_m)\in \mathcal{J}_{\bbe_2}^+$, we have
\[
E_{k-s}(\widetilde{\theta}_1 + t\beta_{2,1}, \widetilde{\theta}_2 + t\beta_{2,2})
<
e_{\ell_m,k_m}(\widetilde{\theta}_1+t\beta_{2,1}, \widetilde{\theta}_2 + t\beta_{2,2})
<
E
=
\min F_{k+1},
\]
which implies 
\begin{align}\label{eq1:Jbeta2-tri}
|\mathcal{J}_{\bbe_2}^+|\leq s.
\end{align}
\item If $(\ell_m,k_m)\in \mathcal{J}_{\bbe_2}^-$, we have 
\[
E_{k+r-s+1}(\widetilde{\theta}_1 + t\beta_{2,1}, \widetilde{\theta}_2 + t\beta_{2,2})
>
e_{\ell_m,k_m}(\widetilde{\theta}_1 + t\beta_{2,1}, \widetilde{\theta}_2 + t\beta_{2,2})
>
E
=
\max F_k,
\]
which implies
\begin{align}\label{eq2:Jbeta2-tri}
|\mathcal{J}_{\bbe_2}^-|\leq r-s.
\end{align}
\end{itemize}
Taking \eqref{eq:beta2trinon-empty'}, \eqref{eq1:Jbeta2-tri} and \eqref{eq2:Jbeta2-tri} into account, we have
\begin{align}\label{eq3:Jbeta2-tri}
|\mathcal{J}_{\bbe_2}^+|=s.
\end{align}
Combining \eqref{eq3:Jbeta2+tri} with \eqref{eq3:Jbeta2-tri}, we arrive at
\begin{align}\label{eq4:Jbeta2tri}
r=2s.
\end{align}
\end{comment}

\subsubsection*{Perturbation along $\bbe_1$}
Now we perturb the eigenvalues along $\bbe_1$.
The case when ${\bm{\ell}} \in \mathcal{J}_{\bbe_1}^{\pm}$ is similar to that of $\bbe_2$.
The difference here is $\mathcal{J}_{\bbe_1}^0\neq \emptyset$.

By Lemma~\ref{lem:constructiontri}, we have
\begin{align}
\cos\Big(\frac{\widetilde{\theta}_1 + 2\pi {\ell_1}}{p_1}\Big) + \cos\Big(\frac{\widetilde{\theta}_2 + 2\pi {\ell_2}}{p_2}\Big)
=
-1+\sqrt{E+3}
\neq
0
\end{align}
for ${\bm{\ell} = (\ell_1,\ell_2)} \in \mathcal{J}_{\bbe_1}^0$.
Thus, by employing \eqref{eq:pertgeneralbetatri2order}, we obtain
\begin{align}
&e_{{\bm{\ell}}}(\tbthe + t\bbe_1)\notag\\
&=
E - {\frac{ t^2}{2(p_1^2+p_2^2)}} \Big(\cos\Big(\frac{\widetilde{\theta}_1 + 2\pi {\ell_1}}{p_1}\Big) + \cos\Big(\frac{\widetilde{\theta}_2 + 2\pi {\ell_2}}{p_2}\Big) \Big) + O(t^3)\notag\\
& = E - {\frac{ t^2}{2(p_1^2+p_2^2)}} \Big(-1+\sqrt{E+3} \Big)+O(t^3).\label{eq:Jbeta10tri}
\end{align}
Notice that the choice of $\bbe_1$ causes the {third} $t^2$ term of \eqref{eq:pertgeneralbetatri2order} to drop out.

Without loss of generality, we assume $E\in (-2, 6)$. The other case can be handled similarly.
For $E\in (-2,6)$, \eqref{eq:Jbeta10tri} implies that 
\begin{align}\label{eq:Jbeta10tri'}
e_{{\bm{\ell}}}(\tbthe + t\bbe_1) < E = \min F_{k+1},
\end{align}
holds for $|t|>0$ small enough and for any $\bm{\ell} \in \mathcal{J}_{\bbe_1}^0$. 

Combining \eqref{eq:Jbeta10tri'} with \eqref{eq:pertgeneralbetatri1order}, we have the following.

For $t>0$ small enough,
\begin{itemize}
\item If ${\bm{\ell}} \in \mathcal{J}_{\bbe_1}^+$, we have
\[
E_{k+r-s+1}(\tbthe + t\bbe_1)
>
e_{{\bm{\ell}}}(\tbthe + t\bbe_1)
>
E
=
\max F_k,
\]
which implies 
\begin{align}\label{eq1:Jbeta1+tri}
|\mathcal{J}_{\bbe_1}^+|\leq r-s
=
s,
\end{align}
{where the equality follows from \eqref{eq4:Jbeta2tri}.}

\item If ${\bm{\ell}} \in \mathcal{J}_{\bbe_1}^0 \bigcup \mathcal{J}_{\bbe_1}^-$, we have 
\[
E_{k-s-1}(\tbthe + t\bbe_1)
<
e_{{\bm{\ell}}}(\tbthe + t\bbe_1)
<
E
=
\min F_{k+1},
\]
which implies
\begin{align}\label{eq2:Jbeta1+tri}
|\mathcal{J}_{\bbe_1}^0|+|\mathcal{J}_{\bbe_1}^-|\leq s.
\end{align}
\end{itemize}
{In view of \eqref{eq:sumJbetatri} and \eqref{eq4:Jbeta2tri}, Equations~\eqref{eq1:Jbeta1+tri} and \eqref{eq2:Jbeta1+tri} yield
\begin{equation} \label{eq3:Jbeta1+tri}
|\mathcal{J}_{\bbe_1}^+|
=
|\mathcal{J}_{\bbe_1}^0|+|\mathcal{J}_{\bbe_1}^-|
=
s.
\end{equation}
As before, we may observe that $\CJ_{-\bbe_1}^0 = \CJ_{\bbe_1}^0$ and $\CJ_{-\bbe_1}^\pm = \CJ_{\bbe_1}^\mp$. Then, the analysis above applied with $\bbe_1$ replaced by $-\bbe_1$ forces
\begin{equation} \label{eq3:Jbeta1-tri}
|\mathcal{J}_{\bbe_1}^-|
=
|\mathcal{J}_{\bbe_1}^0|+|\mathcal{J}_{\bbe_1}^+|
=
s.
\end{equation}
Taken together, \eqref{eq3:Jbeta1+tri} and \eqref{eq3:Jbeta1-tri} imply $|\CJ_{\bbe_1}^0| = 0$, which contradicts $\CJ_{\bbe_1}^0 \neq \emptyset$.
}

\begin{comment}
Taking \eqref{eq:sumJbetatri}, \eqref{eq1:Jbeta1+tri} and \eqref{eq2:Jbeta1+tri} into account, we have
\begin{align}\label{eq3:Jbeta1+tri}
|\mathcal{J}_{\bbe_1}^+|=r-s.
\end{align}

For $t<0$ small enough, 
\begin{itemize}
\item If $(\ell_m,k_m)\in  \mathcal{J}_{\bbe_1}^0 \bigcup \mathcal{J}_{\bbe_1}^+ $, we have
\[
E_{k-s}(\widetilde{\theta}_1 + t\beta_{1,1}, \widetilde{\theta}_2 + t\beta_{1,2})
<
e_{\ell_m,k_m}(\widetilde{\theta}_1 + t\beta_{1,1}, \widetilde{\theta}_2 + t\beta_{1,2})
<
E
=
\min F_{k+1},
\]
which implies 
\begin{align}\label{eq1:Jbeta1-tri}
|\mathcal{J}_{\bbe_1}^0|+|\mathcal{J}_{\bbe_1}^+|\leq s.
\end{align}
\item If $(\ell_m,k_m)\in \mathcal{J}_{\bbe_1}^-$, we have 
\[
E_{k+r-s+1}(\widetilde{\theta}_1 + t\beta_{1,1}, \widetilde{\theta}_2 + t\beta_{1,2})
>
e_{\ell_m,k_m}(\widetilde{\theta}_1 + t\beta_{1,1}, \widetilde{\theta}_2 + t\beta_{1,2})
>
E
=
\max F_k,
\]
which implies
\begin{align}\label{eq2:Jbeta1-tri}
|\mathcal{J}_{\bbe_1}^-|\leq r-s.
\end{align}
\end{itemize}
Taking \eqref{eq:sumJbetatri}, \eqref{eq1:Jbeta1-tri} and \eqref{eq2:Jbeta1-tri} into account, we have
\begin{align}\label{eq3:Jbeta1-tri}
|\mathcal{J}_{\bbe_1}^0|+|\mathcal{J}_{\bbe_1}^+|=s.
\end{align}
Combining \eqref{eq3:Jbeta1+tri} with \eqref{eq3:Jbeta1-tri}, we arrive at
\begin{align}\label{eq4:Jbeta1tri}
r=2s-|\mathcal{J}_{\bbe_1}^0|.
\end{align}
Since $\mathcal{J}_{\bbe_1}^0\neq \emptyset$, \eqref{eq4:Jbeta1tri} is in contradiction with \eqref{eq4:Jbeta2tri}.
\end{comment}

\end{proof}

\subsubsection{$E=-2$}\

First, we would like to make a remark on our strategy of the proof of the $E=-2$ case, and on the importance of one of the periods being odd.
\begin{remark}\label{rem:tri}
We will choose $\tbthe=(\tthe_1, \tthe_2)$ and $\bm{\ell}^{(1)}=(\ell_1^{(1)}, \ell_2^{(1)})$ such that $e_{\bm{\ell}^{(1)}}(\tbthe)=-2$ and 
$\nabla e_{\bm{\ell}^{(1)}}(\tbthe)=\bm{0}$. 
Lemma \ref{lem:triJ0empty} yields three possibilities $(p_1^{-1}(\tthe_1+2\pi \ell_1^{(1)}), p_2^{-1}(\tthe_2+2\pi \ell_2^{(1)}))=(0,\pi)$, $(\pi, 0)$ or $(\pi, \pi)$.
Depending on which one of $p_1, p_2$ is odd, we will choose $(0,\pi)$ (if $p_1$ is odd), or $(\pi, 0)$ (if $p_2$ is odd).
This choice guarantees that the only eigenvalue located at $-2$ with vanishing gradient is $e_{\bm{\ell}^{(1)}}(\tbthe)$.
Consequently, it suffices to control the second order perturbation of (a single eigenvalue) $e_{\bm{\ell}^{(1)}}(\tbthe)$ along a given direction $(\beta_1, \beta_2)$.
When $p_1$ is odd, this is equivalent to controlling the sign of the following expression (compare \eqref{eq:remtri}):
$$-\beta_{2}\Big(\frac{\beta_{1}}{p_1}-\frac{\beta_{2}}{p_2}\Big).$$
We can easily choose two directions such that the expression above has different signs, which leads to un-even {eigenvalue counts and hence to the desired contradiction}.

{{\it A posteriori}, the existence of a $(2,2)$-periodic potential satisfying the conclusion of Theorem~\ref{t:triExamples} implies that this argument must fail if both $p_1$ and $p_2$ are even; let us briefly describe why this must be the case.} If both $p_1, p_2$ are even, there will be three eigenvalues at $-2$ with vanishing gradients, corresponding to all three solutions $(0,\pi)$, $(\pi, 0)$, $(\pi, \pi)$.
Trying to control the second order perturbations of all these three eigenvalues along $(\beta_1, \beta_2)$ is equivalent to controlling the signs of the following three expressions simultaneously
\begin{align*}
-\beta_{2}\Big(\frac{\beta_{1}}{p_1}-\frac{\beta_{2}}{p_2}\Big),\ \ 
\beta_{1}\Big(\frac{\beta_{1}}{p_1}-\frac{\beta_{2}}{p_2}\Big),\ \ \text{and}\ \ 
\beta_1\beta_2.
\end{align*}
A simple inspection of these three expressions yields that two of them are always non-negative with the other one being non-positive. 
Therefore we can never choose two different directions that lead to un-even {eigenvalue counts}.
This explains why at least one of the periods must be odd for our argument to work. 
\end{remark}

\begin{proof}[Proof of Theorem~\ref{thm:trimain}.2]
Now let us give a detailed proof. Without loss of generality, assume $p_1$ is odd, let $E = -2$, and assume for the sake of contradiction that $E = \max F_k = \min F_{k+1}$ for some $k$.
We choose $\widetilde\bthe$ and $\bm{\ell}^{(1)}$ via
\begin{align*}
\tthe_1=0,\ {\ell_1^{(1)}} = 0,\ \ \ 
(\tthe_2, {\ell_2^{(1)}})=
\begin{cases}
\Big(0, \frac{p_2}{2}\Big),\ \ \text{if } p_2 \text{ is even},\\
\Big(\pi, \frac{p_2-1}{2}\Big),\ \ \text{if } p_2 \text{ is odd}.
\end{cases}
\end{align*}

With these choices of $\bm{\ell}^{(1)}$ and $\widetilde{\bthe}$, one can check that $e_{{\bm{\ell}^{(1)}}}(\widetilde\bthe) = -2 = E$. As before, let $r$ denote the multiplicity of $E$ and let $\Lambda_E(\tbthe)$ denote the set of $\bm{\ell} \in \Lambda$ with $e_{\bm{\ell}}(\widetilde\bthe) = -2$.
Note that we also have $\nabla e_{{\bm{\ell}^{(1)}}}(\widetilde\bthe) = \bm{0}$, and thus $\mathcal{J}_0\neq \emptyset$. Moreover, we claim that $\CJ_0 = \{\bm{\ell}^{(1)}\}$. To see this, suppose there exists $\bm{\ell} \neq \bm{\ell}^{(1)}$ in $\CJ_0$. In view of Lemma~\ref{lem:triJ0empty}, we must have
\[\frac{\tthe_1+2\pi\ell_1}{p_1}=\pi,\]
which implies $p_1 = 2\ell_1$, which is impossible, since $p_1$ is odd.
Consequently,
\[\mathcal{J}_0=
{\{ \bm{\ell}^{(1)} \}}.\]

Let us choose $\bbe_1=(\beta_{1,1},\beta_{1,2})=(0,1)$ and a unit vector
\[\bbe_2=(\beta_{2,1},\beta_{2,2}) \sim (2p_1, p_2)/\sqrt{4p_1^2+p_2^2}\] 
such that 
\begin{align}\label{eq:beta2triodd1}
\beta_{2,2}\Big(\frac{\beta_{2,1}}{p_1}-\frac{\beta_{2,2}}{p_2}\Big)>0,
\end{align}
and
\begin{align}\label{eq:beta2triodd2}
\bbe_2\cdot \nabla e_{{\bm{\ell}}}(\tbthe)\neq 0\ \ \text{holds for any }\ \ell \in \Lambda_E(\tbthe) \setminus \{\bm{\ell}^{(1)}\}.
\end{align}

We will use \eqref{eq:beta2triodd1} to control the perturbation of $e_{{\bm{\ell}^{(1)}}}(\tbthe)$ along the $\bbe_2$ direction. 
We also note that \eqref{eq:beta2triodd2} simply says 
\begin{align}\label{eq:beta20triodd}
\mathcal{J}_{\bbe_2}^0=\mathcal{J}_{0}
=
\{{\bm{\ell}^{(1)}}\}.
\end{align}

\subsubsection*{Perturbation along $\bbe_2$}
We first perturb the eigenvalues along the $\bbe_2$ direction.\

By \eqref{eq:beta20triodd}, we need only consider first-order perturbation theory as in \eqref{eq:pertgeneralbetatri1order} for $\bm{\ell} \in \Lambda_E(\tbthe) \setminus \{\bm{\ell}^{(1)}\}$.
Since ${\bm{\ell}^{(1)}} \in \mathcal{J}_{0}$, we need to employ \eqref{eq:pertgeneralbetatri2order} for $e_{{\bm{\ell}^{(1)}}}$.
Indeed, by \eqref{eq:pertgeneralbetatri2order}, we have for $|t|>0$ small enough,
\begin{align}\label{eq:remtri}
e_{{\bm{\ell}^{(1)}}}(\tbthe + t\bbe_2)
& =
e_{{\bm{\ell}^{(1)}}}(\tbthe)
-\frac{t^2}{2}
\Bigg[\frac{\beta_{2,1}^2}{p_1^2}-\frac{\beta_{2,2}^2}{p_2^2}-\Big(\frac{\beta_{2,1}}{p_1}-\frac{\beta_{2,2}}{p_2}\Big)^2\Bigg]+O(t^3) \notag\\
& =-2-\frac{\beta_{2,2}}{p_2}\Big(\frac{\beta_{2,1}}{p_1}-\frac{\beta_{2,2}}{p_2}\Big)t^2+O(t^3)\\
& <-2 \notag\\
& = \min F_{k+1}, \notag
\end{align}
where we used \eqref{eq:beta2triodd1} in the last inequality.

For $t > 0$ small enough, we then have the following.
\begin{itemize}
\item If ${\bm{\ell}} \in \mathcal{J}_{\bbe_2}^+$, we have
\[
E_{k+r-s+1}(\tbthe + t\bbe_2)
>
e_{{\bm{\ell}}}(\tbthe + t\bbe_2)
>
E
=
\max F_k,
\]
which implies 
\begin{align}\label{eq1:Jbeta2+triodd}
|\mathcal{J}_{\bbe_2}^+|\leq r-s.
\end{align}

\item If ${\bm{\ell}} \in \mathcal{J}_{\bbe_2}^-\bigcup \mathcal{J}_0$, we have 
\[
E_{k-s}(\tbthe + t\bbe_2)
<
e_{{\bm{\ell}}}(\tbthe + t\bbe_2)
<
E
=
\min F_{k+1},
\]
which implies
\begin{align}\label{eq2:Jbeta2+triodd}
|\mathcal{J}_{\bbe_2}^-|+|\mathcal{J}_0|\leq s.
\end{align}
\end{itemize}
Taking \eqref{eq:sumJbetatri}, {\eqref{eq:beta20triodd},} \eqref{eq1:Jbeta2+triodd}, and \eqref{eq2:Jbeta2+triodd} into account, we have
\begin{align}\label{eq3:Jbeta2+triodd}
{|\mathcal{J}_{\bbe_2}^-|=s-1}.
\end{align}
{Replacing $\bbe_2$ by $-\bbe_2$ as in previous phases of the argument, we arrive at
\begin{equation} \label{eq3:Jbeta2-triodd}
|\CJ_{\bbe_2}^+|
=
s-1.
\end{equation}}

\begin{comment}
When $t$ is small enough and $t < 0$, we have the following.
\begin{itemize}
\item If $(\ell_m,k_m)\in \mathcal{J}_{\bbe_2}^+\bigcup \mathcal{J}_0$, we have
\[
E_{k-s}(\widetilde{\theta}_1 + t\beta_{2,1}, \widetilde{\theta}_2 + t\beta_{2,2})
<
e_{\ell_m,k_m}(\widetilde{\theta}_1+t\beta_{2,1}, \widetilde{\theta}_2 + t\beta_{2,2})
<
E
=
\min F_{k+1},
\]
which implies 
\begin{align}\label{eq1:Jbeta2-triodd}
|\mathcal{J}_{\bbe_2}^+|+|\mathcal{J}_0|\leq s.
\end{align}
\item If $(\ell_m,k_m)\in \mathcal{J}_{\bbe_2}^-$, we have 
\[
E_{k+r-s+1}(\widetilde{\theta}_1 + t\beta_{2,1}, \widetilde{\theta}_2 + t\beta_{2,2})
>
e_{\ell_m,k_m}(\widetilde{\theta}_1 + t\beta_{2,1}, \widetilde{\theta}_2 + t\beta_{2,2})
>
E
=
\max F_k,
\]
which implies
\begin{align}\label{eq2:Jbeta2-triodd}
|\mathcal{J}_{\bbe_2}^-|\leq r-s.
\end{align}
\end{itemize}
Taking \eqref{eq:sumJbetatri}, \eqref{eq1:Jbeta2-triodd} and \eqref{eq2:Jbeta2-triodd} into account, we have
\begin{align}\label{eq3:Jbeta2-triodd}
|\mathcal{J}_{\bbe_2}^+|+|\mathcal{J}_0|=s.
\end{align}
\end{comment}
Combining \eqref{eq3:Jbeta2+triodd} with \eqref{eq3:Jbeta2-triodd}, we arrive at
\begin{align}\label{eq4:Jbeta2triodd}
{r
= |\CJ_{\bbe_2}^+| + |\CJ_{\bbe_2}^-| + |\CJ_0|
=
2s-1.}
\end{align}

\subsubsection*{Perturbation along $\bbe_1$}
Now we perturb the eigenvalues along $\bbe_1 = (0,1)$.
The case when ${\bm{\ell}} \in \mathcal{J}_{\bbe_1}^{\pm}$ is similar to that of $\bbe_2$.
The difference here is the {behavior of} perturbations of $e_{{\bm{\ell}^{(1)}}}$ {in the direction $\bbe_1$.}
Indeed, by \eqref{eq:pertgeneralbetatri2order}, we have
\begin{align*}
e_{{\bm{\ell}^{(1)}}}(\tbthe + t\bbe_1)
&=e_{{\bm{\ell}^{(1)}}}(\tbthe)
-\frac{t^2}{2}
\Bigg[\frac{\beta_{1,1}^2}{p_1^2}-\frac{\beta_{1,2}^2}{p_2^2}-\Big(\frac{\beta_{1,1}}{p_1}-\frac{\beta_{1,2}}{p_2}\Big)^2\Bigg]+O(t^3)\\
&=-2+\frac{t^2}{p_2^2}+O(t^3)\\
&>-2=\max F_{k}.
\end{align*}
Thus, the perturbations of $e_{{\bm{\ell}^{(1)}}}$ {in the direction $\bbe_1$} always move up.

For $t>0$ small enough,
\begin{itemize}
\item If ${\bm{\ell}} \in \mathcal{J}_{\bbe_1}^+\bigcup \mathcal{J}_0$, we have
\[
E_{k+r-s+1}(\tbthe + t\bbe_1)
>
e_{{\bm{\ell}}}(\tbthe + t\bbe_1)
>
E
=
\max F_k,
\]
which implies 
\begin{align}\label{eq1:Jbeta1+triodd}
|\mathcal{J}_{\bbe_1}^+|+|\mathcal{J}_0| \leq r-s.
\end{align}
\item If ${\bm{\ell}} \in \mathcal{J}_{\bbe_1}^-$, we have 
\[
E_{k-s-1}(\tbthe + t\bbe_1)
<
e_{{\bm{\ell}}}(\tbthe + t\bbe_1)
<
E
=
\min F_{k+1},
\]
which implies
\begin{align}\label{eq2:Jbeta1+triodd}
|\mathcal{J}_{\bbe_1}^-| \leq s.
\end{align}
\end{itemize}
In view of \eqref{eq:sumJbetatri}, Equations~\eqref{eq1:Jbeta1+triodd} and \eqref{eq2:Jbeta1+triodd} yield
\begin{align}\label{eq3:Jbeta1+triodd}
{|\mathcal{J}_{\bbe_1}^-|
=
s.}
\end{align}
{Applying the usual symmetry argument, we also arrive at $|\CJ_{\bbe_1}^+| = s$, which leads to
\[
r
=
|\CJ_{\bbe_1}^+| + |\CJ_{\bbe_1}^-| + |\CJ_0|
=
2s+1,
\]
which in turn contradicts \eqref{eq4:Jbeta2triodd}.}
\end{proof}

\begin{comment}
For $t<0$ small enough, 
\begin{itemize}
\item If $(\ell_m,k_m)\in \mathcal{J}_{\bbe_1}^+ $, we have
\[
E_{k-s}(\widetilde{\theta}_1 + t\beta_{1,1}, \widetilde{\theta}_2 + t\beta_{1,2})
<
e_{\ell_m,k_m}(\widetilde{\theta}_1 + t\beta_{1,1}, \widetilde{\theta}_2 + t\beta_{1,2})
<
E
=
\min F_{k+1},
\]
which implies 
\begin{align}\label{eq1:Jbeta1-triodd}
|\mathcal{J}_{\bbe_1}^+|\leq s.
\end{align}
\item If $(\ell_m,k_m)\in \mathcal{J}_{\bbe_1}^-\bigcup \mathcal{J}_0$, we have 
\[
E_{k+r-s+1}(\widetilde{\theta}_1 + t\beta_{1,1}, \widetilde{\theta}_2 + t\beta_{1,2})
>
e_{\ell_m,k_m}(\widetilde{\theta}_1 + t\beta_{1,1}, \widetilde{\theta}_2 + t\beta_{1,2})
>
E
=
\max F_k,
\]
which implies
\begin{align}\label{eq2:Jbeta1-triodd}
|\mathcal{J}_{\bbe_1}^-|+|\mathcal{J}_0|\leq r-s.
\end{align}
\end{itemize}
Taking \eqref{eq:sumJbetatri}, \eqref{eq1:Jbeta1-triodd} and \eqref{eq2:Jbeta1-triodd} into account, we have
\begin{align}\label{eq3:Jbeta1-triodd}
|\mathcal{J}_{\bbe_1}^+|=s.
\end{align}
Combining \eqref{eq3:Jbeta1+triodd} with \eqref{eq3:Jbeta1-triodd}, we arrive at
\begin{align}\label{eq4:Jbeta1triodd}
r=2s+|\mathcal{J}_0|=2s+1.
\end{align}
This contradicts \eqref{eq4:Jbeta2triodd}.
\end{comment}

\subsection{Proof of Lemmas \ref{lem:constructiontri} and \ref{lem:triJ0empty}}\label{sec:constructionlemmaproof}
\begin{proof}[Proof of Lemma~\ref{lem:constructiontri}]
Let {$E \in [-3,6]$ be given}, let $x$ be as-yet-unspecified, set $y = 2\pi - x$, and note that \eqref{eq:xyCondB} holds. Then, using $y = 2\pi - x$, we note that
\begin{align*}
\cos(x) + \cos(y) + \cos(x-y)
& =
2\cos(x) + \cos(2x) \\
& =
2\cos(x) + 2\cos^2(x) - 1.
\end{align*}
Setting $z = \cos(x)$, we seek to solve $2z+2z^2-1 = E/2$, which gives
\[
z^2 + z - \frac{1}{2} - \frac{E}{4} = 0
\implies
z = 
\frac{-1 \pm \sqrt{3+E}}{2}.
\]
Thus, we may take $x$ so that
\[
\cos(x)
=
\frac{-1+\sqrt{3+E}}{2}.
\]
In fact, since $-3 \le E \le 6$, we may take $0 \le x \le 2\pi/3$. Thus, with this choice of $x$ (and $y = 2\pi - x$), we get \eqref{eq:xyCondA}. 

Finally, suppose $x$ and $y$ solve \eqref{eq:xyCondA} and \eqref{eq:xyCondB} for $E \neq -2$. From \eqref{eq:xyCondB}, we deduce that either $x+y = 2\pi$ or $|x-y|=\pi$. The second option leads to $E = -2$, so we must have $y=2\pi - x$. Then, $E \neq -2$ guarantees
\[
\cos(x) + \cos(2\pi - x)
=
2\cos(x)
=
-1 + \sqrt{3+E}
\neq 0,
\]
which proves \eqref{eq:xyCondC}.
\end{proof}

\begin{proof}[Proof of Lemma~\ref{lem:triJ0empty}]
Suppose that $x$ and $y$ solve
\begin{align}
\label{eq:triJ0empty:cossum}
\cos(x) + \cos(y) + \cos(x-y)
& =
\lambda \\
\label{eq:triJ0empty:sin1}
\sin(x) + \sin(x-y) & = 0 \\
\label{eq:triJ0empty:sin2}
\sin(y) - \sin(x-y) & = 0
\end{align}
for some $\lambda \in (-3/2,3)$. Adding \eqref{eq:triJ0empty:sin1} and \eqref{eq:triJ0empty:sin2}, we arrive at 
\[
\sin(x) = - \sin(y).
\]
For $(x,y) \in [0,2\pi)^2$, this forces either $|x-y| = \pi$ or  $x+y = 2\pi$. In the case $|x-y| = \pi$, substituting in to \eqref{eq:triJ0empty:sin1} and \eqref{eq:triJ0empty:sin2} gives $\sin(x) = \sin(y) = 0$, forcing $x,y \in \{0,\pi\}$. Plugging the various possibilities into \eqref{eq:triJ0empty:cossum}, one either gets $\lambda = 3 \notin (-3/2,3)$ (when $x=y=0$) or $\lambda = -1$ (when at least one of $x$ or $y$ is $\pi$).

Alternatively, if $x = 2\pi - y$, \eqref{eq:triJ0empty:sin1} yields $\sin(x) + \sin(2x) = 0$,  which leads to
\[
\sin(x)(1 + 2\cos(x))
=
0.
\]
Setting $\sin(x) = 0$ yields $x \in \{0,\pi\}$ which leads to the same solutions as before. Setting $1 +2\cos(x) = 0$ yields $(x,y) = (2\pi/3,4\pi/3)$ or $(x,y) = (4\pi/3, 2\pi/3)$. Plugging in either possibility into \eqref{eq:triJ0empty:cossum} yields
\[
\cos(x) + \cos(y) + \cos(x-y)
=
-\frac{3}{2} \notin (-3/2,3),
\]
as claimed.
\end{proof}

\subsection{\boldmath Opening a Gap at $-2$}
Let us exhibit a $(2,2)$-periodic potential that perturbatively opens a gap at energy $E = -2$ for the triangular lattice.

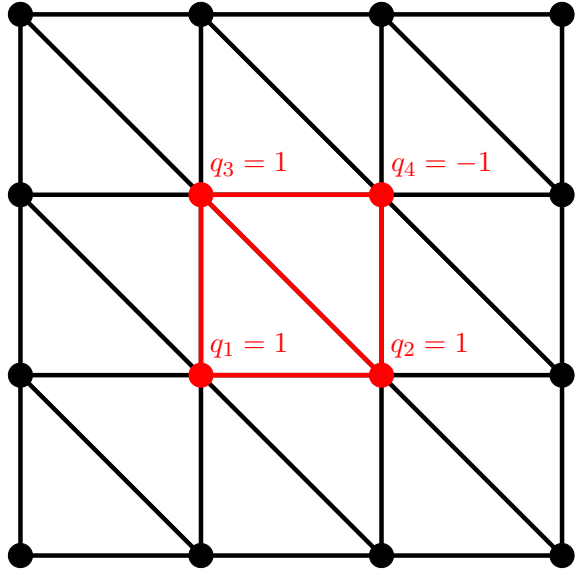
\begin{figure*}[t]

\begin{tikzpicture}[yscale=.8,xscale=.8]
\draw [-,line width = .06cm] (0,0) -- (9,0);
\draw [-,line width = .06cm] (0,0) -- (0,9);
\draw [-,line width=.06cm] (0,9) -- (9,9);
\draw [-,line width = .06cm] (9,0) -- (9,9);
\draw [-,line width = .06cm] (0,9) -- (9,0);
\draw [-,line width = .06cm] (3,0) -- (3,9);
\draw [-,line width = .06cm] (6,0) -- (6,9);
\draw [-,line width = .06cm] (0,3) -- (9,3);
\draw [-,line width = .06cm] (0,6) -- (9,6);
\draw [-,line width = .06cm] (0,3) -- (3,0);
\draw [-,line width = .06cm] (0,6) -- (6,0);
\draw [-,line width = .06cm] (3,9) -- (9,3);
\draw [-,line width = .06cm] (6,9) -- (9,6);
\filldraw[color=black, fill=black](0,0) circle (.2);
\filldraw[color=black, fill=black](0,3) circle (.2);
\filldraw[color=black, fill=black](0,6) circle (.2);
\filldraw[color=black, fill=black](0,9) circle (.2);
\filldraw[color=black, fill=black](3,0) circle (.2);
\filldraw[color=red, fill=red](3,3) circle (.2);
\filldraw[color=red, fill=red](3,6) circle (.2);
\filldraw[color=black, fill=black](3,9) circle (.2);
\filldraw[color=black, fill=black](6,0) circle (.2);
\filldraw[color=red, fill=red](6,3) circle (.2);
\filldraw[color=red, fill=red](6,6) circle (.2);
\filldraw[color=black, fill=black](6,9) circle (.2);
\filldraw[color=black, fill=black](9,0) circle (.2);
\filldraw[color=black, fill=black](9,3) circle (.2);
\filldraw[color=black, fill=black](9,6) circle (.2);
\filldraw[color=black, fill=black](9,9) circle (.2);
\draw [-,line width = .06cm,color=red] (3,3) -- (3,6);
\draw [-,line width = .06cm,color=red] (3,3) -- (6,3);
\draw [-,line width = .06cm,color=red] (6,3) -- (6,6);
\draw [-,line width = .06cm,color=red] (3,6) -- (6,6);
\draw [-,line width = .06cm,color=red] (6,3) -- (3,6);
\node  at (3.8,3.5) {$\hot{q_1=1}$};
\node  at (6.8,3.5) {$\hot{q_2=1}$};
\node  at (3.8,6.5) {$\hot{q_3=1}$};
\node  at (7,6.5) {$\hot{q_4=-1}$};
\end{tikzpicture}
\caption{A $(2,2)$ periodic potential on the triangular lattice with a gap at $E = -2$ for all positive coupling constants.}\label{fig:tri2x2fundDomain}
\end{figure*}

\begin{theorem} \label{thm:triExampleGapLength}
Define
\[
Q_{n,m} 
= 
(-1)^{mn}
=
\begin{cases}
1 & \text{ if } m \text{ or } n \text{ is even}, \\
-1 & \text{ if both } m \text{ and } n \text{ are odd,}
\end{cases}
\]
and denote $H_\lambda = \Delta_\tri + \lambda Q$. For all $\lambda > 0$, $\sigma(H_\lambda)$ has two connected components. Moreover, for all $\lambda > 0$ sufficiently small, the gap that opens about $E=-2$ is precisely equal to
\[
\mathfrak{g}_\lambda
=
\left(-\sqrt{4+\lambda^2},-2+\lambda \right).
\]
In particular,
\[
|\mathfrak{g}_\lambda|
=
\lambda + \left( \sqrt{4+\lambda^2}-2 \right)
\sim
\lambda + \frac{\lambda^2}{2},
\]
so the gap opens linearly as $\lambda \downarrow 0$.
\end{theorem}

The following lemma will be used:

\begin{lemma} \label{lem:triGapLength:trigPolyEst}
For all $\bthe \in \T^2$ and all $0 \leq a \leq 54$,
\[
4(\sin\theta_1 + \sin\theta_2 - \sin(\theta_1+\theta_2))^2
+ a(1+\cos\theta_1+\cos\theta_2+\cos(\theta_1+\theta_2))
\geq 0.
\]
\end{lemma}

\begin{proof}
Define
\[
g(\theta_1,\theta_2,a)
=
4(\sin\theta_1 + \sin\theta_2 - \sin(\theta_1+\theta_2))^2
+ a(1+\cos\theta_1+\cos\theta_2+\cos(\theta_1+\theta_2)).
\]
We begin by checking the boundary of $\T^2 \times [0,54]$. It is easy to see that $g \geq 0$ if $a = 0$. For $a = 54$, define $h(\bthe) = g(\bthe,54)$. Using the identities 
\begin{align*}
\sin x+\sin y-\sin(x+y)&=4\sin\left(\frac{x}{2}\right) \sin\left(\frac{y}{2}\right) \sin\left(\frac{x+y}{2}\right),\\
\cos x-\cos (x+y)&=2\sin\left(\frac{y}{2}\right) \sin\left(x+\frac{y}{2}\right)\\
\sin x+\sin (x+y)&=2\cos\left(\frac{y}{2}\right) \sin\left(x+\frac{y}{2}\right),
\end{align*}
we may simplify $\nabla h$ to get
\begin{align}\label{eq:A3}
\frac{\partial h}{\partial \theta_1}
& =
4\sin\left(\theta_1+\frac{\theta_2}{2}\right)
\Bigg[
16\sin\left(\frac{\theta_1}{2}\right) \sin^2\left(\frac{\theta_2}{2}\right) 
\sin\left(\frac{\theta_1+\theta_2}{2}\right)- 27 \cos\left(\frac{\theta_2}{2}\right)
\Bigg],\\
\label{eq:A4}
\frac{\partial h}{\partial \theta_2}
& =
4\sin\left(\theta_2+\frac{\theta_1}{2}\right)
\Bigg[
16\sin^2\left(\frac{\theta_1}{2}\right) \sin\left(\frac{\theta_2}{2}\right) 
\sin\left(\frac{\theta_1+\theta_2}{2}\right)- 27 \cos\left(\frac{\theta_1}{2}\right)
\Bigg].
\end{align}
Consequently, setting $\nabla h = 0$ leads to four cases. For notational convenience, define
\[
\alpha
=
\arcsin\sqrt[4]{\frac{27}{32}} .
\]
\noindent \textbf{Case 1.}
\[\sin\left(\theta_1+\frac{\theta_2}{2}\right)=\sin\left(\theta_2+\frac{\theta_1}{2}\right)=0.\]
This implies $\theta_1 + \frac{1}{2}\theta_2 \in \pi \Z$ and $\theta_2 + \frac{1}{2}\theta_1 \in \pi \Z$. Solving the resulting systems for solutions in $[0,2\pi)$ yields three points:
\[\bthe=(0,0),\ \left( \frac{2\pi}{3},\frac{2\pi}{3}\right),\ \left( \frac{4\pi}{3},\frac{4\pi}{3}\right).\]
\noindent \textbf{Case 2.}
\[\sin\left(\theta_1+\frac{\theta_2}{2}\right)=16\sin^2\left(\frac{\theta_1}{2}\right) \sin\left(\frac{\theta_2}{2}\right) 
\sin\left(\frac{\theta_1+\theta_2}{2}\right)- 27 \cos\left(\frac{\theta_1}{2}\right)=0.\]
As before, the first condition forces $\theta_1 + \frac{1}{2} \theta_2 \in \pi\Z$. Plugging the various possibilities that this yields into the second condition gives three solutions:
\begin{align*}
\bthe=(\pi, 0),\ \ (2\alpha,2\pi - 4\alpha),\ \ (2\pi-2\alpha,4\alpha). %  \\
% \bthe=\left(2\arcsin\left(\frac{27}{32}\right)^{\frac{1}{4}}, 2\pi-4\arcsin\left(\frac{27}{32}\right)^{\frac{1}{4}}\right),\\
%\bthe=\left(2\pi-2\arcsin\left(\frac{27}{32}\right)^{\frac{1}{4}}, 4\arcsin\left(\frac{27}{32}\right)^{\frac{1}{4}}\right).
\end{align*}
\noindent \textbf{Case 3.}
\[\sin\left(\theta_2+\frac{\theta_1}{2}\right)=16\sin\left(\frac{\theta_1}{2}\right) \sin^2\left(\frac{\theta_2}{2}\right) 
\sin\left(\frac{\theta_1+\theta_2}{2}\right)- 27 \cos\left(\frac{\theta_2}{2}\right)=0.\]
Arguing as in Case~2, there are three solutions:
\begin{align*}
\bthe=(0,\pi),\ \ (2\pi - 4\alpha,2\alpha),\ \ (4\alpha,2\pi-2\alpha). % \\
%\bthe=\left(2\pi-4\arcsin\left(\frac{27}{32}\right)^{\frac{1}{4}}, 2\arcsin\left(\frac{27}{32}\right)^{\frac{1}{4}}\right),\\
%\bthe=\left(4\arcsin\left(\frac{27}{32}\right)^{\frac{1}{4}}, 2\pi-2\arcsin\left(\frac{27}{32}\right)^{\frac{1}{4}}\right).
\end{align*}
\noindent \textbf{Case 4.}
\begin{align}
\label{eq:triGapLength:trigPolyEst:Case4a}
16\sin^2\left(\frac{\theta_1}{2}\right) \sin\left(\frac{\theta_2}{2}\right) 
\sin\left(\frac{\theta_1+\theta_2}{2}\right)- 27 \cos\left(\frac{\theta_1}{2}\right)&=0\\
\label{eq:triGapLength:trigPolyEst:Case4b}
16\sin\left(\frac{\theta_1}{2}\right) \sin^2\left(\frac{\theta_2}{2}\right) 
\sin\left(\frac{\theta_1+\theta_2}{2}\right)- 27 \cos\left(\frac{\theta_2}{2}\right)&=0.
\end{align}
Multiply \eqref{eq:triGapLength:trigPolyEst:Case4a} by $\sin(\theta_2/2)$, multiply \eqref{eq:triGapLength:trigPolyEst:Case4b} by $\sin(\theta_1/2)$, and subtract the results to obtain
\[
\sin\left(\frac{\theta_1-\theta_2}{2}\right)
=
0.
\]
Using this, we see that the solutions are 
\begin{align*}
\bthe=(\pi, \pi),\ \ (2\alpha,2\alpha),\ \ (2\pi - 2\alpha,2\pi - 2\alpha) % \\
%\bthe=\left(2\arcsin\left(\frac{27}{32}\right)^{\frac{1}{4}}, 2\arcsin\left(\frac{27}{32}\right)^{\frac{1}{4}}\right),\\
%\bthe=\left(2\pi-2\arcsin\left(\frac{27}{32}\right)^{\frac{1}{4}}, 2\pi-2\arcsin\left(\frac{27}{32}\right)^{\frac{1}{4}}\right).
\end{align*}
Evaluating $g$ at these points, we find out $\max h(\bthe)=216$ attained at $(0,0)$, 
$\min h(\bthe)= 0$, attained at
\[
\bthe
=
\left( \frac{2\pi}{3},\frac{2\pi}{3}\right),
\;
\left( \frac{4\pi}{3},\frac{4\pi}{3}\right),
\; (\pi,0), \; (0,\pi), \; (\pi,\pi).
\]

Finally, we need to look at critical points of $g$  in the interior of $\T^2 \times [0,54]$. However, this is easy. Any zero of $\nabla g$ must in particular satisfy $\frac{\partial g}{\partial a} = 0$, which forces
\[
1+\cos\theta_1 + \cos\theta_2 + \cos(\theta_1+\theta_2)=0,
\]
which clearly implies $g \geq 0$.
\end{proof}

\begin{proof}[Proof of Theorem~\ref{thm:triExampleGapLength}]
For $\bthe = (\theta_1,\theta_2) \in \T^2$, denote by $H_\lambda(\bthe)$ the Floquet matrix corresponding to $H_\lambda$. Ordering the vertices of the $2\times 2$ fundamental domain as shown in Figure~\ref{fig:tri2x2fundDomain}, we obtain
\[
H_\lambda(\bthe) -(-2+\varepsilon) \bbI
=
\begin{bmatrix}
2+\lambda-\varepsilon & 1+ e^{-i\theta_1} & 1 + e^{-i\theta_2}& 1 + e^{-i(\theta_1+\theta_2)} \\
1 + e^{i\theta_1}& 2+\lambda -\varepsilon & e^{i\theta_1} + e^{-i\theta_2} & 1 + e^{-i\theta_2} \\
1 + e^{i\theta_2} & e^{-i\theta_1} + e^{i\theta_2} & 2+\lambda -\varepsilon & 1 + e^{-i\theta_1} \\
1 + e^{i(\theta_1+\theta_2)}& 1 + e^{i\theta_2} & 1 + e^{i\theta_1} & 2 - \lambda -\varepsilon
\end{bmatrix}
\]

For $\bthe \in \T^2$, $\lambda>0$, and $\varepsilon \in \R$, define
\[
p(\bthe,\lambda,\varepsilon)
=
\det(H_\lambda(\bthe) - (-2+\varepsilon) \bbI).
\]
After some calculations, one observes that
\begin{align*}
p(\bthe,\lambda,\varepsilon)
=
-\lambda^4 & - 4\lambda^3 + X(\bthe) - 4\varepsilon \lambda\left(3 - \frac{\lambda^2}{2}- \cos\theta_1 - \cos\theta_2 - \cos(\theta_1+\theta_2) \right) \\
 & + 4\varepsilon^2 (3+3\lambda - \cos\theta_1 - \cos\theta_2 - \cos(\theta_1+\theta_2)) \\
 & - 2\varepsilon^3(4 + \lambda) + \varepsilon^4,
\end{align*}
where 
\[
X(\bthe) = 
-4\big(\sin\theta_1+\sin\theta_2 - \sin(\theta_1+\theta_2)\big)^2
\]
Clearly $X(\bthe) \leq 0$ for all $\bthe$, so we have
\[
\det(H_\lambda(\bthe) + 2\bbI)
=
p(\bthe,\lambda,0)
\leq
-\lambda^4 - 4\lambda^3 <0
\]
for all $\lambda > 0$; consequently $-2 \notin \sigma(H_\lambda)$ for all $\lambda > 0$, which proves the first claim of the theorem. Introducing $W_1(\lambda,\varepsilon) := -\lambda^4 -4\lambda^3 + 2 \varepsilon\lambda^3 + 12\varepsilon^2\lambda - 2\varepsilon^3(4+\lambda) + \varepsilon^4$, we may rewrite $p$ as
\begin{align} \label{eq:triPrewrittenW1}
p(\bthe,\lambda,\varepsilon)
=
 X(\bthe) - 4\varepsilon(\lambda-\varepsilon)(3-\cos\theta_1-\cos\theta_2-\cos(\theta_1+\theta_2)) +W_1(\lambda,\varepsilon).
 \end{align}
By standard eigenvalue perturbation theory, we know that $|g_\lambda^\pm+2| \leq \lambda$, so we need only concern ourselves with $|\varepsilon| \leq \lambda$. Since $X(\bthe) \leq 0$ for all $\bthe$ and the second term of \eqref{eq:triPrewrittenW1} is nonpositive whenever $0 \leq \varepsilon \leq \lambda$, we arrive at
\[
p(\bthe,\lambda,\varepsilon)
\leq
-\lambda^4 -4\lambda^3 + 2 \varepsilon\lambda^3 + 12\varepsilon^2\lambda - 2\varepsilon^3(4+\lambda) + \varepsilon^4 
= 
 W_1(\lambda,\varepsilon)
\]
for all $\bthe \in \T^2$, all $\lambda > 0$, and all $0 \leq \varepsilon \leq \lambda$. Moreover, we observe that $p(\bm{0},\lambda,\varepsilon) = W_1(\lambda,\varepsilon)$, so this bound is sharp. Factoring $W_1$, we arrive at
\[
W_1(\lambda,\varepsilon)
=
(\lambda - \varepsilon)^2(\varepsilon^2 - 8\varepsilon - \lambda^2 - 4\lambda).
\]
Consequently, we see that $W_1(\lambda,\varepsilon) < 0$ for $\varepsilon \in [0,\lambda)$, which implies that $p(\bthe,\lambda,\varepsilon) < 0$ for all $\bthe \in \T^2$, all $\lambda > 0$, and all $0 \le \varepsilon < \lambda$; consequently, $[-2,-2+\lambda) \cap \sigma(H_\lambda) = \emptyset$, which is to say:
\begin{equation} \label{eq:triGapRightSide}
[-2,-2+\lambda) \subseteq \mathfrak{g}_\lambda.
\end{equation}
On the other hand, $p(\bm{0},\lambda,\lambda) = 0$, so 
\begin{equation} \label{eq:triGapRightSide-2+lambda}
-2+\lambda \in \sigma(H_\lambda(\bm{0})) \subseteq \sigma(H_\lambda)
\end{equation} 
Alternatively, $-2+\lambda \in \sigma(H_\lambda)$ is clear from eigenvalue perturbation theory as soon as one has $[-2,2+\lambda)\cap \sigma(H_\lambda) = \emptyset$.

Now, for $- \lambda \leq \varepsilon \leq 0$, we have to be more careful with the term
\[
q(\bthe,\lambda,\varepsilon)
:=
-4\varepsilon(\lambda - \varepsilon)(3- \cos\theta_1 - \cos\theta_2 - \cos(\theta_1+\theta_2)),
\]
as $q$ can be positive when $-\lambda < \varepsilon < 0$. Naively, one can bound
\[
3-\cos\theta_1-\cos\theta_2 - \cos(\theta_1+\theta_2)
\leq
\frac{9}{2},
\]
which leads to the upper bound of $X(\bthe) + q(\bthe,\lambda,\varepsilon) \leq -18\varepsilon(\lambda - \varepsilon)$. However, the maximum of $q$ occurs at the global \emph{minimum} of $X$, so we can do better. Indeed, for $\lambda > 0$ small and $-\lambda \leq \varepsilon \leq 0$, we have
\begin{equation} \label{eq:triGapRefinedXQbound}
X(\bthe) + q(\bthe,\lambda,\varepsilon) \leq - 16\varepsilon(\lambda-\varepsilon).
\end{equation}
In particular, by Lemma~\ref{lem:triGapLength:trigPolyEst}, the bound in \eqref{eq:triGapRefinedXQbound} holds for all $\varepsilon$ such that $-\lambda \leq \varepsilon\leq 0$ as long as $8 \lambda^2 < 54$, i.e.\ $0 < \lambda < \frac{3\sqrt{3}}{2}$. This then leads us to
\begin{align*}
p(\bthe,\lambda,\varepsilon)
\leq
W_2(\lambda,\varepsilon)
:&=
-\lambda^4 -4\lambda^3 + 2 \varepsilon\lambda^3 + 12\varepsilon^2\lambda - 2\varepsilon^3(4+\lambda) + \varepsilon^4 - 16\varepsilon(\lambda-\varepsilon) \\
& =
W_1(\lambda,\varepsilon) - 16\varepsilon(\lambda-\varepsilon)
\end{align*}
for $\lambda > 0$ small and $-\lambda \leq \varepsilon \leq 0$. Factoring $W_2$ yields
\[
p(\bthe,\lambda,\varepsilon)
\leq
W_2(\lambda,\varepsilon)
=
(\varepsilon - \lambda)(\varepsilon - \lambda- 4)(\varepsilon^2 - 4\varepsilon - \lambda^2)
\]
for $\lambda > 0$ small and $-\lambda \leq \varepsilon \leq 0$. It is straightforward to find the roots of $W_2$ and to observe that $W_2(\lambda,\varepsilon) < 0$ when
\[
2 - \sqrt{4+\lambda^2}
<
\varepsilon
\leq 
0.\]
As a result, this implies $p(\bthe,\lambda,\varepsilon) < 0$ for all $\bthe$, all $\lambda > 0$ small, and all $\varepsilon \in (2-\sqrt{4+\lambda^2},0]$, which in turn yields
\begin{equation} \label{eq:triGapLeftSide}
(-\sqrt{4+\lambda^2},-2] \subseteq \mathfrak{g}_\lambda.
\end{equation}
On the other hand, 
\[
p\left((\pi,\pi),\lambda,2 - \sqrt{4+\lambda^2}\right) 
=
 W_2\left(\lambda,2 - \sqrt{4+\lambda^2}\right)
=
0,
\] 
which leads us to conclude 
\begin{equation} \label{eq:triGapLeftSide-sqrt4+lambda2}
-\sqrt{4+\lambda^2} \in \sigma(H_\lambda(\pi,\pi)) \subseteq \sigma(H_\lambda).
\end{equation}
Putting together \eqref{eq:triGapRightSide}, \eqref{eq:triGapRightSide-2+lambda}, \eqref{eq:triGapLeftSide}, and \eqref{eq:triGapLeftSide-sqrt4+lambda2}, we obtain
\[
\mathfrak{g}_\lambda
=
\left(-\sqrt{4+\lambda^2},-2+\lambda \right)
\]
for small $\lambda$, as promised.

\end{proof}

The effort involved in proving Lemma~\ref{lem:triGapLength:trigPolyEst} in order to improve the constant ``18'' to ``16'' is nontrivial, but worthwhile. In particular, this is exactly what enables the exact factorization of $W_2$ and hence the ability to exactly compute the gap edges.

\section{Hexagonal Laplacian} \label{sec:hex}

We now continue with the Laplacian on the hexagonal lattice. Let $\Gamma_\hex = (\CV_\hex,\CE_\hex)$ and
\[
\bm{b}_\pm
=
\frac{1}{2} \begin{bmatrix}  3 \\ \pm \sqrt{3} \end{bmatrix}
\]
be as in the introduction, let periods $p_1,p_2 \in \Z_+$ be given, and view $H = \Delta_\hex$ as a $(p_1,p_2)$-periodic operator.
For this setting, there are two vertices of $\CV_\hex$ in $\set{s\bm{b}_+ + t \bm{b}_- : 0 \le s, t < 1}$, so our Floquet operator $H(\bthe)$ will be $P \times P$ with $P = 2p_1p_2$. As usual, define $\Lambda = \big( [0,p_1)\times[0,p_2)\big) \cap \Z^2$, denote the eigenvalues of $H(\bthe)$ by
\[
E_1^\Lambda(\bthe)
\leq
\cdots
\leq E_P^\Lambda(\bthe),
\]
and let $F_k^\Lambda$ for $1 \le k \le P$ denote the bands of the spectrum. Our main theorem in this section is the following result.

\begin{theorem} \label{thm:hexmain}
Let $p_1,p_2 \in \Z_+$ be given.
\begin{enumerate}
\item Every $E \in (-3,3) \setminus \set{-1,0,1}$ belongs to $\mathrm{int}(F_j)$ for some $1 \le j \le P$.
\item If at least one of $p_1$ or $p_2$ is odd, then $-1 \in \mathrm{int}(F_k)$ and $+1 \in \mathrm{int}(F_\ell)$ for some $1 \le k \le \ell \le P$
\end{enumerate}
\end{theorem}

\begin{proof}[Proof of Theorem~\ref{t:bsc:hex}]
This follows immediately from Theorem~\ref{thm:hexmain}.
\end{proof}

\begin{proof}[Proof of Theorem~\ref{thm:hexmain}] 
From \eqref{eq:hexDecomp}, we have
\[
\Delta_\hex
=
\begin{bmatrix}
0 & S_1^* + S_2^* + \bbI \\ S_1 + S_2 + \bbI & 0 \end{bmatrix},
\]
where $S_j: \ell^2(\Z^2) \to \ell^2(\Z^2)$ denote the shifts
\[
[S_1 \psi]_{n,m} = \psi_{n+1,m},
\quad
[S_2 \psi]_{n,m} = \psi_{n,m+1}.
\]
It is easy to see that
\[
S_1 + S_1^* + S_2 + S_2^* + S_1S_2^* + S_1^*S_2 = \Delta_\tri,
\]
the triangular Laplacian. Thus, a simple calculation shows that
\begin{equation} \label{eq:hexSquareTri}
[\Delta_\hex^2 \Psi]_{\bm{n}}
\begin{bmatrix} [\Delta_\tri \psi^+]_{\bm{n}} + 3\psi_{\bm{n}}^+ \\ [\Delta_\tri \psi^-]_{\bm{n}} + 3\psi_{\bm{n}}^- \end{bmatrix}
\quad
\text{for } \Psi = \begin{bmatrix} \psi^+ \\ \psi^- \end{bmatrix} \in \ell^2(\Z^2,\C^2).
\end{equation}
This calculation extends to the Floquet matrices, so we see that for each $1 \le k \le P$, the bands of $H = \Delta_\hex$ obey
\[
F^\Lambda_{k,\hex}
=
-F^\Lambda_{P+1-k,\hex}
\]
and 
\begin{equation} \label{eq:hexTriBandRel}
F_{k,\hex}^\Lambda
=
\begin{cases}
\sqrt{F^\Lambda_{k - \frac{P}{2},\tri}+3} & \frac{P}{2} < k \leq P \\[3mm]
-\sqrt{F^\Lambda_{\frac{P}{2}+1-k,\tri}+3} & 1 \leq k \leq \frac{P}{2}
\end{cases}
\end{equation}
From this, we deduce that $E \in (-3,3)$ lies in the interior of some $F_{k,\hex}$ if and only if $E^2-3$ lies in the interior of some $F_{\ell,\tri}$. For $E \in (-3,3) \setminus \set{-1,0,1}$, $E^2-3 \in (-3,6) \setminus \{-2\}$, while $(\pm1)^2-3 = -2$. Thus, the conclusions of the theorem follow from Theorem~\ref{thm:trimain}.

\end{proof}

\subsection{\boldmath Opening gaps at $0$ and $\pm 1$}

Define the $(1,1)$-periodic potential $Q_1$ on $\CV_\hex$ by $Q_1(\bm{0}) = 1$ and $Q_1(\bm{a}_1) = -1$, that is,
\[
Q_1(n \bm{b}_+ + m \bm{b}_-) = 1,
\quad
Q_1(\bm{a}_1  + n \bm{b}_+ + m \bm{b}_-) = -1,
\quad n,m \in \Z.
\]
After identifying $\ell^2(\CV_\hex)$ with $\ell^2(\Z^2,\C^2)$ in the usual way, we get (as an operator) $[Q_1 \Psi]_{\bm{n}}= Z\Psi_{\bm{n}}$, where
\[
Z=\begin{bmatrix}
1 & 0 \\ 0 & -1
\end{bmatrix}.
\]

From the calculations $Z U = U = -UZ$ and $Z  L= -L = -L Z$, we deduce that $Q_1 \Delta_\hex + \Delta_\hex Q_1 = 0$, and hence
\[
(\Delta_\hex + \lambda Q_1)^2
=
\Delta_\hex^2 + \lambda^2 
\geq
\lambda^2.
\]
Consequently, $(-\lambda,\lambda) \cap \sigma(\Delta_\hex + \lambda Q_1) = \emptyset$ and there is a gap at zero. In particular, the gap is precisely $(-\lambda,\lambda)$, and so opens linearly at the maximal possible rate.
\medskip

Let us consider the $(2,2)$-periodic case. We parameterize our potential as $(q_1,\ldots,q_8) \in \R^8$ as shown in Figure~\ref{fig:hex22period}.
\begin{figure*}[t]

\begin{tikzpicture}[yscale=1]
\draw [-,line width = .1cm,color=black] (0,0) -- (1,{sqrt(3)});
\draw [-,line width = .1cm,color=red] (0,{2*sqrt(3)}) -- (1,{sqrt(3)});
\draw [-,line width = .1cm,color=red] (0,{2*sqrt(3)}) -- (1,{3*sqrt(3)});
\draw [-,line width = .1cm] (0,{4*sqrt(3)}) -- (1,{3*sqrt(3)});
\draw [-,line width = .1cm,color=red] (1,{sqrt(3)}) -- (3,{sqrt(3)});
\draw [-,line width = .1cm,color=red] (1,{3*sqrt(3)}) -- (3,{3*sqrt(3)});
\draw [-,line width = .1cm] (4,{2*sqrt(3)}) -- (6,{2*sqrt(3)});
\draw [-,line width = .1cm] (4,0) -- (3,{sqrt(3)});
\draw [-,line width = .1cm,color=red] (4,{2*sqrt(3)}) -- (3,{sqrt(3)});
\draw [-,line width = .1cm,color=red] (4,{2*sqrt(3)}) -- (3,{3*sqrt(3)});
\draw [-,line width = .1cm] (4,{4*sqrt(3)}) -- (3,{3*sqrt(3)});
\draw [-,line width = .1cm] (4,0) -- (6,0);
\draw [-,line width = .1cm,color=red] (4,{2*sqrt(3)}) -- (6,{2*sqrt(3)});
\draw [-,line width = .1cm] (4,{4*sqrt(3)}) -- (6,{4*sqrt(3)});
\draw [-,line width = .1cm] (6,0) -- (7,{sqrt(3)});
\draw [-,line width = .1cm] (6,{2*sqrt(3)}) -- (7,{sqrt(3)});
\draw [-,line width = .1cm] (6,{2*sqrt(3)}) -- (7,{3*sqrt(3)});
\draw [-,line width = .1cm] (6,{4*sqrt(3)}) -- (7,{3*sqrt(3)});
\draw [-,line width = .1cm] (-2,0) -- (0,0);
\draw [-,line width = .1cm,color=red] (-2,{2*sqrt(3)}) -- (0,{2*sqrt(3)});
\draw [-,line width = .1cm] (-2,{4*sqrt(3)}) -- (0,{4*sqrt(3)});
\draw [-,line width = .1cm] (-2,0) -- (-3,{sqrt(3)});
\draw [-,line width = .1cm] (-2,{2*sqrt(3)}) -- (-3,{sqrt(3)});
\draw [-,line width = .1cm] (-2,{2*sqrt(3)}) -- (-3,{3*sqrt(3)});
\draw [-,line width = .1cm] (-2,{4*sqrt(3)}) -- (-3,{3*sqrt(3)});

\filldraw[color=black, fill=black](0,0) circle (.2);
\filldraw[color=black, fill=black](4,0) circle (.2);
\filldraw[color=black, fill=black](6,0) circle (.2);
\filldraw[color=red, fill=red](1,{sqrt(3)}) circle (.2);
\filldraw[color=red, fill=red](3,{sqrt(3)}) circle (.2);
\filldraw[color=black, fill=black](7,{sqrt(3)}) circle (.2);
\filldraw[color=red, fill=red](0,{2*sqrt(3)}) circle (.2);
\filldraw[color=red, fill=red](4,{2*sqrt(3)}) circle (.2);
\filldraw[color=red, fill=red](6,{2*sqrt(3)}) circle (.2);
\filldraw[color=red, fill=red](1,{3*sqrt(3)}) circle (.2);
\filldraw[color=red, fill=red](3,{3*sqrt(3)}) circle (.2);
\filldraw[color=black, fill=black](7,{3*sqrt(3)}) circle (.2);
\filldraw[color=black, fill=black](0,{4*sqrt(3)}) circle (.2);
\filldraw[color=black, fill=black](4,{4*sqrt(3)}) circle (.2);
\filldraw[color=black, fill=black](6,{4*sqrt(3)}) circle (.2);
\filldraw[color=black, fill=black](-2,{0*sqrt(3)}) circle (.2);
\filldraw[color=red, fill=red](-2,{2*sqrt(3)}) circle (.2);
\filldraw[color=black, fill=black](-2,{4*sqrt(3)}) circle (.2);
\filldraw[color=black, fill=black](-3,{1*sqrt(3)}) circle (.2);
\filldraw[color=black, fill=black](-3,{3*sqrt(3)}) circle (.2);
\node at (-2.5,{2*sqrt(3)}){\hot{$q_1$}};
\node at (.5,{2*sqrt(3)}){\hot{$q_2$}};
\node at (.5,{3*sqrt(3)}){\hot{$q_3$}};
\node at (3.5,{3*sqrt(3)}){\hot{$q_4$}};
\node at (.5,{sqrt(3)}){\hot{$q_5$}};
\node at (3.5,{sqrt(3)}){\hot{$q_6$}};
\node at (3.5,{2*sqrt(3)}){\hot{$q_7$}};
\node at (6.5,{2*sqrt(3)}){\hot{$q_8$}};
\end{tikzpicture}
\caption{A portion of the hexagonal lattice. A fundamental domain for a $(2,2)$-periodic potential is highlighted in red.}\label{fig:hex22period}
\end{figure*}
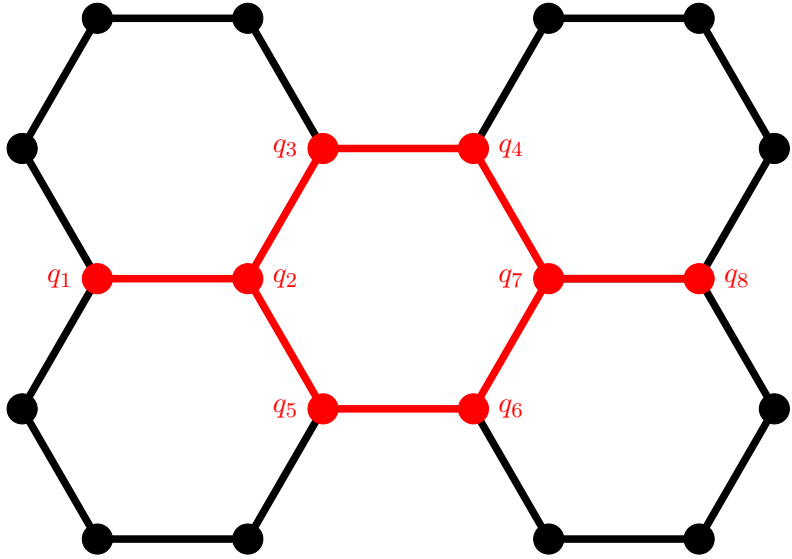

We now turn to the construction of a potential that opens gaps at $0$, $1$, and $-1$ simultaneously. We show that it opens gaps linearly at zero, quadratically at $\pm1$. Later on, we will show that one cannot open gaps linearly at $\pm 1$ on both sides.

\begin{theorem}\label{thm:hexQ}
Order the vertices of a $2\times 2$ fundamental cell of the hexagonal lattice as shown in Fig.~\ref{fig:hex22period}, define a $(2,2)$-periodic potential $Q$ by
\[
(q_1,\ldots,q_8)
=
(1,-1,1,2,-2,-1,1,-1),
\]
and denote $H_\lambda = \Delta_\hex + \lambda Q$. Then, for $ |\lambda| > 0$ sufficiently small, $\sigma(H_\lambda)$ consists of four connected components. Moreover, if $\mathfrak{g}_{E,\lambda} = (g_{E,\lambda}^{-}, g_{E,\lambda}^{+})$ denote the gaps of $\sigma(H_{\lambda})$ that open at $E=0,\pm 1$, one has
\[
 \Big(\pm 1-\frac{\lambda^2}{20}, \pm 1+\frac{\lambda^2}{20}\Big)\subset  \mathfrak{g}_{\pm 1,\lambda} {\subseteq \Big(\pm 1 - \frac{1}{2}\lambda^2,\pm1 + \frac{1}{2}\lambda^2 \Big)},\]
and
\[ \ \ \ \ \Big(-\frac{\lambda}{{5}}, \frac{\lambda}{{5}}\Big)\subset  \mathfrak{g}_{0,\lambda} \subset \Big(-\frac{\lambda}{4}, \frac{\lambda}{4}\Big)\]
for all $|\lambda| > 0$ sufficiently small.
\end{theorem}

We point that we do not carefully optimize the constants; it is possible to get better constants than $1/20$, ${1/2}$, ${1/5}$, and $1/4$.

\begin{proof}
For $\bthe = (\theta_1,\theta_2) \in \T^2$, let $H_\lambda(\bthe)$ denote the Floquet matrix corresponding to $H_\lambda$. Ordering the vertices of the fundamental domain as in Figure~\ref{fig:hex22period}, we obtain:
\begin{align}\label{eq:hexH}
H_\lambda(\bthe)
=
\begin{bmatrix}
\lambda & 1 & 0 & e^{-i\theta_1} & 0 & e^{-i\theta_2} & 0 & 0 \\
1 & -\lambda & 1 & 0 & 1 & 0 &0&0 \\
0 & 1 & \lambda & 1 & 0 & 0 & 0 & e^{-i\theta_2} \\
e^{i\theta_1} & 0 & 1 & 2\lambda & 0 & 0 & 1 & 0 \\
0 & 1 & 0 & 0 & - 2\lambda & 1 & 0 & e^{-i\theta_1} \\
e^{i\theta_2} & 0 & 0 & 0 & 1 & - \lambda & 1 & 0  \\
0 & 0 &0 & 1 & 0 & 1 & \lambda & 1 \\
0 & 0 & e^{i\theta_2} & 0 & e^{i\theta_1} & 0 & 1 & -\lambda
\end{bmatrix}
\end{align}

First, let us consider the gaps at $E=\pm 1$.
Calculations yield
\begin{align}\label{eq:hexsumXZ}
\det\big(H_{\lambda}(\bthe)- (\pm 1+ s\lambda^2) \bbI \big)
= 
X_0^\pm(\bthe) + X_4^\pm(\bthe,s) \lambda^4 + X_6^\pm(\bthe,s) \lambda^6 + O(\lambda^8),%\\
%\det(H_{\lambda}(\bthe)- (-1+ s\lambda^2) \bbI )=&Z_0(\bthe)+Z_4(\bthe,s) \lambda^4+Z_6(\bthe,s) \lambda^6+O(\lambda^8),
\end{align}
in which
\begin{align*}
X_0^\pm(\bthe)&=
-4 (-\sin(\theta_1) + \sin(\theta_1-\theta_2) + \sin(\theta_2))^2\\
X_4^\pm(\bthe,s)
&=
8(s\pm1)(2s\mp 1) (3 -\cos(\theta_1)-\cos(\theta_1-\theta_2)-\cos(\theta_2))\\
X_6^\pm(\bthe,s)
&=
-1 \mp 12 s + 72 s^2 \mp 16 s^3-4 s^2 (\pm 4 s+1)(\cos(\theta_1) + \cos(\theta_1-\theta_2) + \cos(\theta_2))%\\
%Z_4(\bthe,s)&=8(s-1)(2s+1) (3-\cos(\theta_1)-\cos(\theta_1-\theta_2)-\cos(\theta_2))\\
%Z_6(\bthe,s)&=-1 + 12 s + 72 s^2 + 16 s^3+4 s^2 (4s-1)(\cos(\theta_1) + \cos(\theta_1-\theta_2) + \cos(\theta_2))
\end{align*}
It is clear that 
\begin{align}\label{eq:hexX0Z0}
X_0^\pm(\bthe)\leq 0 \quad \text{for all }\bthe \in \T^2.
\end{align}
Since $\cos(\theta_1) + \cos(\theta_1-\theta_2) + \cos(\theta_2)\leq 3$, we also have 
\begin{align}\label{eq:hexX4Z4}
X_4^\pm(\bthe,s)\leq 0 \quad \text{for all } \bthe \in \T^2, \; |s|\leq 1/2.
\end{align}
We also have for $|s|\leq 1/4$,
\[
X_6^+(\bthe,s)\leq -1-12 s+72 s^2-16 s^3+12 s^2 (4s+1)=:T(s),
\]
and 
\[
X_6^-(\bthe,s)= X_6^+(\bthe,-s) \leq T(-s).
\]
One easily checks that $T(s)$ is decreasing on $[-0.05, 0.05]$, and
\[T(-0.05)=-0.194.\]
Hence for $|s|\leq 0.05$,
\begin{align}\label{eq:hexX6Z6}
X_6^\pm(\bthe,s)\leq -0.194.
\end{align}
Combining \eqref{eq:hexX0Z0}, \eqref{eq:hexX4Z4}, and \eqref{eq:hexX6Z6}, we obtain that for $|\lambda|>0$ sufficiently small, and $|s| \leq 1/20$,
\[\det(H_{\lambda}(\bthe)- (\pm 1+ s\lambda^2) \bbI)\leq -0.1\lambda^6<0.\]
This proves the claimed lower bound on the gaps at $\pm 1$.

On the other hand, let us note that $X_0^\pm(0,0) =X_0^\pm(\pi,\pi) = 0$, while
\begin{align*}
\begin{cases}
X_4^+(\bthe, 0.5)= 0 \quad \text{and}\quad X_6^+((\pi,\pi),0.5)=12,\\
X_4^+((0,0),s)=0 \quad \text{and}\quad X_6^+((0,0),-0.5)=28,\\
X_4^-((0,0),s)=0\quad \text{and} \quad X_6^-((0,0),0.5)=28,\\
X_4^-(\bthe, -0.5)=0 \quad \text{and}\quad X_6^-((\pi, \pi), -0.5)=12.
\end{cases}
\end{align*}
Thus for small $\lambda>0$, we have
\begin{align*}
\begin{cases}
\det(H_\lambda(\pi,\pi)-(1+0.5\lambda^2)\bbI) > 0,\\
\det(H_\lambda(0,0)-(1-0.5\lambda^2)\bbI)>0,\\
\det(H_\lambda(0,0)-(-1+0.5\lambda^2)\bbI)>0,\\
\det(H_\lambda(\pi,\pi)-(-1-0.5\lambda^2)\bbI)>0.
\end{cases}
\end{align*}
We also easily check that 
\[X_0^\pm\left(\frac{\pi}{2},\pi\right)=-16,\]
which implies that for small $\lambda>0$, we have
\[\det\left(H_\lambda\left(\frac{\pi}{2},\pi \right)-(\pm 1\pm 0.5 \lambda^2 \bbI) \right)<0.\]
We therefore conclude that 
\[\pm 1 + 0.5 \lambda^2 \in \sigma(H_\lambda) \text{ and } \pm 1 - 0.5 \lambda^2 \in \sigma(H_\lambda),\]
which proves the upper bounds on the gaps at $\pm 1$.

Now let us consider the gap at $E=0$. After calculations, we have
\begin{align}\label{eq:hexsumY}
\det(H_{\lambda}(\bthe)- s\lambda \bbI )=Y_0(\bthe)+Y_2(\bthe,s)\lambda^2+Y_4(\bthe,s)+O(\lambda^6),
\end{align}
where 
\begin{align*}
Y_0(\bthe)
& =
15 + 2\cos(2\theta_1) - 4 \cos(\theta_1 - 2\theta_2) + 2 \cos(2\theta_1-2\theta_2) 
- 4 \cos(2\theta_1 - \theta_2) \\ & \qquad + 2\cos(2\theta_2) - 4 \cos(\theta_1 + \theta_2),
\end{align*}
\[
Y_2(\bthe, s)=2[5-26 s^2+(2+4 s^2)(\cos(\theta_1)+\cos(\theta_1-\theta_2)+\cos(\theta_2))],
\]
and 
\[
Y_4(\bthe,s)=(1-s^2)[-3 - 42 s^2 + 4 (2 + s^2) (\cos(\theta_1)+\cos(\theta_1-\theta_2)+\cos(\theta_2))]
\]
We claim that
\begin{equation}
\label{eq:hexY0geq0}
Y_0(\bthe)\geq 0
\quad
\text{for all } \bthe \in \T^2.
\end{equation} 
Let us see how to use \eqref{eq:hexY0geq0} to prove the claimed gap at zero and defer the proof of \eqref{eq:hexY0geq0} for a moment. Using
\[\cos(\theta_1)+\cos(\theta_1-\theta_2)+\cos(\theta_2)\in \left[-\frac{3}{2}, 3\right],\]
we obtain that for $|s|<1/5$
\begin{align}\label{eq:hexY2}
Y_2(\bthe, s)\geq 2(5-26s^2-3(1+2s^2))=4(1-16 s^2)>\frac{36}{25}.
\end{align}
Combining \eqref{eq:hexsumY} with \eqref{eq:hexY2}, we obtain that for $|\lambda|>0$ sufficiently small
\[\det(H_{\lambda}(\bthe)- s\lambda\, \bbI )>\lambda^2.\]
This proves the claimed lower bound of the gap at $0$, modulo the claim that $Y_0(\bthe) \geq 0$ for all $\bthe \in \T^2$. 

To prove the upper bound, we compute
\begin{align*}
\begin{cases}
Y_0\left(\frac{2\pi}{3}, \frac{4\pi}{3}\right)=0,\\
Y_2\left(\left(\frac{2\pi}{3}, \frac{4\pi}{3}\right), s\right)=4(1-16 s^2),\\
Y_4\left(\left(\frac{2\pi}{3}, \frac{4\pi}{3}\right), s\right)=3(s^2-1)(16s^2+5),
\end{cases}
\end{align*}
which implies that for small $\lambda>0$,
\[\det\left(H_{\lambda}\left(\frac{2\pi}{3}, \frac{4\pi}{3}\right)\pm 0.25 \lambda\, \bbI \right)<0.\]
We also compute that $Y_0(0,0)=9$, which shows for small $\lambda>0$, 
\[\det(H_{\lambda}(0,0)\pm 0.25 \lambda\, \bbI)>0.\]
Thus we conclude that
\[\pm 0.25 \lambda \in \sigma(H_\lambda),\]
which proves the claimed upper bound of the gap at $0$.

To complete the argument, all that remains is to show $Y_0(\bthe)\geq 0$ for all $\bthe\in \T^2$.
To that end, introduce two auxiliary variables
\[
z := \cos\left(\frac{\theta_1 - \theta_2}{2} \right),
\quad
w := \cos\left(\frac{\theta_1 + \theta_2}{2} \right),
\]
and write $g(z,w)$ to mean $Y_0(\bthe)$ in the variables $z$ and $w$. Thus, to optimize $Y_0(\bthe)$ on $\T^2$, it suffices to optimize $g(z,w)$ on the square $[-1,1]^2$. To execute this change of variables, first note the following simple consequences of standard identities:
\begin{align*}
\cos(2\theta_1) +  \cos(2\theta_2)
& =
2(2z^2-1)(2w^2-1),\\
\cos(2\theta_1 - 2\theta_2)
& =
2(2z^2-1)^2-1,\\
\cos(\theta_1 + \theta_2)
& =
2w^2 -1,\\
\cos(\theta_1 - 2\theta_2) + \cos(2\theta_1 - \theta_2)
& =
2 zw(4z^2 - 3).
\end{align*}
Putting all this together,
\begin{align*}
g(z,w)
& =
15+ 4(2z^2-1)(2w^2-1) - 8zw(4z^2-3)+ 2(2(2z^2-1)^2-1) - 4(2w^2-1).
\end{align*}
It is easy to check that $g \geq 0$ holds on the boundary; concretely,
\begin{align*}
g(\pm 1,w)
& =
15 + 4(2w^2-1) \mp 8w + 2 - 4(2w^2-1) \\
& =
17 \mp 8w \\
& \geq
17 - 9 \\
& > 0.
\end{align*}
and
\begin{align*}
g(z,\pm 1)
& =
15 + 4(2z^2-1) \mp 8z(4z^2-3) + (16z^4-16z^2+2) - 4 \\
& =
16 z^4 \mp 32 z^3 - 8z^2 \pm 24z + 9 \\
& =
(3 \pm 4z - 4z^2)^2 \\
& \geq 0.
\end{align*}
So, we now seek zeros of $\nabla g$ for $|z| < 1$ and $|w| < 1$.
One easily computes $\partial_z g$ and $\partial_w g$:
\begin{align*}
\partial_z g
& = 
8(w-2z)(3+4z(w-z)) \\
\partial_w g
& =
8(3z - 4z^3 + 4w(z^2-1)).
\end{align*}
Setting $\partial_w g = 0$ yields
\begin{equation} \label{eq:Y0critical:wval}
w
=
\frac{4z^3 - 3z}{4(z^2-1)}.
\end{equation}
Since we are working on the interior of $[-1,1]^2$, $z \neq \pm 1$ and the denominator does not vanish. Substituting this expression for $w$ into $\partial_z g$ and simplifying, we get
\[
\partial_zg \left(z,\frac{4z^3 - 3z}{z^2-1} \right)
=
2z \left( \frac{1}{(z^2-1)^2} - 16 \right).
\]
Setting this equal to zero, we obtain three values of $z$ with $|z| <1$: $0$ and $\pm \sqrt{3}/2$. Inserting these $z$ values into \eqref{eq:Y0critical:wval}, the corresponding $w$ values are all readily seen to be zero. Plugging in the three critical points $(0,0)$ and $(\pm \sqrt{3}/2,0)$ into $g$ yields 25 and 16, respectively, which concludes the proof that $g \geq 0$ and hence 
\[
Y_0(\bthe) \geq 0
\]
for all $\bthe \in \T^2$, proving \eqref{eq:hexY0geq0}.
\end{proof}

Next, we show that for any $(2,2)$-periodic potential, it is impossible that it opens linear order gaps on both sides of $E=\pm 1$ simultaneously.
\begin{theorem}\label{thm:E=pm1linear}
For any $(2,2)$-periodic potential $Q$ and any constant $c>0$, the following holds for all sufficiently small $\lambda>0$:
\[ \left( (-1-c\lambda, -1+c\lambda) \cup (1-c\lambda, 1+c\lambda) \right) \cap \sigma(H_{\lambda})\neq \emptyset\]
\end{theorem}
\begin{proof}
Let $(q_1,q_2,\ldots,q_8)$ be the potential on a $2\times 2$ fundamental cell, as shown in Fig.~\ref{fig:hex22period}. 
The corresponding Floquet matrix $H_{\lambda}(\bthe)$ is 
\[
H_\lambda(\bthe)
=
\begin{bmatrix}
\lambda q_1 & 1 & 0 & e^{-i\theta_1} & 0 & e^{-i\theta_2} & 0 & 0 \\
1 & \lambda q_2& 1 & 0 & 1 & 0 &0&0 \\
0 & 1 & \lambda q_3& 1 & 0 & 0 & 0 & e^{-i\theta_2} \\
e^{i\theta_1} & 0 & 1 & \lambda q_4& 0 & 0 & 1 & 0 \\
0 & 1 & 0 & 0 & \lambda q_5& 1 & 0 & e^{-i\theta_1} \\
e^{i\theta_2} & 0 & 0 & 0 & 1 & \lambda q_6& 1 & 0  \\
0 & 0 &0 & 1 & 0 & 1 & \lambda q_7& 1 \\
0 & 0 & e^{i\theta_2} & 0 & e^{i\theta_1} & 0 & 1 & \lambda q_8
\end{bmatrix}
\]
For $0<|s|<c$, let us consider 
\[\det\big(H_\lambda(\bthe)-(\pm 1+s\lambda)\bbI \big)
=
\sum_{k=0}^8 X_k^\pm (\bthe, s) \lambda^k.\]
%and
%\[\det(H_\lambda(\bthe)-(1+s\lambda)\bbI )=\sum_{k=0}^8 Y_k(\bthe, s) \lambda^k.\]
After a calculation, we obtain
\begin{align}\label{eq:X0Y00}
%X_0((0,0), s)=Y_0((0,0), s)=X_1((0,0), s)=Y_1((0,0), s)=X_2((0,0), s)=Y_2((0,0), s)=0,
X_0^\pm(\bm{0},s)
=
X_1^\pm(\bm{0},s)
=
X_2^\pm(\bm{0},s)
=0\quad
\text{ for all }s
\end{align}
and
\begin{align}\label{eq:X3Y30}
&X_3^+(\bm{0}, s)=-X_3^-(\bm{0}, s) = a_0 + a_2 s^2 + 64s^3,
\end{align}
where
\begin{align*}
a_0 & =
-2[(q_1 + q_2 + q_7 + q_8) (q_4 + q_5) (q_3 + q_6) + (q_1 + q_8) (q_2 + q_7) (q_3 + q_4 + q_5 + q_6)] \\
  & \quad +8[(q_1+q_8)(q_2+q_7)+(q_3+q_6)(q_4+q_5)+(q_1+q_2+q_7+q_8)(q_3+q_4+q_5+q_6)]\\
a_2  &= -24 \sum_{k=1}^8 q_k .
\end{align*}
By \eqref{eq:X3Y30}, we have 
\[
X_3^+(\bm{0}, s_0) = -X_3^-(\bm{0}, s_0)\neq 0
\]
for some $s_0$ such that $0<|s_0|<c$.
Without loss of generality, we assume
\[X_3^+(\bm{0}, s_0)>0>X_3^-(\bm{0},s_0).\]
Combining this with \eqref{eq:X0Y00}, we obtain
\begin{equation} \label{eq:det00>0}
\det(H_\lambda(\bm{0})-(1+s_0\lambda) \bbI)>0
\end{equation}
for small $\lambda > 0$. We also have
\begin{align}\label{eq:X0Y0Pi/4}
X_0^\pm((\pi/4, 3\pi/4), s_0)=-4.
\end{align}
In particular, \eqref{eq:X0Y0Pi/4} implies that 
\begin{align}\label{eq:detPi/4<0}
\det(H_\lambda(\pi/4,3\pi/4))-(1+s_0\lambda) \bbI)<0
\end{align}
for all $\lambda \geq 0$ small.

Combining \eqref{eq:detPi/4<0} with \eqref{eq:det00>0}, for any sufficiently small $\lambda > 0$, there exists $\bthe$ such that
\[\det(H_{\lambda}(\bthe)-(1+s_0\lambda) \bbI)=0.\]
Hence 
\[(1-c\lambda, 1+c\lambda)\cap \sigma(H_{\lambda})\neq \emptyset\]
as claimed.
\end{proof}

\section{Square Laplacian with Next-Nearest Neighbor Interactions}\label{sec:nnn}
We now turn our attention to the EHM lattice, whose Laplacian is given by
\begin{align*}
[\Delta_\sqn u]_{n,m}
& =
u_{n-1,m} + u_{n+1,m} + u_{n,m-1} + u_{n,m+1} + u_{n-1,m+1} + u_{n-1,m+1} + u_{n+1,m-1} + u_{n+1,m+1}\\
& =
[\Delta_{\rm sq}u]_{n,m} + u_{n-1,m-1} + u_{n-1,m+1} + u_{n+1,m-1} + u_{n+1,m+1} \\
& =
[\Delta_{\tri}u]_{n,m} + u_{n-1,m-1} + u_{n+1,m+1}.
\end{align*}
Now, given $p_1,p_2 \in \Z_+$, we define $P = p_1p_2$ and $\Lambda = \Z^2 \cap \big([0,p_1) \times [0,p_2)\big)$ as before and view $\Delta_\sqn$ as a $(p_1,p_2)$-periodic operator and perform the Floquet decomposition. For $\bthe = (\theta_1,\theta_2) \in \R^2$, it is straightforward to check that
\[
\sigma(H(\bthe))
=
\set{e_{\bm{\ell}}(\bthe) : \bm{\ell}\in \Lambda },
\]
where $\bm{\ell}=(\ell_1,\ell_2)$ and
\begin{align*}
e_{\bm{\ell}}(\bthe)
=
2\cos\left( \frac{\theta_1+2\pi \ell_1}{p_1}\right) 
+ 2\cos\left(\frac{\theta_2+2\pi \ell_2}{p_2}\right) 
+ &2\cos\left(\frac{\theta_1 + 2\pi \ell_1}{p_1}  - \frac{\theta_2 + 2\pi \ell_2}{p_2}\right)\\
+ &2\cos\left(\frac{\theta_1 + 2\pi \ell_1}{p_1}  + \frac{\theta_2 + 2\pi \ell_2}{p_2}\right).
\end{align*}
\begin{comment}
We now consider this as a $(p_1,p_2)$-periodic operator, and hence we seek (generalized) eigenfunctions $u$ with $u_{n+p_1,m} \equiv e^{2\pi i \theta_1} u_{n,m}$ and $u_{n,m+p_2} \equiv e^{2\pi i y} u_{n,m}$ with $x,y \in [0,1]$. Take
\[
u_{n,m}^{(\ell_1,\ell_2)}
=
\exp\left(2\pi i\left(n \frac{x+\ell}{p_1} + m\frac{y+k}{p_2} \right) \right)
\]
with $\ell,k \in \Z$ and $0 \le \ell < p_1$, $0 \le k < p_2$. It is straightforward to check that this leads to eigenvalues

For later use, we compute the derivative of $e_{\ell,k}$:
\[-\frac{1}{4 \pi}
\nabla e_{\ell,k}(x,y)
=
\left( \frac{1}{p_1} \sin\left(\hat{x}\right) + \frac{1}{p_1}\sin\left(\hat{x} - \hat{y}\right), \frac{1}{p_2}\sin(\hat{y}) - \frac{1}{p_2}\sin(\hat{x} - \hat{y})   \right)
\]
with $\hat{x} = 2\pi p_1^{-1} (x+\ell)$ and $\hat{y} = 2\pi p_2^{-1} (y+k)$.
\end{comment}
As in Section~\ref{sec:floquet}, we label these eigenvalues in increasing order according to multiplicity by
\[
E_1(\bthe)
\le 
E_2(\bthe)\le \cdots E_P(\bthe)
\]
and denote the $P$ spectral bands by
\[
F_k
=
\set{E_k(\bthe) : \bthe \in \R^2},
\quad
1 \le k \le P.
\]
Straightforward computations shows that $\sigma(\Delta_\sqn)=[-4,8]$, hence 
\[
\bigcup_{k=1}^P F_k
=
[-4,8].
\]
Our main theorem of this section is
\begin{theorem}\label{thm:sqnmain}
Let $p_1,p_2 \in \Z_+$ be given.
\begin{enumerate}
\item[{\rm 1.}]
Each $E\in (-4, 8)\setminus \{-1\}$ belongs to $\mathrm{int}(F_k)$ for some $1\leq k\leq P$.
\item[{\rm 2.}] If one of the periods $p_1, p_2$ is not divisible by three, then $E=-1$ belongs to $\mathrm{int}(F_k)$ for some $1\leq k\leq P$.
\end{enumerate}
\end{theorem}

\begin{proof}[Proof of Theorem~\ref{t:bsc:nnn}]
This follows immediately from Theorem~\ref{thm:sqnmain}.
\end{proof}

\subsection{Proof of Theorem \ref{thm:sqnmain}} 

As with the proof of Theorem~\ref{thm:trimain}, we will divide the proof into two different cases: $E\neq -1$ and $E=-1$ and argue by contradiction.
To that end, assume for the sake of establishing a contradiction that $E=\min F_{k+1}=\max F_k$ for some $1\leq k\leq P-1$.

We will use the following lemmas, whose proofs we provide at the end of the present section.
\begin{lemma}\label{lem:constructionsqn}
Let us consider the following system:
\begin{align}\label{eq:xyCondABsqn}
\cos(x) + \cos(y) + \cos(x-y)+\cos(x+y) & = \frac{E}{2} \\
\sin(x) + \sin(x-y)+\sin(x+y) & = 0. \notag
\end{align}
For any $E \in (-4,8) \setminus \{-1\}$, the solution set of \eqref{eq:xyCondABsqn} in $[0,2\pi)^2$ satisfies
\begin{align}\label{eq:solutionx=0sqn}
x=0,\ \  1+2\cos(y)=\frac{E+1}{3},
\end{align}
or 
\begin{align}\label{eq:solutionx=pisqn}
x=\pi,\ \ 1+2\cos(y)=-(E+1).
\end{align}
\end{lemma}

\begin{lemma}\label{lem:sqnJ0empty}
Consider the following system:
\begin{equation} \label{eq:sqnJ0syst}
\begin{cases}
\cos(x) + \cos(y) + \cos(x+y) +\cos(x-y) = \frac{E}{2},\\
\sin(x)+\sin(x-y)+\sin(x+y)=0,\\
\sin(y)-\sin(x-y)+\sin(x+y)=0.
\end{cases}
\end{equation}
For any $E \in (-4,8) \setminus \{0, -1\}$, the solution set of \eqref{eq:sqnJ0syst} is empty. 
For $E=0$, the unique solution of \eqref{eq:sqnJ0syst} in $[0,2\pi)^2$ is $(\pi,\pi)$.
For $E =-1$, the solutions of \eqref{eq:sqnJ0syst} in $[0,2\pi)^2$ are $({2\pi}/{3},{2\pi}/{3})$, $({2\pi}/{3},{4\pi}/{3})$, $({4\pi}/{3},{2\pi}/{3})$ and $({4\pi}/{3},{4\pi}/{3})$.
\end{lemma}

We will use Lemma \ref{lem:constructionsqn} in the $E\neq -1$ case, and Lemma \ref{lem:sqnJ0empty} in the $E=-1$ case.

\subsubsection{$E\neq -1$}\

\begin{proof}[Proof of Theorem~\ref{thm:sqnmain}.1]
Let $E \in (-4,8)\setminus\set{-1}$ be given, and suppose towards a contradiction that $E = \max F_k = \min F_{k+1}$ for some $k$. Define $\tbthe=(\tthe_1,\tthe_2)\in [0,2\pi)^2$ and  $\bm{\ell}^{(1)}=(\ell_1^{(1)},\ell_2^{(1)})\in \Lambda$ via
\begin{align}\label{def:ttheellsqnE}
\tthe_1=0,\ \ell_1^{(1)}=0,\ \frac{\tthe_2+2\pi \ell_2^{(1)}}{p_2}=\arccos\Big(\frac{E-2}{6}\Big)\in  (0,\pi).
\end{align}
Note that since $E\in (-4, 8)$, we have $\frac{E-2}{6}\in (-1,1)$, hence $\arccos\Big(\frac{E-2}{6}\Big)$ is always well-defined.
Note also that $\tthe_2$ and $\ell_2^{(1)}$ are uniquely determined.
Using \eqref{def:ttheellsqnE}, one easily checks that 
\[e_{\bm{\ell}^{(1)}}(\tbthe)=E,\]
and
\begin{align}\label{eq:sqnJbeta10}
(1,0)\cdot \nabla e_{\bm{\ell}^{(1)}}(\tbthe)=0.
\end{align}
As in the proof of Theorem~\ref{thm:trimain}, denote $\Lambda_E(\tbthe) = \{\bm{\ell} \in \Lambda : e_{\bm{\ell}}(\tbthe) = E\}$, let $r := |\Lambda_E(\tbthe)|$ be the multiplicity of $E$ as an eigenvalue of $H(\tbthe)$, and choose $s\in \Z\cap [1,r]$ such that 
\[E_{k-s}(\tbthe)<E_{k-s+1}(\tbthe)=\cdots=E_k(\tbthe)=\cdots=E_{k+r-s}(\tbthe)<E_{k+r-s+1}(\tbthe).\]
Since all the eigenvalues are continuous in $\bthe$, we can take $\varepsilon>0$ small enough such that 
\[E_{k-s}(\bthe)<E_{k-s+1}(\bthe),\, \text{ and }\, E_{k+r-s}(\bthe)<E_{k+r-s+1}(\bthe),\]
holds whenever $\|\bthe-\tbthe\|_{\R^2}<\varepsilon$. Given $\bm{\ell}\in \Lambda$ and a unit vector $\bbe=(\beta_1, \beta_2)$, we have
\begin{align}
e_{\bm{\ell}}(\tbthe+ t \bbe)
& = e_{\bm{\ell}}(\tbthe)+t\bbe \cdot \nabla e_{\bm{\ell}}(\tbthe)+O(t^2) \label{eq:pertgeneralbetasqn1order} \\
& =e_{\bm{\ell}}(\tbthe)+t\bbe \cdot \nabla e_{\bm{\ell}}(\tbthe) \label{eq:pertgeneralbetasqn2order}\\
&\qquad -\frac{t^2}{2}\Bigg[\frac{\beta_1^2}{p_1^2}\cos\Big(\frac{\tthe_1+2\pi \ell_1}{p_1}\Big)
+\frac{\beta_2^2}{p_2^2}\cos\Big(\frac{\tthe_2+2\pi \ell_2}{p_2}\Big) \notag\\ 
&\qquad +\Big(\frac{\beta_1}{p_1}-\frac{\beta_2}{p_2}\Big)^2 \cos\Big(\frac{\tthe_1+2\pi \ell_1}{p_1}-\frac{\tthe_2+2\pi \ell_2}{p_2}\Big)\notag\\
&\qquad +\Big(\frac{\beta_1}{p_1}+\frac{\beta_2}{p_2}\Big)^2 \cos\Big(\frac{\tthe_1+2\pi \ell_1}{p_1}+\frac{\tthe_2+2\pi \ell_2}{p_2}\Big)\Bigg] +O(t^3).  \notag
\end{align}
In particular, we will use \eqref{eq:pertgeneralbetasqn1order} if $\bbe\cdot  \nabla e_{\bm{\ell}}(\tbthe)\neq 0$, and \eqref{eq:pertgeneralbetasqn2order} otherwise.

For any vector $\bbe\in \R^2\setminus \{\bm{0}\}$, let 
\begin{equation}\label{def:Jbetasqn}
\begin{aligned}
\mathcal{J}_{\bbe}^0
=
\CJ_{\bbe}^0(\tbthe):
&=
\{\bm{\ell} \in \Lambda_E(\tbthe) :\ \bbe\cdot \nabla e_{\bm{\ell}}(\tbthe)=0\},\\
\mathcal{J}_{\bbe}^{\pm} = \CJ_{\bbe}^\pm(\tbthe):
&=
\{\bm{\ell} \in \Lambda_E(\tbthe) :\ \pm \bbe \cdot \nabla e_{\bm{\ell}}(\tbthe)>0\}.
\end{aligned}
\end{equation}
By definition, we must have
\begin{align}\label{eq:sumJbetasqn}
|\mathcal{J}_{\bbe}^0|+|\mathcal{J}_{\bbe}^+|+|\mathcal{J}_{\bbe}^-|=r
\end{align}
for any $\bbe$. We also define $\mathcal{J}_0$ as follows
\begin{align}\label{def:J0sqn}
\mathcal{J}_0
=\CJ_0(\tbthe)
:=
\{\bm{\ell} \in \Lambda_E(\tbthe):\ \nabla e_{\bm{\ell}}(\tbthe) = \bm{0}\}.
\end{align}
If $E\neq 0$, Lemma~\ref{lem:sqnJ0empty} directly implies $\mathcal{J}_0=\emptyset$.
If $E=0$, $\mathcal{J}_0$ is also empty. 
To see this, suppose on the contrary that $\bm{\ell} = (\ell_1,\ell_2) \in \CJ_0$.
Lemma~\ref{lem:sqnJ0empty} implies that
\begin{equation} \label{eq:sqn:J0emptyE=0eq1}
\frac{\tthe_2+2\pi \ell_2}{p_2}=\pi,
\end{equation}
and \eqref{def:ttheellsqnE} forces
\begin{equation} \label{eq:sqn:J0emptyE=0eq2}
\frac{\tthe_2+2\pi \ell^{(1)}_2}{p_2}=\arccos\left(-\frac{1}{3}\right).
\end{equation}
Subtracting \eqref{eq:sqn:J0emptyE=0eq1} from \eqref{eq:sqn:J0emptyE=0eq2} yields
\[\frac{\ell^{(1)}_2-\ell_2}{p_2}=\frac{1}{2\pi}\arccos\left(-\frac{1}{3}\right)-\frac{1}{2}.\]
However, this implies that $(2\pi)^{-1}\arccos\left(-1/3\right)$ is a rational number, which contradicts the following well-known fact, whose proof we supply at the end of the present section.
\begin{lemma}\label{lem:arccos1/3}
\[\frac{1}{2\pi}\arccos\left(-\frac{1}{3}\right)\in \R\setminus \Q.\]
\end{lemma}
Therefore $\mathcal{J}_0=\emptyset$ for any $E\neq -1$.

We choose $\bbe_1=(1,0)$. Then \eqref{eq:sqnJbeta10} implies $\bm{\ell}^{(1)} \in \CJ_{\bbe_1}^0$, and hence
\begin{align}\label{eq:beta1sqnnon-empty}
\mathcal{J}_{\bbe_1}^0\neq \emptyset.
\end{align}

Next we are going to perturb the point $\tbthe$ and count the eigenvalues.
Since $\mathcal{J}_0=\emptyset$, we can choose a unit vector $\bbe_2$ such that 
\begin{align}\label{eq:beta2sqnnon-empty}
\bbe_2\cdot \nabla e_{\bm{\ell}}(\tbthe)\neq 0,
\end{align}
holds for any $\bm{\ell} \in \Lambda_E(\tbthe)$.
Thus $\mathcal{J}_{\bbe_2}^0=\emptyset$ and
\begin{align}\label{eq:beta2sqnnon-empty'}
|\mathcal{J}_{\bbe_2}^+|+|\mathcal{J}_{\bbe_2}^-|=r.
\end{align}
Arguing as in the proof of Theorem~\ref{thm:trimain}.1, we deduce \begin{comment}
We first perturb the eigenvalues along the $\bbe_2$ direction.
Since $\mathcal{J}_{\bbe_2}^0 = \emptyset$, we will always employ \eqref{eq:pertgeneralbetasqn1order}.

For $t > 0$ small enough, we have the following.
\begin{itemize}
\item  If ${\bm{\ell}}  \in \mathcal{J}_{\bbe_2}^+$, we have
\[
E_{k+r-s+1}(\widetilde{\theta}_1 + t\beta_{2,1}, \widetilde{\theta}_2 + t\beta_{2,2})
>
e_{{\bm{\ell}}}(\widetilde{\theta}_1 + t\beta_{2,1}, \widetilde{\theta}_2 + t\beta_{2,2})
>
E
=
\max F_k,
\]
which implies
\begin{align}\label{eq1:Jbeta2+sqn}
|\mathcal{J}_{\bbe_2}^+|\leq r-s.
\end{align}

\item If ${\bm{\ell} } \in \mathcal{J}_{\bbe_2}^-$, we have 
\[
E_{k-s}(\widetilde{\theta}_1 + t\beta_{2,1}, \widetilde{\theta}_2 + t\beta_{2,2})
<
e_{{\bm{\ell}}}(\widetilde{\theta}_1 + t\beta_{2,1}, \widetilde{\theta}_2 + t\beta_{2,2})
<
E
=
\min F_{k+1},
\]
which implies
\begin{align}\label{eq2:Jbeta2+sqn}
|\mathcal{J}_{\bbe_2}^-|\leq s.
\end{align}
\end{itemize}
{In view of \eqref{eq:beta2sqnnon-empty'}, Equations~\eqref{eq1:Jbeta2+sqn} and \eqref{eq2:Jbeta2+sqn} imply
\begin{equation} \label{eq:Jbeta2MinusCard+sqn}
|\mathcal{J}_{\bbe_2}^-|
=
s.
\end{equation}
Upon realizing that $\mathcal{J}_{-\bbe_2}^0 = \emptyset$ and $\mathcal{J}_{-\bbe_2}^\pm = \mathcal{J}_{\bbe_2}^\mp$, we may apply the analysis above with $\bbe_2$ replaced by $-\bbe_2$ and conclude that
\begin{equation} \label{eq:JMinusBeta2+sqn}
|\mathcal{J}_{\bbe_2}^+| = |\mathcal{J}_{-\bbe_2}^-| = s.
\end{equation}
In particular, \eqref{eq:Jbeta2MinusCard+sqn} and \eqref{eq:JMinusBeta2+sqn} imply \end{comment}
\begin{align}\label{eq4:Jbeta2sqn}
r=2s.
\end{align}

\subsubsection*{Perturbation along $\bbe_1$}
Now we perturb the eigenvalues along $\bbe_1=(1,0)$.
The case when ${\bm{\ell}} \in \mathcal{J}_{\bbe_1}^{\pm}$ is similar to that of $\bbe_2$.
The difference here is that, according to \eqref{eq:beta1sqnnon-empty}, $\mathcal{J}_{\bbe_1}^0\neq \emptyset$.

By Lemma~\ref{lem:constructionsqn}, we have that for $(\ell_1,\ell_2)\in \mathcal{J}_{\bbe_1}^0$,
\begin{align}\label{eq:beta1sqnt2sign}
(E+1)\Bigg[
\cos\Big(\frac{\widetilde{\theta}_1 + 2\pi {\ell_1}}{p_1}\Big)
+&\cos\Big(\frac{\widetilde{\theta}_1 + 2\pi {\ell_1}}{p_1}-\frac{\widetilde{\theta}_2 + 2\pi {\ell_2}}{p_2}\Big)\\
&\qquad +\cos\Big(\frac{\widetilde{\theta}_1 + 2\pi {\ell_1}}{p_1}+\frac{\widetilde{\theta}_2 + 2\pi {\ell_2}}{p_2}\Big) 
\Bigg]
>0. \notag
\end{align}
Indeed, if $(\ell_1,\ell_2)\in \mathcal{J}_{\bbe_1}^0$, 
$(x,y) = (p_1^{-1} (\tthe_1 + 2\pi \ell_1), p_2^{-1}(\tthe_2 + 2\pi \ell_2))$ is a solution to \eqref{eq:xyCondABsqn}.
Hence Lemma~\ref{lem:constructionsqn} implies that 
we have 
either 
\[\frac{\widetilde{\theta}_1 + 2\pi {\ell_1}}{p_1}=0,\ \ 1+2\cos\Big(\frac{\widetilde{\theta}_2 + 2\pi {\ell_2}}{p_2}\Big)=\frac{E+1}{3},\]
or 
\[\frac{\widetilde{\theta}_1 + 2\pi {\ell_1}}{p_1}=\pi,\ \ 1+2\cos\Big(\frac{\widetilde{\theta}_2 + 2\pi {\ell_2}}{p_2}\Big)=-(E+1).\]
Clearly, both cases lead to \eqref{eq:beta1sqnt2sign}.

By employing \eqref{eq:pertgeneralbetasqn2order}, we obtain
\begin{align}
e_{{\bm{\ell}}}(\tbthe + t \bbe_1)
=
E & - {\frac{ t^2}{2 p_1^2}} \Bigg[
\cos\Big(\frac{\widetilde{\theta}_1 + 2\pi {\ell_1}}{p_1}\Big)
+\cos\Big(\frac{\widetilde{\theta}_1 + 2\pi {\ell_1}}{p_1}-\frac{\widetilde{\theta}_2 + 2\pi {\ell_2}}{p_2}\Big)\\
& \qquad\qquad\qquad
+\cos\Big(\frac{\widetilde{\theta}_1 + 2\pi {\ell_1}}{p_1}+\frac{\widetilde{\theta}_2 + 2\pi {\ell_2}}{p_2}\Big) 
\Bigg] + O(t^3)\notag
\end{align}
for $\bm{\ell} \in \CJ_{\bbe_1}^0$. Combining this with \eqref{eq:beta1sqnt2sign}, we obtain that for $|t|>0$ small enough
\begin{align}\label{eq:Jbeta10sqn}
e_{{\bm{\ell}}}(\tbthe + t \bbe_1)
\begin{cases}
<E,\ \ \text{if } E+1>0,\\
>E,\ \ \text{if } E+1<0.
\end{cases}
\end{align}
Notice that the choice of $\bbe_1$ causes the second $t^2$ term of \eqref{eq:pertgeneralbetasqn2order} to drop out.

Without loss of generality, we assume $E\in (-1, 8)$. The complementary case when $E\in (-4, -1)$ can be handled similarly.
For $E\in (-1, 8)$, \eqref{eq:Jbeta10sqn} implies that 
\begin{align}\label{eq:Jbeta10sqn'}
e_{{\bm{\ell}}}(\tbthe + t \bbe_1)
<
E
=
\min F_{k+1},
\end{align}
holds for $|t|>0$ small enough and for any $\bm{\ell} \in \mathcal{J}_{\bbe_1}^0$. 

Combining \eqref{eq:Jbeta10sqn'} with \eqref{eq:pertgeneralbetasqn1order}, we have the following.

For $t>0$ small enough,
\begin{itemize}
\item If ${\bm{\ell}} \in \mathcal{J}_{\bbe_1}^+$, we have
\[
E_{k+r-s+1}(\tbthe + t \bbe_1)
>
e_{{\bm{\ell}}}(\tbthe + t \bbe_1)
>
E
=
\max F_k,
\]
which implies 
\begin{align}\label{eq1:Jbeta1+sqn}
|\mathcal{J}_{\bbe_1}^+|\leq r-s
=
s,
\end{align}
{where the equality follows from \eqref{eq4:Jbeta2sqn}.}

\item If ${\bm{\ell}} \in \mathcal{J}_{\bbe_1}^0 \bigcup \mathcal{J}_{\bbe_1}^-$, we have 
\[
E_{k-s-1}(\tbthe + t \bbe_1)
<
e_{{\bm{\ell}}}(\tbthe + t \bbe_1)
<
E
=
\min F_{k+1},
\]
which implies
\begin{align}\label{eq2:Jbeta1+sqn}
|\mathcal{J}_{\bbe_1}^0|+|\mathcal{J}_{\bbe_1}^-|\leq s.
\end{align}
\end{itemize}
{In view of \eqref{eq:sumJbetasqn} and \eqref{eq4:Jbeta2sqn}, Equations~\eqref{eq1:Jbeta1+sqn} and \eqref{eq2:Jbeta1+sqn} yield
\begin{equation} \label{eq3:Jbeta1+sqn}
|\mathcal{J}_{\bbe_1}^+|
=
|\mathcal{J}_{\bbe_1}^0|+|\mathcal{J}_{\bbe_1}^-|
=
s.
\end{equation}
As before, we may observe that $\CJ_{-\bbe_1}^0 = \CJ_{\bbe_1}^0$ and $\CJ_{-\bbe_1}^\pm = \CJ_{\bbe_1}^\mp$. Then, the analysis above applied with $\bbe_1$ replaced by $-\bbe_1$ forces
\begin{equation} \label{eq3:Jbeta1-sqn}
|\mathcal{J}_{\bbe_1}^-|
=
|\mathcal{J}_{\bbe_1}^0|+|\mathcal{J}_{\bbe_1}^+|
=
s.
\end{equation}
Taken together, \eqref{eq3:Jbeta1+sqn} and \eqref{eq3:Jbeta1-sqn} imply $|\CJ_{\bbe_1}^0| = 0$, which contradicts \eqref{eq:beta1sqnnon-empty}.
}
\end{proof}

\subsubsection{$E=-1$}\

First, we would like to make a remark on our strategy of the proof of the $E=-1$ case, and on the importance of one of the period being not divisible by $3$.
\begin{remark}\label{rem:sqn}
For the exceptional energy {$E = -1$} of the EHM lattice, we can not use eigenvalues with vanishing gradients to create un-even numbers of counting unless neither $p_1$ nor $p_2$ is divisible by $3$.
The reason is the following: suppose {\it only} $p_1$ is not divisible by $3$ and we choose $\tbthe=(\tthe_1, \tthe_2)$ and $\bm{\ell}^{(1)}=(\ell_1^{(1)}, \ell_2^{(1)})$ such that $e_{\bm{\ell}^{(1)}}(\tbthe)=-1$ and $\nabla e_{\bm{\ell}^{(1)}}(\tbthe)=\bm{0}$. 
Lemma \ref{lem:sqnJ0empty} yields four possibilities $(p_1^{-1}(\tthe_1+2\pi \ell_1^{(1)}), p_2^{-1}(\tthe_2+2\pi \ell_2^{(1)}))=(2\pi/3,2\pi/3)$, $(2\pi/3,4\pi/3)$, $(4\pi/3,2\pi/3)$ or $(4\pi/3,4\pi/3)$.
Without loss of generality, we choose $(2\pi/3, 2\pi/3)$, the other three choices are essentially the same.
Since $p_2$ is divisible by $3$, there exists $\bm{\ell}^{(2)}$, such that 
$(p_1^{-1}(\tthe_1+2\pi \ell_1^{(2)}), p_2^{-1}(\tthe_2+2\pi \ell_2^{(2)}))=(2\pi/3,4\pi/3)$.
Hence $e_{\bm{\ell}^{(2)}}(\tbthe)$ is also located at $-1$ with vanishing gradient.
Perturbing $e_{\bm{\ell}^{(1)}}(\tbthe)$ and $e_{\bm{\ell}^{(2)}}(\tbthe)$ along a given direction $(\beta_1, \beta_2)$ is equivalent to controlling the signs of the following two expressions:
$$\beta_1\beta_2\ \ \text{and}\ \ -\beta_1\beta_2.$$
This means we can never choose two different directions that lead to un-even {counts}.
Therefore we need to develop a new argument for this case.

Indeed, when $p_1$ is not divisible by $3$, we choose $p_1^{-1}(\tthe_1+2\pi \ell_1^{(1)})=2\pi/3$ and $\tthe_2$ such that 
$p_2^{-1}(\tthe_2+2\pi \ell_2)\notin \{2\pi/3, 4\pi/3\}$ {{\it regardless} of the choice of $\ell_2$}. Such choices guarantee that there are in total $p_2$ eigenvalues located at $-1$, which are 
$\{e_{\bm{\ell}}(\tbthe),\ \ \ell_1=\ell_1^{(1)}\}$. It then suffices to control the movements of these eigenvalues along any given direction.
A key observation is that along any direction, approximately $2p_2/3$ eigenvalues will move up (down) while the other $p_2/3$ eigenvalues move down (up), see \eqref{eq:E=-1Jbbe}.
This leads to un-even counting that we need.
Let us point out that if both $p_1, p_2$ are divisible by $3$, this argument does not work {(as it must, given the example constructed in Theorem~\ref{t:nnnExamples})}: there will be $2p_2$ eigenvalues located at $-1$, and $p_2$ of them move up while the other $p_2$ of them move down along any given direction.
\end{remark}

\begin{proof}[Proof of Theorem~\ref{thm:sqnmain}.2]
Without loss of generality, we assume $p_1$ is not divisible by $3$.
Let $p_j=3p_j'+k_j$, where $p_j',k_j \in \Z$ with $0 \le k_j < 3$ and then define $\widetilde\bthe$ by
\begin{align*}
\tthe_1= \frac{2\pi k_1}{3},\ \ 
\tthe_2={\frac{k_2+1}{4}\pi}.
\end{align*}
%{I think the whole argument should also work with $\tthe_2 =\pi$ (regardless of the value of $k$); but I should check this carefully. In particular, \eqref{eq:E=-1J0} still works. For, if $p_2^{-1}(\pi + 2\pi \ell) = 2\pi/3$ or $4\pi/3$, then $3(2\ell+1) = 2p_2$ or $4p_2$, implying equality of an even integer and an odd integer, which is no good.}
{As usual, denote $\Lambda_E(\tbthe) = \Lambda_{-1}(\tbthe) = \{\bm{\ell} \in \Lambda : e_{\bm{\ell}}(\tbthe) = -1\}$. We first claim that
\begin{equation} \label{eq:sqnE=-1:multiplicity}
\Lambda_{-1}(\tbthe)
=
\set{(p_1', \ell_2) : 0 \leq \ell_2 < p_2 \text{ and } \ell_2 \in \Z}.
\end{equation}}
Let us consider the trigonometric equation
\begin{align}\label{eq:E=-1}
\cos(x)+\cos(y)+\cos(x-y)+\cos(x+y)=-\frac{1}{2}=\frac{E}{2}.
\end{align}
Using the identity $\cos(x-y)+\cos(x+y)=2\cos(x) \cos(y)$, we see that \eqref{eq:E=-1} is equivalent to 
\[(2\cos(x)+1) (2\cos(y)+1)=0,\]
whose solutions are $\cos(x)=-1/2$ or $\cos(y)=-1/2$. {With our choice of $\tbthe$, it is clear that
\begin{equation} \label{eq:E=-1the1ell1}
\frac{\tthe_1 + 2\pi p_1'}{p_1}
=
\frac{2\pi}{3},
\quad
\cos\left(\frac{\tthe_1 + 2\pi p_1'}{p_1}\right)
=
-\frac{1}{2}.
\end{equation}
Consequently, 
\begin{equation} \label{eq:sqn:E=-1:LambdaEeq1}
e_{(p_1',\ell_2)}(\tbthe) = -1 \quad \text{for every } 0 \le \ell_2 < p_2.
\end{equation}}
Due to our choice of $\tthe_2$, we get
\begin{align}\label{eq:E=-1J0}
\cos(p_2^{-1}(\tthe_2+2\pi \ell_2))\neq -\frac{1}{2}\quad
\text{for any } \ell_2 \in [0,p_2) \cap \Z.
\end{align}
Indeed, since $p_2^{-1}(\tthe_2+2\pi\ell_2) \in [0,2\pi)$, $\cos(p_2^{-1}(\tthe_2+2\pi \ell_2))=-1/2$ would force 
\[\frac{\tthe_2+2\pi \ell_2}{p_2} \in \set{\frac{2\pi}{3}, \frac{4\pi}{3}},\]
which, after doing some algebra, leads to {
\[
3(8\ell_2+k_2+1) \in \{8p_2,16p_2\},
\]
which is plainly impossible, since $\ell_2,p_2 \in \Z$ and $k_2 \in \{0,1,2\}$.}
Additionally, due to our choice of $\tthe_1$, we also have
\begin{align}\label{eq:E=-1Jbeta10}
{\cos(p_1^{-1}(\tthe_1+2\pi \ell_1))\neq -1/2 \quad
\text{for any } \ell_1 \in \big([0,p_1)\cap\Z\big)\setminus\{p_1'\}}.
\end{align}
To see this, suppose on the contrary that \eqref{eq:E=-1Jbeta10} fails. This forces
\[\frac{\tthe_1+2\pi \ell_1}{p_1}=\frac{4\pi}{3}\]
for some $0 \leq \ell_1 < p_1$ with $\ell_1\neq p_1'$. Since
\[\frac{\tthe_1+2\pi p_1'}{p_1}=\frac{2\pi}{3},\]
this implies
\[\frac{2\pi (\ell_1-p_1')}{p_1}=\frac{2\pi}{3},\]
which is impossible since $p_1$ is not divisible by $3$. Combining \eqref{eq:E=-1J0} and \eqref{eq:E=-1Jbeta10} yields
\begin{equation} \label{eq:sqn:E=-1:LambdaEeq2} 
{e_{\bm{\ell}}(\tbthe)\neq - 1
\quad
\text{for any } \bm{\ell} =(\ell_1,\ell_2)\in\Lambda \text{ such that } \ell_1 \neq p_1'.}
\end{equation}
Taken together, \eqref{eq:sqn:E=-1:LambdaEeq1} and \eqref{eq:sqn:E=-1:LambdaEeq2} imply \eqref{eq:sqnE=-1:multiplicity}.
%Note that combining Lemma \ref{lem:sqnJ0empty} with \eqref{eq:E=-1J0} implies $\mathcal{J}_0=\emptyset$.

Let us choose $\bbe=(\beta_1,\beta_2)=(1,0)$. 
We have that for any $\bm{\ell}\in \Lambda$:
\begin{align*}
&\bbe\cdot \nabla e_{\bm{\ell}}(\tbthe)\\
=&-\sin\Big(\frac{\tthe_1+2\pi \ell_1}{p_1}\Big)
-\sin\Big(\frac{\tthe_1+2\pi \ell_1}{p_1}-\frac{\tthe_2+2\pi \ell_2}{p_2}\Big)
-\sin\Big(\frac{\tthe_1+2\pi \ell_1}{p_1}+\frac{\tthe_2+2\pi \ell_2}{p_2}\Big)\\
=&-\sin\Big(\frac{\tthe_1+2\pi \ell_1}{p_1}\Big)\Bigg[1+2\cos\Big(\frac{\tthe_2+2\pi \ell_2}{p_2}\Big)\Bigg].
\end{align*}
By \eqref{eq:sqnE=-1:multiplicity}, \eqref{eq:E=-1the1ell1}, and \eqref{eq:E=-1J0}, we have the following {for any $\bm{\ell} = (\ell_1, \ell_2) \in \Lambda_{-1}(\tbthe)$}:
\[\sin\Big(\frac{\tthe_1+2\pi \ell_1}{p_1}\Big)=\frac{\sqrt{3}}{2},\ \ \cos\Big(\frac{\tthe_2+2\pi \ell_2}{p_2}\Big)\neq -\frac{1}{2}.\]
This implies
\begin{align}\label{eq:E=-1Jbbe}
\mathcal{J}_{\bbe}^0=\emptyset,\ \ \text{and }\ \ 
\mathcal{J}_{\bbe}^{\pm}
=
\Bigg\lbrace \bm{\ell} \in \Lambda_{-1}(\tbthe) :\ \ \mp \frac{1}{2}  \mp \cos\Big(\frac{\tthe_2+2\pi \ell_2}{p_2}\Big)>0 \Bigg\rbrace.
\end{align}
Hence we expect that $|\CJ_{\bbe}^+|\sim p_2/3$, and $|\CJ_{\bbe}^-|\sim 2p_2/3$.
More precisely, we note that
\[
{\CJ_{\bbe}^+
=
\Big\lbrace (p_1', \ell_2):\ \ \frac{2\pi}{3} < \frac{(k_2+1)\pi/4 + 2\pi \ell_2}{p_2} < \frac{4\pi}{3} \Big\rbrace}.\]
Using $p_2=3p_2'+k_2$, we obtain
\[
{\CJ_{\bbe}^+
=
\Big\lbrace (p_1',\ell_2):\ \ p_2'+\frac{5k_2-3}{24} <  \ell_2 < 2p_2'+\frac{13k_2-3}{24} \Big\rbrace}.
\]
Consequently,
\begin{align*}
\CJ_{\bbe}^+=
\begin{cases}
\{(p_1', \ell_2):\ \ p_2'\leq \ell_2 \leq 2p_2'-1\},\ \ \text{if } k=0,\\
\{(p_1', \ell_2):\ \ p_2'+1\leq \ell_2 \leq 2p_2'\},\ \ \text{if } k=1, 2.
\end{cases}.
\end{align*}
Therefore 
\begin{align}\label{eq:Jbetapm}
(|\CJ_{\bbe}^+|, |\CJ_{\bbe}^-|)=
\begin{cases}
(p_2', 2p_2'),\ \ \text{if } k=0,\\
(p_2', 2p_2'+1),\ \ \text{if } k=1,\\
(p_2', 2p_2'+2),\ \ \text{if } k=2.
\end{cases}
\end{align}
Note that $p_2'\geq 1$ whenever $k_2 = 0$.
Thus, a direct consequence of \eqref{eq:Jbetapm} is
\begin{align}\label{eq:Jbetapmneq}
|\CJ_{\bbe}^+|\neq |\CJ_{\bbe}^-|.
\end{align}
On the other hand, since $\CJ_{\bbe}^0 = \emptyset$, following the same argument as in the proof of Theorems~\ref{thm:trimain}.1 yields $|\CJ_{\bbe}^+|=|\CJ_{\bbe}^-|$, which contradicts \eqref{eq:Jbetapmneq}.
\end{proof}

\subsection{Proofs of Lemmas \ref{lem:constructionsqn}, \ref{lem:sqnJ0empty}, and \ref{lem:arccos1/3}}

\begin{proof}[Proof of Lemma~\ref{lem:constructionsqn}]
Let $x$ and $y$ solve \eqref{eq:xyCondABsqn} with $E \neq -1$. The second condition therein yields
\[
\sin(x) + 2\sin(x)\cos(y) = 0,
\]
leading to two possibilities: $\sin(x) = 0$ or $\cos(y) = -1/2$. If $\sin(x) = 0$, we get $x = 0$ or $x = \pi$, which yields \eqref{eq:solutionx=0sqn} and \eqref{eq:solutionx=pisqn} upon plugging in to the first condition in \eqref{eq:xyCondABsqn}. In the event that $\cos(y) = -1/2$, we arrive at
\begin{align*}
\cos(x) + \cos(y) + \cos(x-y) + \cos(x+y)
& =
\cos(x) + \cos(y) + 2\cos(x)\cos(y) \\
& =
\cos(x) - \frac{1}{2} - \cos(x) \\
& =
-\frac{1}{2},
\end{align*}
in contradiction with $E \neq -1$.
\end{proof}

\begin{proof}[Proof of Lemma~\ref{lem:sqnJ0empty}]
Suppose $x$ and $y$ satisfy \eqref{eq:sqnJ0syst}. From the proof of Lemma~\ref{lem:constructionsqn}, the second condition of \eqref{eq:sqnJ0syst} implies $\sin(x) = 0$ or $\cos(y) = -1/2$. Thus, $x = 0$, $x = \pi$, $y=2\pi/3$, or $y = 4\pi/3$. 
When $\sin(x) = 0$, the third condition of \eqref{eq:sqnJ0syst} forces $\sin(y) = 0$. The four points so obtained yield $E = 8$ when $(x,y) = (0,0)$, $E=-4$ when $(x,y) = (0,\pi),(\pi,0)$ and $E = 0$ when $(x,y) = (\pi,\pi)$. 
Alternatively, when $\cos(y) = -1/2$, the third condition of \eqref{eq:sqnJ0syst} yields $\cos(x) = -1/2$, which impies $x = 2\pi/3$ or $x = 4\pi/3$. As in the Proof of Lemma~\ref{lem:constructionsqn}, the four points corresponding to
\[
x,y \in \set{\frac{2\pi}{3}, \frac{4\pi}{3}}
\]
all yield $E = -1$.
\end{proof}

\begin{proof}[Proof of Lemma~\ref{lem:arccos1/3}]
Suppose 
\begin{align}\label{eq1:cos1/3}
\cos\left(\frac{2\pi m}{n}\right)=-\frac{1}{3},
\end{align}
for $m/n\in \Q$.
Let $T_n(\cdot)$ denote the $n$-th degree Cheybeshev polynomial so that
\begin{align}\label{eq2:cos1/3}
T_n\left(\cos\left(\frac{2\pi m}{n}\right)\right)=\cos(2\pi m)=1.
\end{align}
It is well-known that $T_n(x)=\sum_{k=0}^n a_k x^k$, where $a_n=2^{n-1}$ and $a_k\in \Z$ for any $k$.
Hence \eqref{eq1:cos1/3} and \eqref{eq2:cos1/3} imply
\[2^{n-1}\left(-\frac{1}{3}\right)^n+\sum_{k=0}^{n-1} a_k \left(-\frac{1}{3}\right)^{k} =1.\]
Multiplying by $(-3)^n$ on both sides of the equation, we obtain
\[2^{n-1}-3\sum_{k=0}^{n-1} a_k (-3)^{n-k-1}=(-3)^n,\]
which implies $2^{n-1}$ is divisible by $3$.
Contradiction.
\end{proof}

\subsection{\boldmath Opening a gap at $-1$}

\begin{figure*}[b]
\begin{tikzpicture}[yscale=0.8, xscale=0.8]
\draw [-,line width = .1cm] (0,0) -- (12,0);
\draw [-,line width = .1cm] (0,0) -- (0,12);
\draw [-,line width=  .1cm] (0,3) -- (12,3);
\draw [-,line width=  .1cm] (3,0) -- (3,12);
\draw [-,line width=  .1cm] (0,6) -- (12,6);
\draw [-,line width=  .1cm] (6,0) -- (6,12);
\draw [-,line width=  .1cm] (0,9) -- (12,9);
\draw [-,line width=  .1cm] (9,0) -- (9,12);
\draw [-,line width=  .1cm] (0,12) -- (12,12);
\draw [-,line width=  .1cm] (12,0) -- (12,12);
\draw [-,line width=  .1cm] (0,0) -- (12,12);
\draw [-,line width=  .1cm] (3,0) -- (12,9);
\draw [-,line width=  .1cm] (6,0) -- (12,6);
\draw [-,line width=  .1cm] (9,0) -- (12,3);
\draw [-,line width=  .1cm] (0,3) -- (9,12);
\draw [-,line width=  .1cm] (0,6) -- (6,12);
\draw [-,line width=  .1cm] (0,9) -- (3,12);
\draw [-,line width=  .1cm] (3,0) -- (0,3);
\draw [-,line width=  .1cm] (6,0) -- (0,6);
\draw [-,line width=  .1cm] (9,0) -- (0,9);
\draw [-,line width=  .1cm] (12,0) -- (0,12);
\draw [-,line width=  .1cm] (12,3) -- (3,12);
\draw [-,line width=  .1cm] (12,6) -- (6,12);
\draw [-,line width=  .1cm] (12,9) -- (9,12);
\filldraw[color=black, fill=black](0,0) circle (.2);
\filldraw[color=black, fill=black](3,0) circle (.2);
\filldraw[color=black, fill=black](6,0) circle (.2);
\filldraw[color=black, fill=black](9,0) circle (.2);
\filldraw[color=black, fill=black](0,3) circle (.2);
\filldraw[color=red, fill=red](3,3) circle (.2);
\filldraw[color=red, fill=red](6,3) circle (.2);
\filldraw[color=red, fill=red](9,3) circle (.2);
\filldraw[color=black, fill=black](0,6) circle (.2);
\filldraw[color=red, fill=red](3,6) circle (.2);
\filldraw[color=red, fill=red](6,6) circle (.2);
\filldraw[color=red, fill=red](9,6) circle (.2);
\filldraw[color=black, fill=black](0,9) circle (.2);
\filldraw[color=red, fill=red](3,9) circle (.2);
\filldraw[color=red, fill=red](6,9) circle (.2);
\filldraw[color=red, fill=red](9,9) circle (.2);
\filldraw[color=black, fill=black](12,0) circle (.2);
\filldraw[color=black, fill=black](12,3) circle (.2);
\filldraw[color=black, fill=black](12,6) circle (.2);
\filldraw[color=black, fill=black](12,9) circle (.2);
\filldraw[color=black, fill=black](12,12) circle (.2);
\filldraw[color=black, fill=black](0,12) circle (.2);
\filldraw[color=black, fill=black](3,12) circle (.2);
\filldraw[color=black, fill=black](6,12) circle (.2);
\filldraw[color=black, fill=black](9,12) circle (.2);
\draw [-,line width=  .1cm,color=red] (3,3) -- (3,9);
\draw [-,line width=  .1cm,color=red] (6,3) -- (6,9);
\draw [-,line width=  .1cm,color=red] (9,3) -- (9,9);
\draw [-,line width=  .1cm,color=red] (3,3) -- (9,3);
\draw [-,line width=  .1cm,color=red] (3,6) -- (9,6);
\draw [-,line width=  .1cm,color=red] (3,9) -- (9,9);
\draw [-,line width=  .1cm,color=red] (3,6) -- (6,9);
\draw [-,line width=  .1cm,color=red] (3,3) -- (9,9);
\draw [-,line width=  .1cm,color=red] (6,3) -- (9,6);
\draw [-,line width=  .1cm,color=red] (3,6) -- (6,3);
\draw [-,line width=  .1cm,color=red] (3,9) -- (9,3);
\draw [-,line width=  .1cm,color=red] (6,9) -- (9,6);
\node at (3.9,3.3){\hot{$q_1$}};
\node at (6.9,3.3){\hot{$q_2$}};
\node at (9.9,3.3){\hot{$q_3$}};
\node at (3.9,6.3){\hot{$q_4$}};
\node at (6.9,6.3){\hot{$q_5$}};
\node at (9.9,6.3){\hot{$q_6$}};
\node at (3.9,9.3){\hot{$q_7$}};
\node at (6.9,9.3){\hot{$q_8$}};
\node at (9.9,9.3){\hot{$q_9$}};
\end{tikzpicture}
\caption{A $3\times 3$ potential on the square lattice that opens a gap at $E=-1$ with small positive positive coupling.}\label{fig:sqn33period}
\end{figure*}
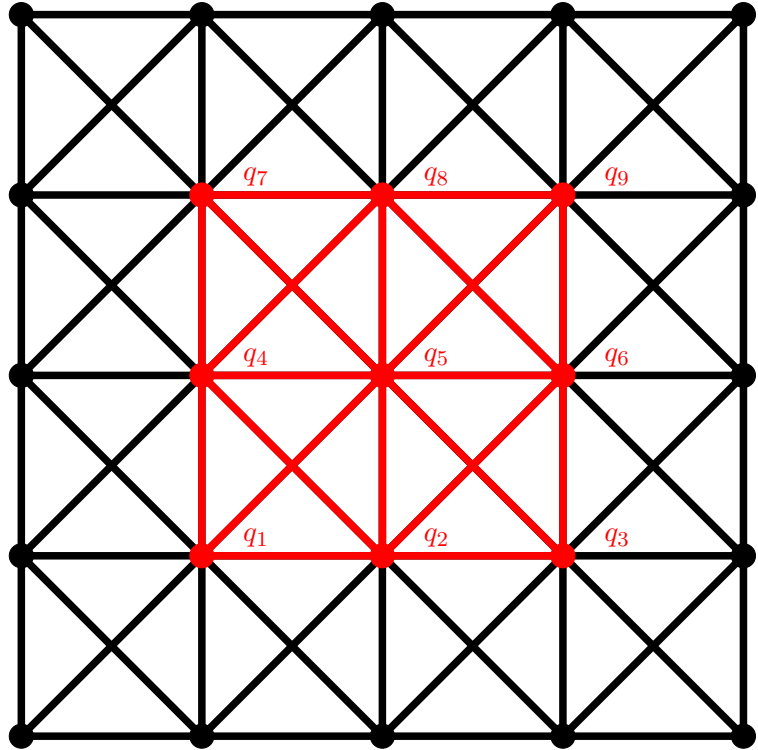

\begin{theorem} \label{thm:nnnExGapLength}
Enumerate the vertices of a $3\times 3$ fundamental cell of the square lattice as in Figure~\ref{fig:sqn33period}, denote $r=\sqrt{4-\sqrt{15}}$, define a $(3,3)$-periodic potential $Q$ on $\Z^2$ via
\[
(q_1,\ldots,q_9)
=
\Big(-r-\frac{1}{r}+2,\ -r,\ -r+\frac{1}{r}-2,\ -\frac{1}{r},\ 0,\ +\frac{1}{r},\ r-\frac{1}{r}-2,\ r,\ r+\frac{1}{r}+ 2\Big),
\]
and denote $H_\lambda = \Delta_\sqn + \lambda Q$. Then, for all $ \lambda > 0$ sufficiently small, $\sigma(H_\lambda)$ consists of two connected components. Moreover, if $\mathfrak{g}_\lambda$ denotes the gap that opens at energy $-1$, one has
\[
\left(-1 - \frac{\lambda}{10}, -1 + \frac{\lambda}{10} \right)
\subseteq
\mathfrak{g}_\lambda
\subseteq
\left(-1 - \frac{\lambda}{4},-1+\frac{\lambda}{4}\right).
\]
In particular, the gap opens linearly.
\end{theorem}

Let us observe that the proof below can be refined a bit to yield sharper constants than $1/10$ and $1/4$.

\begin{proof}
For $\bthe = (\theta_1,\theta_2) \in \T^2$, let $H_\lambda(\bthe)$ denote the Floquet matrix corresponding to $H_\lambda$. Ordering the vertices of the fundamental domain as in Figure~\ref{fig:sqn33period}, we obtain:
\[
H_\lambda(\bthe)
=
\begin{bmatrix}
\lambda q_1 & 1 & e^{-i\theta_1} & 1 & 1 & e^{-i\theta_1} & e^{-i \theta_2} & e^{-i\theta_2} &e^{-i(\theta_1+\theta_2)} \\
1 & \lambda q_2 & 1 & 1 & 1 & 1 & e^{-i \theta_2}& e^{-i\theta_2} & e^{-i\theta_2} \\
e^{i \theta_1} & 1 & \lambda q_3& e^{i\theta_1} & 1 & 1 & e^{i(\theta_1-\theta_2)} & e^{-i\theta_2}  & e^{-i\theta_2}\\
1 & 1 & e^{-i\theta_1} & \lambda q_4 & 1 & e^{-i\theta_1} & 1 & 1 &e^{-i\theta_1} \\
1 & 1 & 1 & 1 & \lambda q_5 & 1 & 1 & 1 &1 \\
e^{i\theta_1} & 1 & 1 & e^{i\theta_1} & 1 &  \lambda q_6 & e^{i\theta_1} & 1 & 1  \\
e^{i\theta_2} & e^{i\theta_2} & e^{-i(\theta_1-\theta_2)} & 1 & 1 &e^{-i\theta_1} & \lambda q_7& 1 & e^{-i\theta_1} \\
e^{i\theta_2} & e^{i\theta_2} & e^{i\theta_2} & 1 & 1 & 1 & 1 & \lambda q_8 & 1 \\
e^{i(\theta_1+\theta_2)} & e^{i\theta_2} & e^{i\theta_2} & e^{i\theta_1} &1 &1 & e^{i\theta_1} & 1 &\lambda q_9  
\end{bmatrix}.
\]
For $s\in (-1,1)$, let us consider 
\[
\det(H_{\lambda}(\bthe) + (1+ s\lambda) \bbI)
=
\sum_{k=0}^9 X_k(\bthe, s) \lambda^k.
\]
Our goal is to show $\det(H_{\lambda}(\bthe) + (1+ s\lambda) \bbI)$ never vanishes for sufficiently small $\lambda>0$ and for $|s| < 0.1$.
Direct computations yield
\begin{align*}
X_0(\bthe, s)&=4096 \sin^6\left(\frac{\theta_1}{2}\right) \sin^6\left(\frac{\theta_2}{2}\right)\\
X_1(\bthe, s)&=0\\
X_2(\bthe, s)&=Y_2(s) \sin^4\left(\frac{\theta_1}{2}\right) \sin^4\left(\frac{\theta_2}{2}\right)\\
X_3(\bthe, s)&=Y_3(s) \sin^4\left(\frac{\theta_1}{2}\right) \sin^4\left(\frac{\theta_2}{2}\right)\\
X_4(\bthe, s)&=Y_4(s) \sin^2\left(\frac{\theta_1}{2}\right) \sin^2\left(\frac{\theta_2}{2}\right)\\
X_5(\bthe, s)&=Y_5(s) \sin^2\left(\frac{\theta_1}{2}\right) \sin^2\left(\frac{\theta_2}{2}\right)\\
X_6(\bthe, s)&=Y_{6,1}(s)+Y_{6,2}(s)\cos(\theta_1)+Y_{6,3}(s)\cos(\theta_2)\\ & \qquad +Y_{6,4}(s)\cos(\theta_1)\cos(\theta_2)+Y_{6,5}(s) \sin(\theta_1)\sin(\theta_2)\\
X_7(\bthe, s)&=0\\
X_8(\bthe, s)&=Y_8(s)\\
X_9(\bthe, s)&=Y_9(s),
\end{align*}
in which
\begin{align*}
Y_2(s)&=512 (20-9 s^2)\\
Y_3(s)&=256 (4 - 20 s + 3 s^3)\\
Y_4(s)&=16(364 + 144 s - 504 s^2 + 81 s^4)\\
Y_5(s)&=16 (64 - 196 s - 48 s^2 + 104 s^3 - 9 s^5)\\
Y_{6,1}(s)&=176 + 704 s - 3132 s^2 - 496 s^3 + 1376 s^4 - 96 s^6\\
Y_{6,2}(s)&=-80 + (96 \sqrt{15}-320) s + (1380 +144 \sqrt{15}) s^2 + 208 s^3 - (584+54 \sqrt{15}) s^4 + 42 s^6\\
Y_{6,3}(s)&=-80 - (320 +96 \sqrt{15}) s + (1380 - 144 \sqrt{15}) s^2 + 208 s^3 - (584- 54 \sqrt{15}) s^4 + 42 s^6\\
Y_{6,4}(s)&=-16 - 64 s + 372 s^2 + 80 s^3 - 208 s^4 + 12 s^6\\
Y_{6,5}(s)&=8(2s-1)^3\\
Y_8(s)&=12 + 32 s - 360 s^2 - 512 s^3 + 1025 s^4 + 96 s^5 - 224 s^6 + 9 s^8\\
Y_9(s)&=12 s + 16 s^2 - 120 s^3 - 128 s^4 + 205 s^5 + 16 s^6 - 32 s^7 + s^9.\\
\end{align*}
One simple observation is that 
\begin{align}\label{eq:sumY6}
Y_{6,1}(s)+Y_{6,2}(s)+Y_{6,3}(s)+Y_{6,4}(s)=0.
\end{align}
It is easy to see that for $|s|<0.1$, 
\[Y_2(s), Y_3(s),Y_5(s)>0.\]
It is easy to compute that
\[
Y'_9(s)=12 + 32 s - 360 s^2 - 512 s^3 + 1025 s^4 + 96 s^5 - 224 s^6 + 9 s^8
=
Y_8(s).
\]
Thus, 
\begin{equation} \label{eq:Y9prime(s)}
Y_9'(s)>12 - 32 \times 0.1 - 360 \times (0.1)^2  - 512 \times (0.1)^3 - 96 \times (0.1)^5 - 224\times (0.1)^6{>4.5}>0\end{equation}
for $|s|<0.1$, which implies
\begin{align}\label{eq:Y9s}
Y_9(s)\geq Y_9(-0.1)>-1
\end{align}
for all $|s| < 0.1$. Carefully estimating $Y_4(s)$ and $Y_8(s)$ will help us bound the $\lambda^6$ order term from below using the AM-GM inequality.
\begin{align}\label{eq:Y4Y8s}
Y_4(s)&\geq 16(364-144\times 0.1-504\times (0.1)^2 {-81 \times (0.1)^4})>5500,\\
Y_8(s)&\geq 12  - 32 \times 0.1 - 360 \times (0.1)^2 - 512 \times (0.1)^3 - 96 \times (0.1)^5 - 224 \times (0.1)^6>4.5. \notag
\end{align}
{In fact, since $Y_8 = Y_9'$, the second inequality already follows from \eqref{eq:Y9prime(s)}.}
For the $Y_{6,j}$ terms, we have
\begin{equation}\label{eq:Y6s}
\begin{aligned}
Y_{6,1}(s)&\geq 176 - 704 \times 0.1 - 3132 \times (0.1)^2 - 496 \times (0.1)^3 - 96 \times (0.1)^6>0,\\
  Y_{6,2}(s)&\leq -80 + (96 \sqrt{15}-320) \times 0.1 + (1380 +144 \sqrt{15}) \times (0.1)^2 \\   & \qquad\qquad + 208 \times (0.1)^3  + 42\times (0.1)^6<0 \\
  Y_{6,3}(s)&\leq -80 + (320 + 96 \sqrt{15}) \times 0.1 + (1380 - 144 \sqrt{15}) \times (0.1)^2 \\   & \qquad\qquad+ 208 \times (0.1)^3  + 42 \times (0.1)^6<0, \\
Y_{6,4}(s)&\leq -16 + 64 \times 0.1 + 372 \times (0.1)^2 + 80 \times (0.1)^3  + 12 \times (0.1)^6<0, \\
-14 & \leq Y_{6,5}(s)< 0.
\end{aligned}
\end{equation}
Using \eqref{eq:sumY6} and \eqref{eq:Y6s}, we obtain
\begin{align} \label{eq:X6LB}
\notag
X_6(\bthe) & \geq Y_{6,1}(s)+Y_{6,2}(s)+Y_{6,3}(s)+Y_{6,4}(s) + Y_{6,5}(s)\sin(\theta_1)\sin(\theta_2) \\
\notag
& =
Y_{6,5}(s)\sin(\theta_1)\sin(\theta_2) \\
& \geq
-14 |\sin(\theta_1)\sin(\theta_2)|.
\end{align}
In particular, the first line uses $Y_{6,2}, Y_{6,3}, Y_{6,4} < 0$, the second line uses \eqref{eq:sumY6}, and the final line uses $-14 \leq Y_{6,5} < 0$.

Now we combine our estimates together. 
Note that 
\begin{align}\label{eq:sumX0235}
X_0(\bthe,s)+X_2(\bthe,s)\lambda^2+X_3(\bthe,s)\lambda^3+X_5(\bthe,s)\lambda^5\geq 0.
\end{align}
Using $a^2+b^2\geq 2|ab|$, we obtain the following from \eqref{eq:Y4Y8s}
\[
X_4(\bthe,s)\lambda^4+\frac{1}{2}X_8(\bthe,s)\lambda^8 \geq 2\sqrt{2.25\times 5500} \left|\sin\left(\frac{\theta_1}{2}\right) \sin\left(\frac{\theta_2}{2}\right)\right| \lambda^6.
\]
Using $2|\sin(x/2)|\geq 2|\sin(x/2)\cos(x/2)|=|\sin(x)|$, we obtain from above that
\[X_4(\bthe,s) {\lambda^4} + \frac{1}{2}X_8(\bthe,s) {\lambda^8} 
\geq 
55 |\sin(\theta_1)\sin(\theta_2)|{\lambda^6}.\]
Combining this with \eqref{eq:X6LB}, we have
\begin{equation} \label{eq:sumX486}
X_4(\bthe,s)\lambda^4+\frac{1}{2}X_8(\bthe,s)\lambda^8+X_6(\bthe,s)\lambda^6 
\geq
41 |\sin(\theta_1)\sin(\theta_2)|\lambda^6
\geq 0.
\end{equation}

Finally using \eqref{eq:Y9s} and \eqref{eq:Y4Y8s}, we have
\begin{align}\label{eq:sumY89}
\frac{1}{2}X_8(\bthe,s)\lambda^8+X_9(\bthe,s)\lambda^9=\frac{1}{2}Y_8(s)\lambda^8+Y_9(s)\lambda^9\geq 2.25\lambda^8-\lambda^9>0.25 \lambda^8,
\end{align}
provided that $\lambda<2$.
Combining \eqref{eq:sumX0235}-\eqref{eq:sumY89}, we have
\[\det(H_{\lambda}(\bthe) + (1+ s\lambda) \bbI) \geq 0.25\lambda^8>0,\]
for any $\bthe\in \T^2$ and $|s|<0.1$. 
This proves the lower bound on the gap.

{For the upper bound, observe that $X_j\big((\pi,0),s\big) = 0$ for all $s$ and for every $0 \le j \le 5$ and
\[
X_6\big((\pi,0),\pm 1/4\big)
<-85.
\]
Thus, for small $\lambda > 0$, 
\[
\det\big(H_\lambda(\pi, 0) + (1\pm\lambda/4)\bbI \big) < -85\lambda^6 + O(\lambda^8)<0.
\]
It is also clear that $X_0((0,0),s)=4096$,
which implies 
\[\det \big(H_\lambda(0, 0) +(1\pm\lambda/4)\bbI \big) =4096+O(\lambda)>0.\]
Thus we conclude that 
\[1\pm \frac{\lambda}{4}\in \sigma(H_\lambda),\]
which concludes the proof of the upper bound on the length of the gap.}
\end{proof}

\section*{Acknowledgement}
We would like to thank Svetlana Jitomirskaya for comments on an earlier version of the manuscript, and Tom Spencer for useful discussions.
R.H. would like to thank IAS, Princeton, for its hospitality during the 2017-18 academic year,
and Virginia Tech for its hospitality during which part of the work was done.
R.H. is supported in part by the National Science Foundation under Grant No. DMS-1638352. 
J.F.\ was supported in part by an AMS-Simons Travel Grant 2016--2018.

\end{document}